\documentclass[12pt,a4paper]{article}
\pdfoutput=1
\usepackage[utf8]{inputenc}  
\usepackage[T1]{fontenc}  
\usepackage[english]{babel}
\usepackage{tikz}
\usetikzlibrary{positioning,chains,fit,shapes,calc}
\usetikzlibrary{patterns}

\usepackage{geometry}
\usepackage{amsmath}
\usepackage{amsfonts}
\usepackage{amssymb}
\usepackage{amsthm}
\usepackage{yhmath}
\usepackage{stmaryrd}
\usepackage{dsfont}
\usepackage{bbm}
\usepackage{mathtools}

\newtheorem{thm}{Theorem}[section]
\newtheorem{prop}[thm]{Proposition}
\newtheorem{lem}[thm]{Lemma}
\newtheorem*{rk}{Remark}

\newtheorem*{defi}{Definition}
\newtheorem*{thm*}{Theorem}
\newtheorem*{prop*}{Proposition}
\newtheorem*{lem*}{Lemma}
\usepackage{graphicx}
\usepackage{float}
\restylefloat{figure}

\newcommand{\E}{\mathbb{E}}
\renewcommand{\P}{\mathbb{P}}

\newcommand{\R}{\mathbb{R}}
\newcommand{\N}{\mathbb{N}}
\newcommand{\Z}{\mathbb{Z}}
\newcommand{\C}{\mathbb{C}}
\newcommand{\D}{\mathbb{D}}

\newcommand{\bL}{\mathbb{L}}
\newcommand{\cP}{\mathcal{P}}
\newcommand{\cK}{\mathcal{K}}

\newcommand{\cT}{\mathcal{T}}
\newcommand{\cL}{\mathcal{L}}

\newcommand{\cM}{\mathcal{M}}
\newcommand{\cC}{\mathcal{C}}
\newcommand{\cN}{\mathcal{N}}
\newcommand{\cV}{\mathcal{V}}
\newcommand{\cA}{\mathcal{A}}
\newcommand{\cEG}{\mathcal{EG}}
\newcommand{\kU}{\mathfrak{U}}
\newcommand{\kM}{\mathfrak{M}}
\newcommand{\kL}{\mathfrak{L}}

\newcommand{\oD}{\overline{\D}}

\usepackage[colorlinks=false]{hyperref}
\tikzset{cross/.style={cross out, draw=black, minimum size=2*(#1-\pgflinewidth), inner sep=0pt, outer sep=0pt},
cross/.default={1pt}}
\usepackage{animate}
\usepackage{todonotes}

\author{Paul Thevenin\footnote{CMAP \& \'Ecole polytechnique, paul.thevenin@polytechnique.edu \newline The author acknowledges partial support from Agence Nationale de la Recherche,
Grant Number ANR-14-CE25-0014 (ANR GRAAL).}}
\title{A geometric representation of fragmentation processes\\ on stable trees}
\date{\today}
\begin{document}

\maketitle

\begin{abstract}
\small{We provide a new geometric representation of a family of fragmentation processes by nested laminations, which are compact subsets of the unit disk made of noncrossing chords. We specifically consider a fragmentation obtained by cutting a random stable tree at random points, which split the tree into smaller subtrees. When coding each of these cutpoints by a chord in the unit disk, we separate the disk into smaller connected components, corresponding to the smaller subtrees of the initial tree. This geometric point of view allows us in particular to highlight a new relation between the Aldous-Pitman fragmentation of the Brownian continuum random tree and minimal factorizations of the $n$-cycle, i.e. factorizations of the permutation $(1 \, 2 \, \cdots \, n)$ into a product of $(n-1)$ transpositions. We discuss various properties of these new lamination-valued processes, and we notably show that they can be coded by explicit Lévy processes.}
\end{abstract}

\begin{figure}
\center
\caption{The image represents an approximation of the lamination $\bL^{(1.8)}_{100}$. By using Adobe Acrobat and by clicking on the ``play'' button, one can view an approximation of the process $\left(\bL^{(1.8)}_c\right)_{c \geq 0}$.}
\label{fig:animate}
\animategraphics[width=.5\textwidth,controls]{5}{film/moiii}{0}{50}
\end{figure}

\section{Introduction}
\label{sec:intro}

The purpose of this work is to investigate a  geometric and dynamical representation of fragmentation processes derived from random stable trees in terms of laminations, with an application to permutation factorizations.  Specifically, we shall code the analogue of the Aldous-Pitman fragmentation on a stable tree by a new lamination-valued càdlàg process. Also, in the Brownian case, we shall establish a connection between this lamination-valued process and minimal factorizations of a cycle into transpositions. Before stating our results, let us first present the main objects of interest.

\subsection{Fragmentations and laminations}
\label{ssec:fraglam}

\paragraph*{Fragmentation processes derived from stable trees.}

Fragmentation processes describe the evolution of an object with given mass, which splits into smaller pieces as time passes.  Specifically, a \textit{fragmentation process} $\Lambda = \left( \Lambda(t), t \geq 0 \right)$ is a càdlàg process (that is, left-continuous with right limits) on the set

\begin{align*}
\Delta \coloneqq \left\{ x_1 \geq x_2 \geq \cdots \geq 0, \sum_{i \geq 1} x_i = 1 \right\}.
\end{align*}
such that, if one denotes by $P_s$ the law of $\Lambda$ starting from $s \coloneqq (s_1, s_2, ...)$, then $P_s$ is the nonincreasing reordering of the elements of independent processes of laws $P_{(s_1, 0, 0, ...)}, P_{(s_2, 0, 0, ...)}, ...$. This means that each fragment breaks independently of the others, in a way that only depends on its mass.

The starting point of this paper is a well-known fragmentation process which was introduced by Aldous and Pitman \cite{AP98} and which consists in cutting a specific random tree - namely, Aldous' Brownian tree - at random points. These \textit{cutpoints} are spread out on the tree following a homogeneous Poisson distribution of density $c \, d\ell$, where $c>0$ and $\ell$ is the length measure on the tree. The Brownian tree (sometimes called CRT, for continuum random tree) is therefore split into smaller components as $c$ increases. This process has been studied in depth, notably by Bertoin \cite{Ber00} who gives a different surprising construction from a linearly drifted standard Brownian excursion over its current infimum, as the slope of the drift varies. Miermont \cite{Mie05} has considered more generally fragmentations obtained by cutting at random the so-called \textit{stable trees}. These random trees $\cT^{(\alpha)}$ (for $\alpha \in (1,2]$), introduced by Duquesne and Le Gall \cite{DLG02} (see also \cite{LGLJ98}), can be coded by $\alpha$-stable spectrally positive Lévy processes and arise as scaling limits of size-conditioned Galton-Watson trees. They generalize Aldous' Brownian tree, which can be seen as the $2$-stable tree. Miermont investigates a way of cutting these stable trees only at \textit{branching points} --- that is, points whose removal splits the tree into three or more different subtrees --, while Abraham \& Serlet  \cite{AS02} cut them uniformly on their \textit{skeleton} (made of points which are not branching points). This gives birth to two different fragmentation processes. Let us also mention Voisin \cite{Voi11} who studies a mixture of these two processes. Fragmentations can also more generally be derived from Lévy trees (see \cite{AD08, ADV10, AS02}), which are trees coded by Lévy processes.

Let us briefly mention that a fragmentation process can be seen as a time-reversed coalescent process, where particles with given masses merge at a rate that depends on their respective masses.  The so-called \textit{standard additive coalescent} is the coalescent process where only two particles merge at each time, at a rate that is the sum of their masses. This standard coalescent is the time-reversed analogue of the previously mentioned Aldous-Pitman fragmentation process on the Brownian tree \cite{AP98}. Several other models of coalescent processes have been investigated, such as Kingman's coalescent \cite{Kin82} where two particles merge at rate $1$, or Aldous' multiplicative coalescent \cite{Ald97, Ald98} where particles merge proportionally to the product of their masses. See also the book of Bertoin \cite{Ber06} for fully detailed information about coalescent processes. Let us finally mention Chassaing and Louchard \cite{CL02} who provide a representation of the standard additive coalescent as \textit{parking schemes} (see also \cite{MW19}).

In this paper, we consider the previously mentioned analogue of the Aldous-Pitman fragmentation on a stable tree. Specifically, we  fix $\alpha \in (1,2]$ and focus on cutting the $\alpha$-stable tree $\cT^{(\alpha)}$ homogeneously on its skeleton by a homogeneous Poisson process $\cP_c(\cT^{(\alpha)})$ of intensity $c \, d\ell$, where $c>0$ and $\ell$ is the length measure on the tree, consistently as $c$ increases (we refer to Section \ref{sec:construction} for precise definitions, and \cite{ALW17} for a rigorous definition of $\ell$). Cutting $\cT^{(\alpha)}$ at the points of $\cP_c(\cT^{(\alpha)})$ then splits the tree into a random set of smaller components, whose decreasingly reordered sequence $\textbf{m}_c^{(\alpha)}$ of \textit{masses} (i.e., the proportion of leaves of the tree in these components, see again Section~\ref{sec:construction} for precise definitions) is an element of $\Delta$. This defines the $\alpha$-fragmentation process

\begin{align*}
F^{(\alpha)} \coloneqq \left(F_c^{(\alpha)} \right)_{c \geq 0} = \left( \textbf{m}_c^{(\alpha)} \right)_{c \geq 0}.
\end{align*}
In the case $\alpha=2$, this is the Aldous-Pitman fragmentation of $\cT^{(2)}$.

\paragraph{Laminations and excursion-type functions.} The aim of this paper is to code the analogue of the Aldous-Pitman fragmentation on a stable tree by a nondecreasing lamination-valued process, where, roughly speaking, a chord in the lamination corresponds to a cutpoint on the tree.  By definition, a lamination is a closed subset of the closed unit disk $\overline{\mathbb{D}}$ which can be written as the union of the unit circle $\mathbb{S}^1$ and a set of chords which do not intersect in the open unit disk $\mathbb{D}$. Laminations are important objects in topology and in hyperbolic geometry, see for instance \cite{Bon01} and references therein. If $L$ is a lamination, a \textit{face} of $L$ is a connected component of the complement of $L$ in $\overline{\mathbb{D}}$.

The connection between random trees and random laminations goes back to Aldous \cite{ald} who used the Brownian excursion to code the so-called Brownian triangulation (see Fig.  \ref{fig:arbconlam}, right, for a simulation). The Brownian triangulation is a random lamination whose faces are all triangles, and its ``dual'' tree is, in some sense, the Brownian CRT. Since then, this object has appeared as the limit of several discrete structures \cite{CK14, KM17, Bet17}, and in the theory of random planar maps \cite{LGP08}.

Other models of random laminations have been recently studied. The Brownian triangulation has been generalized by Kortchemski \cite{Kor14}, who introduced, for $\alpha \in (1,2]$ the so-called \textit{$\alpha$-stable lamination}, whose ``dual'' tree is in a certain sense the $\alpha$-stable tree, and which appears as the limit of certain models of random dissections (which are collections of noncrossing diagonals of a regular polygon). In a different direction, Curien and Le Gall \cite{CLG11} consider laminations built by recursively adding chords. Another family of random laminations connected to random minimal factorizations of a cycle into transpositions, which will be one of the objects of interest in this paper, has been introduced in \cite{FK17}. While all these random laminations can be coded by random excursion-type functions, other laminations such as the hyperbolic triangulation \cite{CW13} or triangulated stable laminations \cite{KM16} cannot.

Let us immediately explain how to construct laminations from so-called excursion-type functions. Let $f: [0,1] \rightarrow \R$. We say that $f$ is an \emph{excursion-type function} if the following conditions are verified: 

\begin{enumerate}
\item[(i)] $f$ is càdlàg (that is, right-continuous on $[0,1)$, with left limits on $(0,1]$);
\item[(ii)] $f$ is nonnegative on $[0,1]$ and $f(1)=0$;
\item[(iii)] $f$ only makes positive jumps, that is, for all $x \in (0,1]$, $f(x-) \leq f(x)$.
\end{enumerate}

Following the construction of \cite{Kor14}, to an excursion-type function $f$, one can associate a lamination $\bL(f)$ as follows.  For any $0 \leq s<t \leq 1$, say that $s \sim_f t$ if $t \coloneqq \inf \{ u>s,f(u) \leq f(s-) \}$ (where we set $f(0-)=0$). For $t>s$, we say that $t \sim_f s$ if $s \sim_f t$, and we say that for any $s \in [0,1]$, $s \sim_f s$. This way, $\sim_f$ is an equivalence relation on $[0,1]$. The lamination $\bL(f)$ is defined as the closure
\begin{align*}
\bL(f) \, = \, \overline{\mathbb{S}^1 \cup \underset{\substack{s,t \in (0,1)\\ s \sim_f t}}{\bigcup} \left[e^{-2i\pi s}, e^{-2i\pi t}\right]}
\end{align*}
where $[y,z]$ denotes the line segment joining the two complex numbers $y$ and $z$. 

The $\alpha$-stable lamination, which plays an important role in our work, can be constructed from a planar version of the $\alpha$-stable tree (we refer to Section \ref{ssec:construction} for precise definitions). Indeed, we view  $\cT^{(\alpha)}$  as coded by a continuous normalized $\alpha$-stable height process  $( H^{(\alpha)}_t)_{t \in [0,1]}$ (so that, informally, $H^{(\alpha)}$ is the contour function of $\cT^{(\alpha)}$). We define the $\alpha$-stable lamination $\bL_\infty^{(\alpha)}$ as
\begin{equation}
\label{eq:bLa}\bL_\infty^{(\alpha)} \quad \coloneqq  \quad \bL\left( H^{(\alpha)} \right).
\end{equation}
It is possible to check (see \cite{Kor14}) that faces of $\bL_\infty^{(\alpha)}$ are in correspondence with branching points of $\cT^{(\alpha)}$, and that there are chords which are not adjacent to any face (one can find chords arbitrarily close to such a chord, from both sides) which are in correspondence with the points of $\cT^{(\alpha)}$ that are not leaves nor branching points. See Fig. \ref{fig:stablearbconlam} for an approximation of these items, for $\alpha=1.5$.

\begin{figure}
\center
\caption{An approximation of $\left( \cT^{(1.5)}, H^{(1.5)}, \bL_\infty^{(1.5)} \right)$.}
\label{fig:stablearbconlam}
\includegraphics[scale=.6]{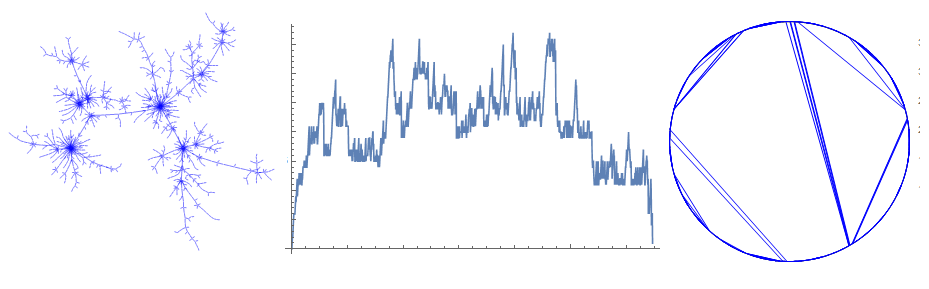}
\end{figure}

We conclude this section with a last definition concerning laminations.  We define the \textit{mass} of a face $F$ of a lamination $L$ as $\frac{1}{2\pi}$ times the Lebesgue measure of $\partial F \cap \mathbb{S}^1$ (roughly speaking, it corresponds to the part of the perimeter of $F$ that lies on the unit circle). Finally, the \emph{mass sequence} of $L$, denoted by $\cM[L]$, is the sequence of the masses of its faces, sorted in nonincreasing order.

\subsection{The  lamination-valued process \texorpdfstring{$( \bL_c^{(\alpha)})_{c \in [0,+\infty]}$}{Lcalpha}}

For a fixed $\alpha \in (1,2]$, we now introduce a new  lamination-valued process $( \bL_c^{(\alpha)})_{c \in [0,+\infty]}$ which encodes, in a certain sense, the fragmentation  $F^{(\alpha)}$ of the  $\alpha$-stable tree. Here we give a rather informal definition, and defer to Section \ref{sec:construction} precise definitions.

\paragraph{Definition of $( \bL_c^{(\alpha)})_{c \in [0,+\infty]}$.} As above, we view $\cT^{(\alpha)}$  as coded by a normalized $\alpha$-stable height process  $( H^{(\alpha)}_t)_{t \in [0,1]}$. We consider a homogeneous Poisson process $\cP_c(\cT^{(\alpha)})$ of intensity $c \, d\ell$ on the skeleton of  $\cT^{(\alpha)}$, where $c \geq 0$ and $\ell$ is the length measure on the tree, consistently as $c$ increases. For $c \geq 0$, we define the lamination $\bL_c^{(\alpha)}$ as the subset of the $\alpha$-stable lamination $\bL_\infty^{(\alpha)}$, obtained by keeping only the chords which correspond to the vertices of $\cP_c(\cT^{(\alpha)})$  (recall that to points of the skeleton of $\cT^{(\alpha)}$  correspond   chords of $\bL_\infty^{(\alpha)}$). Intuitively speaking, one obtains the process $( \bL_c^{(\alpha)})_{c \in [0,+\infty]}$  by revealing the chords of  $\bL_\infty^{(\alpha)}$ in a Poissonian way (see Fig.~\ref{fig:animate} for an approximation of $( \bL_c^{(1.8)} )_{c \in [0,+\infty]}$).

\paragraph{Connection with fragmentations.}

The process $c \mapsto  \bL_c^{(\alpha)}$, which is an increasing lamination-valued process,  is the main object of interest in this paper. It encodes the Aldous-Pitman fragmentation   of the  $\alpha$-stable tree in the following sense (where we recall that $\cM[L]$ is the mass sequence of a lamination $L$).

\begin{figure}

\caption{An approximation of $\left( \cT^{(2)}, H^{(2)}, \bL_\infty^{(2)} \right)$.}

\label{fig:arbconlam}

\includegraphics[scale=0.4]{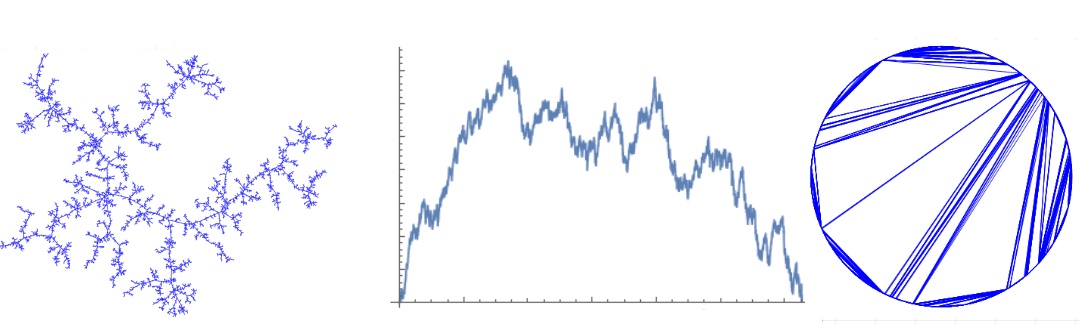}

\end{figure}

\begin{thm}

\label{thm:existence}

The following equality holds in distribution in $\Delta$:

\begin{align*}
\left( \cM \left[ \bL_c^{(\alpha)} \right] \right)_{c \geq 0}  \quad \mathop{=}^{(d)} \quad  F^{(\alpha)}.
\end{align*}

\end{thm}

In a certain sense, $( \bL_c^{(\alpha)})_{c \in [0,+\infty]}$ can be viewed as a ``dual planar representation'' of  the Aldous-Pitman fragmentation of the  $\alpha$-stable tree, and as a ``linearization'' of the associated time-reversed coalescent process. In order to prove Theorem \ref{thm:existence}, we view the Poissonian cuts on the skeleton of $\cT^{(\alpha)}$ as a non-homogenous Poisson process in the epigraph of $ H^{(\alpha)}$ (see Section \ref{ssec:construction}).

Let us mention that for fixed $c>0$ and $\alpha=2$,  the lamination $\bL_c^{(2)}$ appears in \cite{FK17} in the context of random minimal factorizations of a cycle, without any connection to fragmentations. In addition, defining a  coupling $\bL_c^{(2)}$ as $c$ increases  and obtaining a functional convergence was left open in \cite{FK17}. Also, Shi \cite{Shi15} used fragmentation theory to study large faces in the Brownian triangulation and in stable laminations, by using the so-called fragmentation by heights of stable trees (which is different from the one that appears here, see \cite{Mie03}).

Also, throughout the paper, the lamination-valued processes will be defined on $[0,+\infty]$, while the associated fragmentation processes $F^{(\alpha)}$ are only defined on $\R_+$. Observe indeed that, almost surely, $\sup F_c^{(\alpha)} \rightarrow 0$ as $c \rightarrow +\infty$, which corresponds to extinction at $+\infty$. On the other hand, the increasing process $(\bL_c^{(\alpha)})_{c \geq 0}$ has a non-trivial limit at $+\infty$.

\subsection{Connections with random minimal factorizations}

One of the main contributions of this paper is to show that the process $( \bL_c^{(2)} )_{c \in [0,+\infty]}$ appears as the functional limit of a natural coding of  so-called minimal factorizations of the $n$-cycle. More precisely, for $n \in \Z_+$, denote by $\mathfrak{S}_n$ the group of permutations acting on $\llbracket 1, n \rrbracket$ and by $\mathfrak{T}_n$ the set of transpositions of $\mathfrak{S}_n$. Then, the elements of the set

\begin{align*}
\mathfrak{M}_n \coloneqq \left\{ \left( t_1,  \ldots , t_{n-1} \right) \in \mathfrak{T}_n^{n-1}, \, t_1 \cdots t_{n-1} = (1 \, 2 \, \cdots \, n) \right\}
\end{align*}
are called minimal factorizations of the $n$-cycle into transpositions, or just \textit{minimal factorizations} in short. Their study goes back to Dénès \cite{Den59} and Moszkowski \cite{Mos89}. By convention, we read transpositions from left to right, so that $t_1 t_2$ corresponds to $t_2 \circ t_1$.

Goulden and Yong \cite{GY02} view minimal factorizations in a geometric way, noticing that it is possible to represent each of them by a non-crossing tree in the unit disk. More specifically, if $ ( t_1,  \ldots , t_{n-1} ) \in \mathfrak{M}_n$ and $t_{j}= (a_{j},b_{j})$ for $1 \leq j \leq n-1$, then

$$\bigcup_{j=1}^{n-1}  \left[ e^{-2i\pi a_{j}/n}, e^{-2i\pi b_{j}/n} \right]$$
is a non-crossing tree and, in particular, a lamination (adding $\mathbb{S}^1$, see Fig. \ref{fig:minfaclam}). In this direction, for a uniform minimal factorization  $t^{(n)}$ of the $n$-cycle, Féray and Kortchemski \cite{FK17} have shown that a phase transition occurs  when roughly $\sqrt{n}$ transpositions have been read. More precisely, for $c>0$, if  $\cL^{(n)}_{c}$ is the lamination obtained by drawing the chords corresponding to the first $\lfloor c \sqrt{n} \rfloor$ transpositions of  $t^{(n)}$, then \cite[Theorem $3$, (i)]{FK17} shows that for $c>0$, $\cL_{c}^{(n)}$ converges in distribution for the Hausdorff distance to a limiting random lamination, defined by using a certain Lévy process (and not fragmentations nor Poisson processes). 

One of the main results of this paper is to show that this convergence actually holds in the functional sense  (that is, jointly in $c \in [0,\infty]$) and that the limiting process is $( \bL_c^{(2)} )_{c \in [0,+\infty]}$. As a corollary, we obtain an alternative and, in our opinion, simpler proof of the one-dimensional convergence \cite[Theorem $3$, (i)]{FK17}.

Let us quickly give some background concerning this notion of convergence. The set $\bL(\oD)$  of laminations of the closed unit disk is endowed with the Hausdorff distance $d_H$ between compact subsets of $\oD$, so that $(\bL(\oD),d_{H})$ is a Polish metric space (that is, separable and complete). The Hausdorff distance is defined as follows. If $K$ is a compact subset of $\oD$ and $\epsilon>0$, define the $\epsilon$-neighbourhood of $K$ as $K^\epsilon \coloneqq \left\{ x \in \oD, d(x,K) < \epsilon \right\}$, where $d$ denotes the usual Euclidean distance on $\R^2$. Then, for $K_1,K_2$ compact subsets of the unit disk, we define

\begin{align*}
d_H(K_1,K_2) \coloneqq \inf \left\{ \epsilon>0, K_2 \subset K_1^\epsilon \text{ and } K_1 \subset K_2^\epsilon  \right\}.
\end{align*}

\begin{figure}
\center
\caption{The lamination associated to the minimal factorization $F \coloneqq (34)(89)(35)(13)(16)(18)(23)(78) \in \kM_9$.}
\label{fig:minfaclam}
\begin{tikzpicture}[scale=0.7, every node/.style={scale=0.7}, rotate=-40]

\draw (0,0) circle (3);

\foreach \i in {1,...,9}
{
\draw[auto=right] ({3.3*cos(-(\i-1)*360/9)},{3.3*sin(-(\i-1)*360/9)}) node{\i};
}

\draw ({3*cos(-2*360/9)},{3*sin(-2*360/9)}) -- ({3*cos(-3*360/9)},{3*sin(-3*360/9)});

\draw ({3*cos(2*360/9)},{3*sin(2*360/9)}) -- ({3*cos(1*360/9)},{3*sin(1*360/9)});

\draw ({3*cos(-2*360/9)},{3*sin(-2*360/9)}) -- ({3*cos(-4*360/9)},{3*sin(-4*360/9)});

\draw ({3*cos(-2*360/9)},{3*sin(-2*360/9)}) -- ({3*cos(0*360/9)},{3*sin(0*360/9)});

\draw ({3*cos(-5*360/9)},{3*sin(-5*360/9)}) -- ({3*cos(0*360/9)},{3*sin(0*360/9)});

\draw ({3*cos(-7*360/9)},{3*sin(-7*360/9)}) -- ({3*cos(0*360/9)},{3*sin(0*360/9)});

\draw ({3*cos(-2*360/9)},{3*sin(-2*360/9)}) -- ({3*cos(-1*360/9)},{3*sin(-1*360/9)});

\draw ({3*cos(-6*360/9)},{3*sin(-6*360/9)}) -- ({3*cos(-7*360/9)},{3*sin(-7*360/9)});

\end{tikzpicture}

\end{figure}
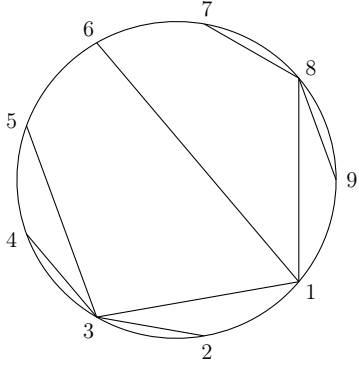

In the rest of the paper, for $E,F$ two metric spaces, $\mathbb{D}\left( E, F \right)$ denotes the set of càdlàg processes from $E$ to $F$, endowed with the Skorokhod $J_1$ topology (see Annex $A2$ in \cite{Kal02} for background). Finally, we denote by $[0,\infty]$ the Alexandrov extension of $\R_+$, which is compact by definition.

\begin{thm}
\label{thm:dicretecv2}

The following convergence holds in distribution in $\D \left([0,+\infty],\bL(\oD)\right)$:

\begin{equation*}
\left(\cL^{(n)}_{c}\right)_{c \in [0,+\infty]}  \quad \mathop{\longrightarrow}^{(d)}_{n \rightarrow \infty} \quad  \left(\bL_c^{(2)}\right)_{c \in [0,+\infty]}.
\end{equation*}

\end{thm}

To establish this result, we actually prove a more general result (Theorem~\ref{thm:discreteconvergence} below). We show that $ (\bL_c^{(2)})_{c \in [0,+\infty]}$ is the functional limit of discrete lamination valued-processes, obtained by marking vertices of discrete trees  (this can be seen as the discrete analogue of the Aldous-Pitman fragmentation). We then use a bijection between the set of minimal factorizations of the $n$-cycle and a subset of plane  trees with $n$ labelled vertices, which allows us to reformulate Theorem~\ref{thm:dicretecv2} in terms of random trees. The main difficulty is that the labelling of the vertices has constraints. To lift these constraints and to reduce the study to uniform labellings,  an important tool in the study of these random trees is an operation that shuffles the labels of their vertices in two ways (see Section \ref{sec:permutations} for details).

The process $(\bL_c^{(2)})_{c \in [0,+\infty]}$ is therefore the limit of discrete lamination-valued processes which code a uniform minimal factorization into transpositions; in a forthcoming work, we establish an analogous result concerning the processes $(\bL_c^{(\alpha)})_{c \in [0,+\infty]}$ for $1<\alpha<2$, by proving that they appear as limits of discrete lamination-valued processes which code other random factorizations of the $n$-cycle. Notably, cycles of length $\geq 3$ are allowed in these new models of factorizations.

\subsection{Coding \texorpdfstring{$\bL_c^{(\alpha)}$}{Lcalpha} by a function}

For fixed $\alpha \in (1,2]$ and $c>0$, we show that $\bL_c^{(\alpha)}$ can be coded by a Lévy process, similarly to the way $\bL_\infty^{(\alpha)} $ is coded by $H^{(\alpha)}$ in \eqref{eq:bLa}. In the case of $\bL_c^{(\alpha)}$, we introduce the $\alpha$-stable spectrally positive Lévy process $Y^{(\alpha)}$, which is the Lévy process whose Laplace exponent is given by $\E[ e^{-\lambda Y^{(\alpha)}_s} ] = e^{s \lambda^\alpha}$ for $s, \lambda \geq 0$. Then, for any $s \geq 0$, we define the stopping time $\tau^{(\alpha),c}_s$ as 

\begin{align*}
\tau^{(\alpha),c}_s = \inf \left\{ t > 0, Y^{(\alpha)}_t - c^{1/\alpha} t < - c^{1+1/\alpha} s \right\} - c s.
\end{align*}

It is not difficult (see Section~\ref{sec:distribution})  to check that $(\tau^{(\alpha),c}_s)_{s \in \R^+}$ is a Lévy process with Laplace exponent given by 
\begin{equation*}
\E \left[ \exp (- \lambda \tau^{(\alpha),c}_s) \right] = \exp \left(- s \ c \ (\overline{\phi}(\lambda)-\lambda)\right), \qquad \lambda>0, \quad s \geq 0.
\end{equation*}
where $\overline{\phi}(\lambda)$ is the unique nonnegative solution of the equation $X^\alpha + cX = \lambda c$. It is interesting to note that this equation appears in the work of Bertoin \cite[Section $6.1$]{Ber10'}, in the study of a random spatial branching process with emigration.

It turns out that  $\bL_c^{(\alpha)}$ can be coded by the normalized excursion $\tau^{(\alpha),c,exc}$ of the Lévy process $s \mapsto  \tau^{(\alpha),c}_{s}$, as stated in the following theorem:

\begin{thm}
\label{thm:levyprocess}
The following equality holds in distribution, for any $c \geq 0$:
\begin{align*}
\mathbb{L}^{(\alpha)}_{c} \overset{(d)}{=} \bL(\tau^{(\alpha),c,exc})
\end{align*}
\end{thm}
Here, $\bL \left(\tau^{(\alpha),c,exc} \right)$ is the lamination constructed from $\tau^{(\alpha),c,exc}$ by the method described in Section \ref{ssec:fraglam}. 

The main idea of the proof of Theorem \ref{thm:levyprocess} is to exhibit a new family of random trees, which can be seen as a randomly reduced version, in some sense, of Galton-Watson trees conditioned by their number of vertices. It happens that these reduced trees code a new sequence of random laminations, which converges at the same time towards $\bL_c^{(\alpha)}$ and $\bL(\tau^{(\alpha),c,exc})$.

\subsection{An estimate on generating functions}

An important ingredient to code $\bL_c^{(\alpha)}$ by the normalized excursion of $\tau^{(\alpha),c}$,  which is crucial in the proof of Theorem \ref{thm:levyprocess} and which we believe to be of independent interest, is a general estimate of the behavior of generating functions in the complex plane, involving slowly varying functions. Recall that a function $L: \R_+ \rightarrow \R_+^*$ is slowly varying if, for any $c>0$, ${L(cx)}/{L(x)} \rightarrow 1$ as $x \rightarrow +\infty$. Precise estimates concerning these functions are needed in many different contexts, see e.g. \cite{Sla68} \cite{Tok08}, although nothing seems to have been proved regarding asymptotics in the complex domain.

\begin{thm}
\label{thm:estimate}

Let $\mu$ be a probability distribution on the nonnegative integers and denote by $F_{\mu}$ its generating function. Assume that there exists $ \alpha \in (1,2]$ and a slowly varying function $L: \R_{+} \rightarrow \R^{*}_{+}$ such that

$$F_\mu(1-s) - (1-s)  \quad \underset{s \downarrow 0}{\sim} \quad s^\alpha L \left( \frac{1}{s} \right).$$

Then

$$F_\mu(1 + \omega) - (1 + \omega) \underset{\substack{|\omega| \rightarrow 0 \\ |1+\omega|<1}}{\sim} (-\omega)^\alpha L \left( \frac{1}{|\omega|} \right).$$
\end{thm}

In the terminology of Galton-Watson trees, this is an estimate in the complex unit disk, near $1$, of the generating function of a critical offspring distribution which belongs to the domain of attraction of a \textit{stable law}. Very often, additional assumptions, such as $\Delta$-analyticity, are made in order to obtain estimates for generating functions in the complex plane (see \cite[Section 6]{FS09}). Observe that here it is not the case, and no assumptions on $L$ are made.

The proof of this theorem is given in Section~\ref{sec:appendix}. The main idea is to use an integral representation, see \eqref{eq:estim}.

\paragraph*{Outline}

After describing a general construction of trees and laminations coded by excursion-type functions, we rigorously define the process $(\bL_c^{(\alpha)})_{c \in [0,+\infty]}$ in Section~\ref{sec:construction} and prove Theorem~\ref{thm:existence}. In a second time, in Section~\ref{sec:discrete}, we make $(\bL_c^{(\alpha)})_{c \in [0,+\infty]}$ appear as the limit of a process of laminations coded by discrete trees; this framework is used in Section~\ref{sec:permutations} to extend the results of Féray \& Kortchemski \cite{FK17} and highlight a relation between the Aldous-Pitman fragmentation of the Brownian tree and minimal factorizations of the $n$-cycle as $n \rightarrow \infty$. Finally, in Section~\ref{sec:distribution}, we recover the $1$-dimensional marginal of the lamination process as the lamination coded by $\tau^{(\alpha),c,exc}$ (Theorem~\ref{thm:levyprocess}), while Section \ref{sec:appendix} is devoted to the proof of Theorem \ref{thm:estimate}.

\paragraph*{Acknowledgements.}

I would like to thank Igor Kortchemski for asking the questions at the origin of this paper and for his help, comments, suggestions and corrections. I would also like to thank Bénédicte Haas, Cyril Marzouk and Loïc Richier for fruitful discussions and comments on the paper, and Ger\'onimo Uribe Bravo for his useful remarks on Theorem \ref{thm:cub}.

\section*{Notations}

Let us immediately sum up some notations that will often appear throughout the paper. We write $\overset{\P}{\rightarrow}$ for the convergence in probability, and $\overset{(d)}{\rightarrow}$ for the convergence in distribution of a sequence of random variables. We say that an event $E_n$ (depending on $n$) occurs with high probability if $\P(E_n) \rightarrow 1$ as $n \rightarrow \infty$. When talking about trees, deterministic ones will be denoted by a straight $T$, while random ones will be denoted by a curved $\cT$. Finally, at the beginning of each section, we sum up the most important notations that we use in this section.

\section{Construction of lamination-valued processes}

\label{sec:construction}

This section is devoted to the construction of càdlàg processes taking their values in the set of laminations of the unit disk. We start by explaining a general method of construction of lamination-valued processes, starting from a deterministic excursion-type function. Then, we apply this in the particular case of an $\alpha$-stable excursion, for $\alpha \in (1,2]$, giving birth to a random lamination-valued process.

\subsection*{Notations of Section \ref{sec:construction}}

In this table of notations, $f$ always denotes a continuous excursion-type function such that $f(0)=0$ (except for $\bL(f)$, which is defined for any excursion-type function). $u \coloneqq (s,t)$ denotes an element of $\R^2$.

\renewcommand{\arraystretch}{1.6}
\begin{center}
\begin{tabular}{|c|c|}
\hline
$\cEG(f)$ & epigraph of $f$\\
\hline
$g(f,u),d(f,u)$ & $\sup \{ s' \leq s, f(s') < t \}, \, \inf \{ s' \geq s, f(s') < t \}$\\
\hline
$\cN_c(f)$ & Poisson point process of intensity $\frac{2c ds dt}{d(f,u)-g(f,u)}$ on  $\cEG(f)$\\
\hline
$\cP_c(T)$ & Poisson point process of intensity $c d\ell$ on a tree $T$\\
\hline
$\bL(f)$ & lamination coded by $f$\\
\hline
$\bL_c(f)$ & lamination coded by $\cN_c(f)$\\
\hline
$\cT^{(\alpha)}$ & $\alpha$-stable tree\\
\hline
$H^{(\alpha)}$ & contour function of the $\alpha$-stable tree\\
\hline
$\bL_\infty^{(\alpha)}$ & $\alpha$-stable lamination, coded by $H^{(\alpha)}$\\
\hline
$\bL_c^{(\alpha)}$ & lamination coded by $\cN_c\left(H^{(\alpha)}\right)$\\
\hline 
\end{tabular}
\end{center}

\subsection{Excursions and laminations}
\label{ssec:excursions}

Starting from an excursion-type function $f$, we have seen in Section \ref{sec:intro} that we can define a lamination $\bL(f)$. In the particular case of a continuous $f$ verifying $f(0)=0$, we shall recall in this section the classical construction of the tree $T(f)$ as the quotient of $[0,1]$ by the equivalence relation $\sim_f$ defined in Section \ref{ssec:fraglam}. Then, we shall construct a nondecreasing lamination-valued process $(\bL_c(f))_{0 \leq c \leq \infty}$, by associating chords in the unit disk to straight lines under the graph of $f$ (see Fig. \ref{fig:cut2}), such that $\bL_\infty(f)=\bL(f)$. It is to note that, when $f$ is deterministic, $\bL(f)$ and $T(f)$ are also deterministic, while $(\bL_c(f))_{0 \leq c \leq \infty}$ will be a random process. This coding is used in the next sections, when the function $f$ is the contour function of a tree (later in this section and in the next one) or when it is the standard excursion of the Lévy process $\tau^{(\alpha),c}$ (Section \ref{sec:distribution}).

\paragraph*{The tree associated to continuous excursion-type function.}

Assume that $f$ is a continuous excursion-type function with $f(0)=0$. In this case, the equivalence relation $\sim_f$ defined in Section \ref{ssec:fraglam} can be understood in a nicer way. For $0 \leq s \leq t \leq 1$, define $m(s,t) \coloneqq {\inf}_{{[s,t]}} f$ and $d(s,t) \coloneqq f(s)+f(t)-2 m(s,t)$. For $t>s$, set $d(t,s)=d(s,t)$. For $s,t \in [0,1]$, we write $s \sim_f t$ if $d(s,t)=0$, which matches the definition of Section \ref{ssec:fraglam}. 

From this continuous function $f$, we define the tree $T(f)$ as
\begin{align*}
T(f) = [0,1]/\sim_f.
\end{align*}
One can check (see \cite{DLG02}) that $d$ induces a distance on $T(f)$, which we still denote by $d$ with a slight abuse of notation, and that the metric space $(T(f),d)$ is a tree, in the sense that from one point of $T(f)$ to another, there exists a unique path in $T(f)$. See Fig.~\ref{fig:stablearbconlam} and \ref{fig:arbconlam} for two examples of a continuous excursion-type function, its associated lamination and its associated tree.

Let us immediately define some important notions about trees. We say that an equivalence class $\overline{x} \in T(f)$ is a \textit{branching point} if $T(f) \backslash \{\overline{x} \}$ has at least three disjoint connected components, and the set of points that are not branching points is called the \textit{skeleton} of $T(f)$. A \textit{leaf} of the tree is an equivalence class $\overline{x}$ such that $d(\overline{0},\cdot)$ has a local maximum at $\overline{x}$ in $T(f)$ (where $\overline{0}$ denotes the equivalence class of $0$). In other words, a branching point is a point where the tree splits into two or more branches, and leaves are ends of branches. The \textit{volume measure} $h$, or mass measure on $T(f)$, is defined as the projection on $T(f)$ of the Lebesgue measure on $[0,1]$. Finally, the \textit{length measure} $\ell$ on $T(f)$, supported by the set of non-leaf points, is the unique $\sigma$-finite measure on this set such that, for $x,y \in T(f)$ non-leaf points, $\ell([x,y]) = d(x,y)$, where $[x,y]$ is the path from $x$ to $y$ in $T(f)$. See \cite{ALW17} for further details about this length measure. This $\sigma$-finite measure expresses the intuitive notion of length of a branch in the tree.

\paragraph*{Poisson point processes on epigraphs.} Assume as above that $f$ is a continuous excursion-type function with $f(0)=0$. We explain how to obtain a Poisson point process on the skeleton of $T(f)$ from a Poisson point process under the graph of $f$. First, define the \textit{epigraph} of $f$, denoted by $\cEG(f)$, as the set of points under the graph of $f$:

\begin{align*}
\cEG(f) \coloneqq \left\{ (s,t) \in \R^2: s \in (0,1), 0 \leq t < f(s) \right\}.
\end{align*}

To $u \coloneqq (s,t) \in \cEG(f)$, associate $g(f,u) \coloneqq \sup \{ s' \leq s, f(s')<t \}$ and $d(f,u) \coloneqq \inf \{ s' \geq s, f(s')<t \}$ (see  Fig.~\ref{fig:cut2}). In particular, note that one can associate  to each $u \in \cEG(f)$ the chord $[e^{-2i\pi g(f,u)},e^{-2i\pi d(f,u)}]$, and that for two different points of $\cEG(f)$, the associated chords are either equal or disjoint.

We now consider a Poisson point process $\cN(f)$ on $\R^2 \times \R_+$, with intensity

$$\frac{2}{d(f,u)-g(f,u)} \mathds{1}_{u \in \cEG(f)} {\mathrm{d}}u {\mathrm{d}r},$$
thinking of the second coordinate as time. Using $ \mathcal{N}(f)$, for every $c \geq 0$, we shall now define $\cN_c(f)$, $\cP_c(T(f))$, $\bL_c(f)$ (see Fig. \ref{fig:cut1} and  \ref{fig:cut2}).

\emph{Definition of $\cN_c(f)$.} For $c\geq 0$, let $\cN_c(f)$ be the projection on the first coordinate of $\cN(f) \cap (\R^2 \times [0,c])$.  Roughly speaking, $\cN_c(f)$ is the set of  all points that have appeared before or at time $c$. Therefore $\cN_{c}(f)$ is a Poisson point process on $\cEG(f)$ of intensity $\frac{2 c}{d(f,u)-g(f,u)} \mathds{1}_{u \in \cEG(f)} {\mathrm{d}}u$. Moreover, the processes $(\cN_c(f))_{c \geq 0}$ are coupled in a nondecreasing way.

\emph{Definition of $\cP_c(T(f))$.} To $u \in \cN_c(f)$, associate the vertex $x_u \in T(f)$, which is the equivalence class of $g(f,u)$ in $T(f)$ for $\sim_f$ (see Fig.~\ref{fig:cut1}). Then  $\cP_c(T(f)) \coloneqq \{ x_u, u \in \cN_c(f) \}$ is a Poisson point process on $T(f)$ of intensity $c d\ell$. It can be checked that there are only countably many branching points in $T(f)$, and therefore almost surely all points of $\cP_c(T(f))$ are points of the skeleton of $T(f)$. Furthermore, by construction, the process $(\cP_c(T(f)))_{c \geq 0}$ is nondecreasing for the inclusion.  

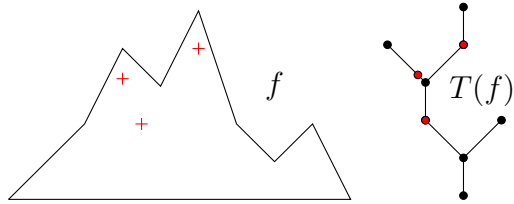
\begin{figure}[h!]

\caption{A continuous excursion-type function $f$ with three points in its epigraph, which correspond to three points in its associated tree $T(f)$.}

\center

\label{fig:cut1}
\begin{tabular}{c c}
\begin{tikzpicture}[scale=0.5]
\draw (0,0) -- (2,2) -- (3,4) -- (4,3) -- (5,5) -- (6,2) -- (7,1) -- (8,2) -- (9,0) -- (0,0);
\draw (3.5,2) node[cross=2pt,rotate=45,red]{};
\draw (3,3.2) node[cross=2pt,rotate=45,red]{};
\draw (5,4) node[cross=2pt,rotate=45,red]{};
\draw (7,3) node{  $f$};
\end{tikzpicture}
&
\begin{tikzpicture}[scale=0.5]
\draw (1,2) -- (0,1) -- (-1,2) -- (-1,3) (-2,4) -- (-1,3) -- (0,4) -- (0,5) (0,0) -- (0,1);

\draw[fill] (0,1) circle (.1);

\draw[fill] (1,2) circle (.1);

\draw[fill] (-1,2) circle (.1);

\draw[fill] (0,0) circle (.1);

\draw[fill] (-2,4) circle (.1);

\draw[fill] (0,5) circle (.1);

\draw[fill] (0,4) circle (.1);

\draw[fill] (-1,3) circle (.1);

\draw[fill=red] (-1,2) circle (.1);

\draw[fill=red] (0,4) circle (.1);

\draw[fill=red] (-1.2,3.2) circle (.1);

\draw (.5,2.8) node{$T(f)$};

\end{tikzpicture}
\end{tabular}
\end{figure}

\emph{Definition of $\bL_c(f)$.} Finally, associate to $\cN_c(f)$ the lamination $\bL_c(f)$ as follows: $\bL_c(f)$ is a sublamination of $\bL(f)$, constructed by drawing only the chords that correspond to the points of $\cN_c(f)$. More precisely,

$$\bL_c(f)  \quad \coloneqq  \quad \mathbb{S}^1 \cup \bigcup_{u \in \cN_c(f)}[e^{-2i\pi g(f,u)}, e^{-2i\pi d(f,u)}].$$
Define finally
$$\bL_\infty(f) \coloneqq \overline{\bigcup_{c \geq 0} \bL_c(f)}.$$
Remark that, since $f$ is continuous, $\bL_\infty(f)$ is exactly $\bL(f)$ as defined in Section \ref{ssec:fraglam}.

The next proposition highlights a relation between the mass sequence of $\bL_c(f)$ and the mass measure on the tree $T(f)$. For $f$ a continuous excursion-type function on $[0,1]$ with $f(0)=0$ and $c\geq 0$ fixed, let $\textit{\textbf{m}}_c(f)$ be the sequence of $h$-masses of the connected components of $T(f)$ delimited by the points of $\cP_c(T(f))$, sorted in nondecreasing order.

\begin{prop}
\label{prop:massequality}
Let $f$ be a continuous excursion-type function on $[0,1]$ with $f(0)=0$. Then the following equality holds almost surely in $\Delta$:

\begin{align*}
\left( \cM\left[ \bL_c(f) \right] \right)_{c \geq 0} = \left(\textbf{m}_c(f)\right)_{c \geq 0}.
\end{align*}
\end{prop}

\begin{proof}
Fix $c>0$. For any $u \coloneqq (s,t) \in \cN_c(f)$, draw $I_u(f) \coloneqq [(g(f,u), t), (d(f,u),t)]$ the horizontal line in $\cEG(f)$ containing $u$ (see Fig.~\ref{fig:cut2}). As seen above, almost surely the corresponding vertex $x_u \in \cP_c(T(f))$ is not a branching point, and therefore the line $I_u$ separates the epigraph into exactly two connected components. Let $\ell_u(f)=d(f,u)-g(f,u)$ be the length of $I_u(f)$. The cutpoint of $\cP_c(T(f))$ corresponding to $u$ splits $T(f)$ into two connected components of $h$-masses $\ell_u(f)$ and $1-\ell_u(f)$, by definition of $h$. On the other hand, the chord corresponding to $u$ in $\bL(f)$ splits the disk into two components of masses $\ell_u(f)$ and $1-\ell_u(f)$. The result follows, since this holds jointly for all $c>0$ and all $u \in \cN_c(f)$.
\end{proof}

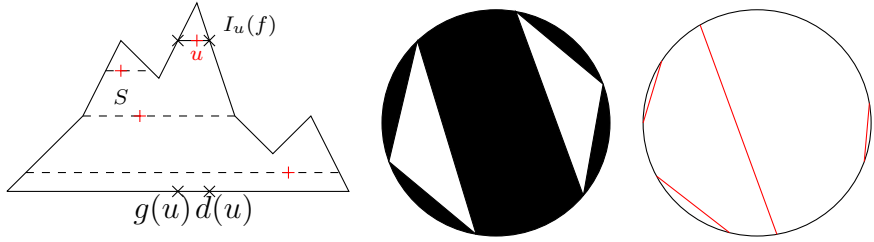
\begin{figure}[h!]

\caption{From left to right: a continuous excursion-type function $f$ with four points on its epigraph and the five components of $\cEG(f)$ delimited by these points; the lamination $\mathbb{L}(f)$ coded by $f$; its sublamination formed by the chords associated to these four points.}

\center

\label{fig:cut2}
\begin{tabular}{c c c}
\begin{tikzpicture}[scale=.5]

\draw (0,0) -- (2,2) -- (3,4) -- (4,3) -- (5,5) -- (6,2) -- (7,1) -- (8,2) -- (9,0) -- (0,0);

\draw[dashed] (2,2) -- (6,2);

\draw (4.5,4) -- (5.33,4);

\draw[dashed] (2.6,3.2) -- (3.8,3.2);

\draw (3.5,2) node[cross=2pt,rotate=45,red]{};

\draw (3,3.2) node[cross=2pt,rotate=45,red]{};

\draw (5,4) node[cross=2pt,rotate=45,red]{};

\draw (4.5,4) node[cross=2.5pt]{};

\draw (5.33,4) node[cross=2.5pt]{};

\draw (3,2.5) node{\scriptsize $S$};

\draw (5,3.6) node[red]{ \scriptsize $u$};

\draw (6.4,4.4) node{\scriptsize $I_{u}(f)$};

\draw[dashed] (.5,.5) -- (8.75,.5);

\draw (7.41,.5) node[cross=2pt,rotate=45,red]{};

\draw (4.5,0) node[cross=2.5pt]{};

\draw (5.33,0) node[cross=2.5pt]{};

\draw (4.1,-.5) node{$g(u)$};

\draw (5.73,-.5) node{$d(u)$};
\end{tikzpicture}

&

\begin{tikzpicture}[scale=1.5]
\draw[fill=black] (0,0) circle (1);

\draw[fill=white] ({cos(-40*1)},{sin(-40*1}) -- ({cos(-40*7)},{sin(-40*7}) -- ({cos(-40*8.5)},{sin(-40*8.5}) -- ({cos(-40*1)},{sin(-40*1});

\draw[fill=white] ({cos(-40*2.5)},{sin(-40*2.5}) -- ({cos(-40*4)},{sin(-40*4}) -- ({cos(-40*5.66)},{sin(-40*5.66}) -- ({cos(-40*2.5)},{sin(-40*2.5});

\end{tikzpicture}

&

\begin{tikzpicture}[scale=1.5]

\draw (0,0) circle (1);

\draw[red] ({cos(-40*2)},{sin(-40*2}) -- ({cos(-40*6)},{sin(-40*6});

\draw[red] ({cos(-40*3.8)},{sin(-40*3.8}) -- ({cos(-40*2.6)},{sin(-40*2.6});

\draw[red] ({cos(-40*4.5)},{sin(-40*4.5}) -- ({cos(-40*5.33)},{sin(-40*5.33});

\draw[red] ({cos(-40*.5)},{sin(-40*.5}) -- ({cos(-40*8.75)},{sin(-40*8.75});
\end{tikzpicture}

\end{tabular}

\end{figure}

We end this subsection by highlighting the nested structure of the lamination-valued process $(\bL_c(f))_{c \geq 0}$.

\begin{prop}

\label{prop:properties}

Let $f$ be a continuous excursion-type function such that $f(0)=0$. Then:

\begin{enumerate}

\item[(i)] for every $0 \leq c \leq c'$, $\bL_c(f) \subset \bL_{c'}(f) \subset \bL(f)$;

\item[(ii)]the convergence $\underset{c \rightarrow \infty}{\lim} \bL_c(f) = \bL(f)$ holds almost surely for the Hausdorff distance.

\end{enumerate}

\end{prop}

The first assertion is straightforward by definition of $(\bL_c(f))_{c \geq 0}$, while Proposition~\ref{prop:properties} (ii) is a consequence of the following deterministic lemma. The idea is to choose a finite subset of chords of $\bL(f)$ which is close to the whole lamination $\bL(f)$, and then prove that, as $c$ grows, this finite subset of chords is well approximated by $\bL_c(f)$. For $\epsilon>0$, we say that $L'$ is an $\epsilon$-sublamination of $L$ if $L' \subset L$ and $d_{H}( L', L) \leq \epsilon$.

\begin{lem}

\label{lem:sublamination}

Fix $\epsilon > 0$. There exists a deterministic constant $K_\epsilon \in \Z_+$ such that any lamination has an $\epsilon$-sublamination with at most $K_\epsilon$ chords. 

\end{lem}

\begin{proof}

Set $r\coloneqq \lfloor {2\pi}/{\epsilon} \rfloor + 1$ and let $I_r$ be the set of arcs of the form $(e^{-2i\pi k/r}, e^{-2i\pi (k+1)/r})$ for $k \in \llbracket 0,r-1 \rrbracket$. Fix a lamination $L$  and remark that, for $a_1,a_2$ two arcs of $I_r$, two chords of $L$ connecting $a_1$ to $a_2$ are at Hausdorff distance at most $ \epsilon$. Therefore, we construct an $\epsilon$-sublamination of $L$ by choosing, for each pair $(a_1,a_2) \in I_r^2$ such that $L$ contains at least one chord connecting $a_1$ and $a_2$, exactly one of them. By construction, the sublamination $L'$ made of $\mathbb{S}^1$ and these chords is at Hausdorff distance at most $\epsilon$ of $L$. The result follows, with $K_\epsilon = |I_r|^2 \leq \left( \lfloor {2\pi}/{\epsilon} \rfloor + 1 \right)^2$.

\end{proof}

\begin{proof}[Proof of Proposition~\ref{prop:properties} (ii)]
Fix $\epsilon > 0$. Using Lemma~\ref{lem:sublamination}, take $L'$ an $\epsilon$-sublamination of $\bL(f)$ with at most $K_\epsilon$ chords, and consider the points in $\cEG(f)$ corresponding to the chords of $L'$. Let $u$ be one of these points and set $g\coloneqq g(f,u)$, $d\coloneqq d(f,u)$ to simplify notation. If $g=0$ then the chord associated to $u$ is reduced to a point of $\mathbb{S}^1$, and therefore is in $\bL_c(f)$ for all $c \geq 0$. If $g \neq 0$, set $m=\max( \, \inf_{ [g-\epsilon, g] } f , \,\inf_{[d, d+\epsilon]} f \, )$. By definition of $g$ and $d$, $f(g) = f(d)$ is not a local minimum of $f$ at $g$ nor at $d$, which implies $m < f(g)$. Therefore, $[g,d] \times [m, f(g)]$ has positive $2$-dimensional Lebesgue measure. Moreover, a point of this set corresponds to a chord at distance at most $2 \pi \epsilon$ of the chord corresponding to $u$. Hence, with  probability tending to $1$ as $c \rightarrow \infty$, there exists a point of $\cN_c(f)$ in this set. This means that $\P( d_{H}( \bL_c(f), L')> 2\pi \epsilon) \rightarrow 0$ as $c \rightarrow \infty$, which concludes the proof, since $L'$ is an $\epsilon$-sublamination of $\bL(f)$.
\end{proof}

\subsection{Construction of \texorpdfstring{$(\bL_c^{(\alpha)})_{c \in [0,+\infty]}$}{Lca}}

\label{ssec:construction}

Fix $\alpha \in (1,2]$. We are now ready to introduce the lamination valued-process $( \bL_c^{(\alpha)} )_{c \in [0,+\infty]}$. To this end, denote by $H^{(\alpha)}=(H^{(\alpha)}_{t})_{0 \leq t \leq 1}$ the continuous normalized $\alpha$-stable height process defined in \cite[Chapter 1]{DLG02}. In particular, $H^{(\alpha)}$ is a continuous excursion-type function with $H^{(\alpha)}_{0}=0$. In addition, for $\alpha=2$, $H^{(2)}$ is (a multiple of) the Brownian excursion and $\cT^{(2)} \coloneqq T(H^{(2)})$ is Aldous' Brownian tree. 

We now specify the definitions of Section \ref{ssec:excursions} with the random excursion-type function $H^{(\alpha)}$, by letting $\cT^{(\alpha)} \coloneqq T\left( H^{(\alpha)} \right)$ and $\bL_\infty^{(\alpha)} \coloneqq \bL (H^{(\alpha)})$ be respectively the $\alpha$-stable tree and the $\alpha$-stable lamination. Finally, we set $(\bL_c^{(\alpha)})_{c \geq 0}=(\bL_c(H^{(\alpha)}))_{c \geq 0}$. 

The proof of Theorem~\ref{thm:existence} is now just an application of Proposition~\ref{prop:massequality} in this specific case.

\begin{rk}
The lamination-valued process $( \bL_c^{(\alpha)} )_{c \in [0,+\infty]}$ is almost surely càdlàg. Indeed, the process is nondecreasing and therefore admits a limit from the left and from the right at each $c>0$. Furthermore, for any $c \geq 0$, any $\epsilon>0$, one can check that almost surely there are only finitely many chords of length $> \epsilon$ in $\bL_{c+1}^{(\alpha)}$, and therefore there exists $\delta>0$ such that no chord of length $> \epsilon$ appears in the process between times $c$ and $c+\delta$. Hence the process is right-continuous.
\end{rk}

\subsection{A limit theorem for lamination-valued processes}

We exhibit here a way of translating the convergence of a sequence of excursion-type functions to the convergence of the associated lamination-valued processes.

\begin{thm}

\label{thm:discreteconvergence}

Let $(f_{n})_{n \geq 1}$ be a sequence of continuous excursion-type functions such that $f_{n}(0)=0$ for every $n \geq 1$. Assume that $(f_{n})$ converges uniformly to a continuous excursion-type function  $f$  such that $f(0)=0$. Then, for every $c>0$, the convergence

\begin{align*}
\left( \bL_{s}(f_{n}) \right)_{s \in [0,c]}  \quad \mathop{\longrightarrow}^{(d)}_{n \rightarrow \infty} \quad \left( \bL_s(f) \right)_{s \in [0,c]}
\end{align*}
holds in distribution in the space $\D([0,c],\bL(\overline{\D}))$.
\end{thm}

In general, the convergence of Theorem \ref{thm:discreteconvergence} does not hold in $\D([0,\infty],\bL(\overline{\D}))$. Nevertheless, it is the case when the functions $f_n$ are the contour functions of certain trees (Theorem \ref{thm:discreteconvergencestable}).

The idea of the proof is to focus on the emergence of large chords, and to prove that there is only a finite number of them that appear up to time $c$. To this end, one reformulates the emergence of large chords in terms of the Poisson point processes $\cN_{c}(f_{n})$ and $\cN_{c}(f)$.

Let us introduce some notation. For an integer $s \geq 1$ and $k \in \llbracket 0,s-1 \rrbracket$, we denote by $x_k$ the arc of the form $(e^{-2i\pi k/s}, e^{-2i\pi (k+1)/s})$. We furthermore define, for any $K \geq 1$, $I_s^{(K)} := \{(x_{i_1}, \ldots, x_{i_K}) \in I_s^K, i_1 \leq i_2 \leq \ldots \leq i_K \}$. Fix $\epsilon>0$ and an integer $K \geq 1$. Take $A = (a_1,\ldots,a_K)  \in I_s^{(K)}$, $B = (b_1,\ldots,b_K) \in I_s^{(K)}$, as well as $R = (r_1(i),r_2(i))_{1 \leq i \leq K} \subset ([0,c]^{2})^{K}$ with $r_1(i) < r_2(i)$ for every $1 \leq i \leq K$. 

Now, given a nondecreasing lamination-valued process $\kL \coloneqq (L_r)_{r \in [0,+\infty]}$,  we define the event $E^c_{A,B,R}(\kL)$ as follows:

$E^c_{A,B,R}(\kL)$: ``$L_c$ has exactly $K$ chords of length greater than $\epsilon$, which can be indexed so that the $i$-th one connects the arcs $a_i$ and $b_i$, and has appeared between times $r_1(i)$ and $r_2(i)$.''

To simplify notation, we set $\cL(f_{n}) =(\bL_r(f_{n}) )_{r \in [0,c]}$ and $\cL(f)=( \bL_r(f) )_{r \in [0,c]}$. The following result is the key ingredient to prove Theorem \ref{thm:discreteconvergence}:

\begin{prop}

\label{prop:K}

The following convergence holds:

\begin{align*}
\P\left(E^c_{A,B,R}(\cL(f_n))\right)  \quad \mathop{\longrightarrow}_{n \rightarrow \infty} \quad  \P \left(E^c_{A,B,R}(\cL(f))\right).
\end{align*}

\end{prop}

Let us immediately see how this implies our main result:

\begin{proof}[Proof of Theorem \ref{thm:discreteconvergence} from Proposition \ref{prop:K}] Define the diameter of $R$ as $\Delta(R) \coloneqq \max \{ r_2(i) - r_1(i), 1 \leq i \leq K \}$.  The  idea of the proof is the following observation: for $\kL \coloneqq (L_r)_{r \in [0,\infty]},\kL' \coloneqq (L'_r)_{r \in [0,\infty]}$ two processes, if $E^c_{A,B,R}(\kL)$ and $E^c_{A,B,R}(\kL')$ both hold, where $R$ has diameter $\leq D$, then the Skorokhod distance between $(L_r)_{r \leq c}$ and $(L'_r)_{r \leq c}$ is deterministically bounded by a constant $C(K,\epsilon, s, D)$, which tends to $0$ as $\epsilon, 1/s, D \rightarrow 0$.  Also, for fixed $K$ and $R$, for different couples $(A,B) \in (I_s^{(K)})^2$, the events $E^c_{A,B,R}(\cL(f_n))$ are all disjoint almost surely, and there exists only a finite number of events of this form (for $\epsilon, s, R$ fixed).

In order to use Proposition \ref{prop:K} and prove Theorem \ref{thm:discreteconvergence}, it is therefore enough to show that the number of chords of length greater than $\epsilon$ in the lamination $\bL_{c}(f_{n})$ is tight as $n \rightarrow \infty$. For this, remark that, for any $\delta>0$, the expectation of the number of chords in $\bL_{c}(f_{n})$ corresponding to points $u \in \cN_{c}(f_{n})$ such that $d(f_n,u)-g(f_n,u)> \delta$ has the expression:

\begin{align*}
\int_{\R^2} \frac{2c}{d(f_n,u)-g(f_n,u)} du \mathds{1}_{d(f_n,u) - g(f_n,u)>\delta} \mathds{1}_{u \in \cEG(f_{n})} \leq \frac{2c}{\delta} \| f_{n} \|_{\infty}.
\end{align*}
Furthermore, a chord in $\bL_{c}(f_{n})$ of length greater than $\epsilon$ necessarily corresponds to a point $u \in \cN_{c}(f_{n})$ such that $d(f_n,u)-g(f_n,u)> \epsilon/2\pi$.
Since $(f_{n})$ converges uniformly, it follows that the number of chords in $\bL_{c}(f_{n})$ whose length is greater than $\epsilon$ is asymptotically stochastically bounded by a Poisson distribution. By taking $\epsilon, 1/s, \Delta(R) \rightarrow 0$, we get the desired result.

\end{proof}

It remains to prove Proposition \ref{prop:K}.

\begin{proof}[Proof of Proposition \ref{prop:K}]
By inclusion-exclusion, we can assume that the couples $(a_i,b_i)$ for $1 \leq i \leq K$ are all different.  The idea of the proof is to reformulate the events $E^c_{A,B,R}(\cL(f_{n}))$ and $E^c_{A,B,R}(\cL(f))$ in terms of the Poisson point processes $\cN(f_{n})$ on $\cEG(f_{n})$ and $\cN(f)$ on $\cEG(f)$. We write, for $1 \leq i \leq K$, $a_i = (e^{-2i\pi j_i/s}, e^{-2i \pi (j_i+1)/s})$ and $b_i = (e^{-2i\pi k_i/s}, e^{-2i \pi (k_i+1)/s})$ for some $j_i$, $k_i \in \llbracket 0,s-1 \rrbracket$. The probability that $\bL_c(f)$ has exactly $K$ chords of length greater than $ \epsilon$, the $i$-th of them connecting $a_i$ to $b_i$ and having appeared between times $r_1(i)$ and $r_2(i)$, is equal to

\begin{align*}
\P \left( \nexists u \in \cN_c(f) \cap \cA_{K+1}(f) \right) \prod_{i=1}^K \P \left( \exists! u \in (\cN_{r_2(i)}(f_{n})  \backslash \cN_{r_1(i)}(f_{n})) \cap \cA_i(f) \right),
\end{align*}
where, for $1 \leq i \leq K$, we have set $\cA_i(f) = \{ u \in \cEG (f), \, d(f,u)-g(f,u) > \epsilon, \, g(f,u) \in [j_i/s, (j_i+1)/s], \, d(f,u) \in [k_i/s, (k_i+1)/s] \}$ and $\cA_{K+1}(f) = \{ u \in \cEG (f), \, d(f,u)-g(f,u) > \epsilon  \} \backslash \cup_{i=1}^K \cA_i(f)$. A similar formula holds with $f$ replaced by $f_{n}$.

Therefore, proving Proposition~\ref{prop:K} boils down to proving that, for any $(a,b,x,y) \in [0,1]^4$, any $0 \leq r_1<r_2$:

\begin{align}
\label{eq:eq2}
&\int_\R \mathds{1}_{r_1 \leq t \leq r_2} dt \int_{\R^2} \frac{2c}{d(f_n,u)-g(f_n,u)} du \mathds{1}_{d(f_n,u) - g(f_n,u)>\epsilon, g(f_n,u) \in [a,b], d(f_n,u) \in [x,y]} \mathds{1}_{u \in \cEG(f_{n}  )} \nonumber\\
& \qquad  \qquad \longrightarrow \int_\R \mathds{1}_{r_1 \leq t \leq r_2} dt \int_{\R^2} \frac{2c}{d(f,u)-g(f,u)} du \mathds{1}_{d(f,u) - g(f,u)>\epsilon, g(f,u) \in [a,b], d(f,u) \in [x,y]} \mathds{1}_{u \in \cEG( f)}
\end{align}
as $n \rightarrow \infty$. To this end, we use dominated convergence. Indeed, consider $\mathcal{R}$ the set of points $u \coloneqq (s,t) \in \R^2$ such that $f(g(f,u))$ is not attained at a local minimum of $f$ between $g(f,u)$ and $d(f,u)$. Remark that the pointwise convergence of the function under the integral holds for all $u \in \mathcal{R}$, and its complement $\mathcal{R}^c$ has Lebesgue measure $0$. Furthermore, for every $n \geq 1$ and  $u \in \mathcal{R}$,

$$\frac{2c}{d(f_n,u)-g(f_n,u)} \mathds{1}_{d(f_n,u) - g(f_n,u)>\epsilon, g(f_n,u) \in (a,b), d(f_n,u) \in (x,y)} \mathds{1}_{u \in \cEG(f_{n} )}\leq \frac{2c}{\epsilon} \mathds{1}_{u \in [0,1] \times [0, \| f_{n}\|_{\infty}]},$$
and the convergence \eqref{eq:eq2} follows by dominated convergence, since $(f_{n})$ converges uniformly to $f$.
\end{proof}

\begin{rk}
We make here a small abuse of words, saying that we prove the convergence of these lamination-valued processes towards $(\bL_r(f))_{0 \leq r \leq c}$ under the condition that there are $K$ chords of length $>\epsilon$ in $\bL_c(f_n)$. This has to be understood as follows: under the event that $\bL_c(f)$ has $K$ such chords, with high probability $\bL_c(f_n)$ has exactly $K$ such chords for $n$ large enough, and $(\bL_r(f_n))_{0 \leq r\leq c}$ converges towards $(\bL_r(f))_{0 \leq r\leq c}$ conditioned to have $K$ such chords. Since $\bL_c(f)$ has almost surely a finite number of chords of length $>\epsilon$, this implies the convergence of the unconditioned processes. We will always make this abuse of words, by saying that we prove such convergences on disjoint events, whose union has probability $1$.
\end{rk}

\section{Limit of cut processes on discrete trees}

\label{sec:discrete}

In this section, our goal is to prove that the lamination-valued process $(\bL_c^{(\alpha)})_{c \in [0,+\infty]}$ is the functional limit of a discrete analogue, namely a discrete lamination-valued process constructed from labelled size-conditioned Galton--Watson trees. This is natural since stable trees appear as limits of certain size-conditioned Galton-Watson trees (see  Theorem~\ref{thm:duquesne}) and since  $(\bL_c^{(\alpha)})_{c \in [0,+\infty]}$ is coded by an $\alpha$-stable tree with some additional structure (the Poisson point process $(\cP_c(\cT^{(\alpha)}))_{c \geq 0}$ on its skeleton).

\subsection*{Notations of Section \ref{sec:discrete}}

\renewcommand{\arraystretch}{1.6}
\begin{center}
\begin{tabular}{|c|c|}
\hline
$\mu$ & critical law in the domain of attraction of an $\alpha$-stable law\\
\hline
$\cT$ & nonconditioned $\mu$-Galton-Watson tree\\
\hline
$\cT_n$ & $\mu$-Galton-Watson tree conditioned to have $n$ vertices\\
\hline
$C(\cT_n)$ & contour function of $\cT_n$\\
\hline
$\tilde{C}(\cT_n)$ & renormalized contour function of $\cT_n$\\
\hline
$\bL_{n,c}$ & $\bL_c(\tilde{C}(\cT_n))$\\
\hline
\end{tabular}
\end{center}

\subsection{Background on trees}

We first  define \textit{plane trees}, following Neveu's formalism \cite{Nev86}. First, let $\N^* = \left\{ 1, 2, \ldots \right\}$ be the set of all positive integers, and $\mathcal{U} = \cup_{n \geq 0} (\N^*)^n$ be the set of finite sequences of positive integers, with $(\N^*)^0 = \{ \emptyset \}$ by convention.

By a slight abuse of notation, for $k \in \Z_+$, we write an element $u$ of $(\N^*)^k$ by $u=u_1 \cdots u_k$, with $u_1, \ldots, u_k \in \N^*$. For $k \in \Z_+$, $u=u_1\cdots u_k \in (\N^*)^k$ and $i \in \Z_+$, we denote by $ui$ the element $u_1 \cdots u_ki \in (\N^*)^{k+1}$. A plane tree $T$ is formally a subset of $\mathcal{U}$ satisfying the following three conditions:

(i) $\emptyset \in T$ (the tree has a root);

(ii) if $u=u_1\cdots u_n \in T$, then, for all $k \leq n$, $u_1\cdots u_k \in T$ (these elements are called ancestors of $u$, and the set of all ancestors of $u$ is called its ancestral line; $u_1 \cdots u_{n-1}$ is called the \textit{parent} of $u$);

(iii) for any $u \in T$, there exists a nonnegative integer $k_u(T)$ such that, for every $i \in \N^*$, $ui \in T$ if and only if $1 \leq i \leq k_u(T)$ ($k_u(T)$ is called the number of children of $u$, or the outdegree of $u$).

See an example of a plane tree on Fig. \ref{fig:treecontour}, left. The elements of $T$ are called \textit{vertices}, and we denote by $|T|$ the total number of vertices in $T$. The height $h(u)$ of a vertex $u$ is its distance to the root, that is, the integer $k$ such that $u \in (\N^*)^k$. We define the height of a tree $T$ as $H(T) = {\sup}_{{u \in T}} \, h(u)$. In the sequel, by tree we always mean plane tree unless specifically mentioned.

The \textit{lexicographical order} $\prec$ on $\mathcal{U}$ is defined as follows:  $\emptyset \prec u$ for all $u \in \mathcal{U} \backslash \{\emptyset\}$, and for $u,w \neq \emptyset$, if $u=u_1u'$ and $w=w_1w'$ with $u_1, w_1 \in \N^*$, then we write $u \prec w$ if and only if $u_1 < w_1$, or $u_1=w_1$ and $u' \prec w'$.  The lexicographical order on the vertices of a tree $T$ is the restriction of the lexicographical order on $\mathcal{U}$; for every $0 \leq k \leq |T|-1$ we write $v_k(T)$ for the $(k+1)$-th vertex of $T$ in the lexicographical order.

We do not distinguish between a finite tree $T$, and the corresponding planar graph where each vertex is connected to its parent by an edge of length $1$, in such a way that the vertices with same height are sorted from left to right in lexicographical order.

It is useful to define the contour function $C(T): [0,2n] \rightarrow \R_+$ of a finite plane tree $T$ with $n$ vertices: imagine a particle exploring $T$ from left to right at unit speed. Then, for $0 \leq t \leq 2n-2$, $C_t(T)$ is the distance to the root of the particle at time $t$. For convenience, we set $C_t(T) = 0$ for $2n-2 \leq t \leq 2n$. See Fig. \ref{fig:treecontour} for an example.

\begin{figure}
\center
\caption{A tree $T$ with $9$ vertices labelled à la Neveu, and its contour function $(C_t(T))_{0 \leq t \leq 18}$.}
\label{fig:treecontour}
\begin{tabular}{c c}
\begin{tikzpicture}[scale=0.7]
\draw (1,1) -- (0,0) -- (-1,1) -- (0,2);
\draw (-1,1) -- (-1,2) -- (0,3) -- (0,4);
\draw (-1,2) -- (-2,3) -- (-2,4);
\draw[fill] (0,0) circle (.1);
\draw[fill] (1,1) circle (.1);
\draw[fill] (-1,1) circle (.1);
\draw[fill] (0,2) circle (.1);
\draw[fill] (-1,2) circle (.1);
\draw[fill] (-2,3) circle (.1);
\draw[fill] (-2,4) circle (.1);
\draw[fill] (0,3) circle (.1);
\draw[fill] (0,4) circle (.1);
\draw (1,3) node{$T$};
\draw (-.5,0) node{$\emptyset$};
\draw (-1.5,1) node{$1$};
\draw (0.5,1) node{$2$};
\draw (-1.5,2) node{$11$};
\draw (-2.7,3) node{$111$};
\draw (-2.8,4) node{$1111$};
\draw (-.7,3) node{$112$};
\draw (-.8,4) node{$1121$};
\draw (-.5,2) node{$12$};
\end{tikzpicture}
&
\begin{tikzpicture}[scale=.5]
\draw[->] (0,-1) -- (0,5);
\draw[->] (-1,0) -- (18.5,0);
\draw (0,0) -- (4,4) -- (6,2) -- (8,4) -- (11,1) -- (12,2) -- (14,0) -- (15,1) -- (16,0);
\draw (1,-.2) -- (1,.2);
\draw (2,-.2) -- (2,.2);
\draw (3,-.2) -- (3,.2);
\draw (4,-.2) -- (4,.2);
\draw (5,-.2) -- (5,.2);
\draw (6,-.2) -- (6,.2);
\draw (7,-.2) -- (7,.2);
\draw (8,-.2) -- (8,.2);
\draw (9,-.2) -- (9,.2);
\draw (10,-.2) -- (10,.2);
\draw (11,-.2) -- (11,.2);
\draw (12,-.2) -- (12,.2);
\draw (13,-.2) -- (13,.2);
\draw (14,-.2) -- (14,.2);
\draw (15,-.2) -- (15,.2);
\draw (16,-.2) -- (16,.2);
\draw (17,-.2) -- (17,.2);
\draw (18,-.2) -- (18,.2);
\draw (1,-.5) node{1};
\draw (2,-.5) node{2};
\draw (3,-.5) node{3};
\draw (4,-.5) node{4};
\draw (5,-.5) node{5};
\draw (6,-.5) node{6};
\draw (7,-.5) node{7};
\draw (8,-.5) node{8};
\draw (9,-.5) node{9};
\draw (10,-.5) node{10};
\draw (11,-.5) node{11};
\draw (12,-.5) node{12};
\draw (13,-.5) node{13};
\draw (14,-.5) node{14};
\draw (15,-.5) node{15};
\draw (16,-.5) node{16};
\draw (17,-.5) node{17};
\draw (18,-.5) node{18};
\draw (-.5,-.5) node{0};
\draw (-.2,1) -- (.2,1) (-.2,2) -- (.2,2) (-.2,3) -- (.2,3) (-.2,4) -- (.2,4);
\draw (-.5,1) node{1};
\draw (-.5,2) node{2};
\draw (-.5,3) node{3};
\draw (-.5,4) node{4};
\draw (14,3) node{$C(T)$};
\end{tikzpicture}
\end{tabular}
\end{figure}
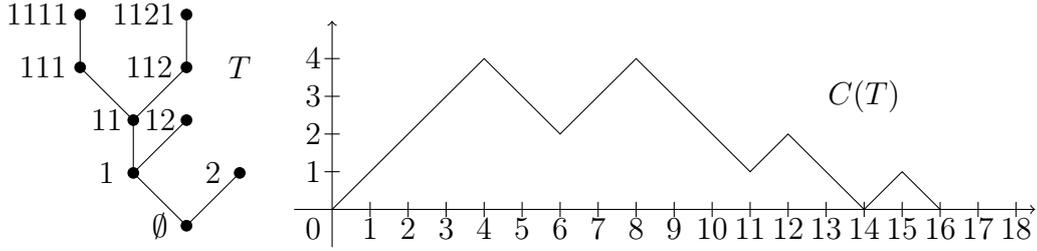

\paragraph*{Slowly varying functions}

Slowly varying functions appear in the study of the domain of attraction of $\alpha$-stable laws (for $\alpha \in (1,2]$). We recall here their definition and useful properties.

A function $L: \R_+ \rightarrow \R_+^*$ is said to be slowly varying if, for any $c>0$,
\begin{align*}
\frac{L(cx)}{L(x)} \underset{x \rightarrow \infty}{\rightarrow} 1.
\end{align*}
As their name says, such functions vary slowly, and in particular more slowly than any polynomial. This statement is quantified by the following useful Potter bounds (see e.g. \cite[Theorem $1.5.6$]{BGT89} for a proof):

\begin{thm}[Potter bounds]
\label{thm:potterbounds}
Let $L: \R_+ \rightarrow \R_+^*$ be a slowly varying function. Then, for any $\epsilon>0$, $A>0$, there exists $X>0$ such that, for $x,y \geq X$,
\begin{align*}
\frac{L(y)}{L(x)} \leq A \max \left\{ \left( \frac{y}{x}\right)^\epsilon, \left( \frac{x}{y} \right)^\epsilon \right\}.
\end{align*}
\end{thm}

\paragraph*{Galton--Watson trees}

Let $\mu$ be a probability distribution on $\Z_+$ with mean at most $1$, such that $\mu_0+\mu_1<1$ (this assumption is made to avoid degenerate cases). A $\mu$-Galton-Watson tree (in short, $\mu$-GW tree) is a random variable $\cT$ on the space of finite trees such that, for any finite tree $T$, $\P \left( \cT = T \right) = \prod\limits_{v \in T} \mu_{k_v(T)}$. $\mu$ is then said to be the \textit{offspring distribution} of $\cT$. In what follows, $\cT_n$ will stand for $\cT$ conditioned to have exactly $n$ vertices (provided that it holds with positive probability).

In the whole paper, we mostly focus on distributions $\mu$ that are critical - that is, with mean $1$ - and in the domain of attraction of a stable law - that is, there exists a slowly varying function $L$ such that, if $X$ is a random variable of law $\mu$, then the following statement holds:

\begin{equation}
\label{eq:L}
\E \left[ X^2 \mathds{1}_{X \leq x} \right] \underset{x \rightarrow +\infty}{\sim} x^{2-\alpha} L(x)+1.
\end{equation}
In what follows, when $\mu$ is a given distribution that is in the domain of attraction of a stable law, $(B_n)_{n \in \Z_+}$ will always denote a sequence verifying 
\begin{equation}
\label{eq:Bn}
\forall n \geq 1, \frac{n L(B_n)}{B_n^\alpha} = \frac{\alpha (\alpha - 1)}{\Gamma \left( 3 - \alpha \right)}.
\end{equation}
where $L$ is a slowly varying function which verifies \eqref{eq:L}. Furthermore,  we define the renormalized contour function of $\cT_n$ as
$$\tilde{C}_t(\cT_n):=\frac{B_n}{n} C_{2nt}(\cT_n)$$
for all $t \in [0,1]$.

The following useful theorem, due to Duquesne \cite{Duq08}, relates the contour function of $\cT_n$ to the process $H^{(\alpha)}$ and is a cornerstone of the paper.

\begin{thm}
\label{thm:duquesne}
Let $\alpha \in (1,2]$, $\mu$ be a critical distribution in the domain of attraction of an $\alpha$-stable law and $(B_n)_{n \in \Z_+}$ a sequence verifying \eqref{eq:Bn}. Then the following convergence holds in distribution in $\mathbb{D} \left( [0,1],\R \right)$:
\begin{align*}
\tilde{C}(\cT_n) \overset{(d)}{\underset{n \rightarrow \infty}{\rightarrow}} H^{(\alpha)}.
\end{align*}
\end{thm}

\subsection{Convergence of the discrete cut processes in the case of contour functions}

We now translate the convergence obtained in Theorem \ref{thm:duquesne} into the convergence of the associated lamination-valued processes. In this subsection, to avoid heavy notations, $\bL_{n, +\infty}$ stands for $\bL(\tilde{C}(\cT_n))$ and $\bL_{n,c}$ for $\bL_c(\tilde{C}(\cT_n))$. Our goal is to prove the following convergence:

\begin{thm}
\label{thm:discreteconvergencestable}
Jointly with the convergence of Theorem~\ref{thm:duquesne}, the following convergence holds in distribution:
\begin{align*}
\left( \bL_{n,c} \right)_{c \in [0,+\infty]} \underset{n \rightarrow \infty}{\overset{(d)}{\rightarrow}} \left( \bL^{(\alpha)}_c \right)_{c \in [0,+\infty]}
\end{align*}
\end{thm}

Note that Theorem \ref{thm:discreteconvergence} already provides a proof of the convergence of these discrete lamination-valued processes, stopped at a finite time $c < \infty$. Hence, we have here to study what happens at $+\infty$. To this end, we rely on the following lemma, which investigates the local structure of $\cT_n$. In what follows, we say that $x \in \cT_n$ is an $a$-node for $a\geq 0$ if the set of its children can be partitioned into two subsets $A_1(x), A_2(x)$ such that $\sum_{u \in A_1(x)} |\theta_u(\cT_n)| \geq a$, $\sum_{u \in A_2(x)} |\theta_u(\cT_n)| \geq a$, where $\theta_u(T)$ denotes the subtree of a tree $T$ rooted in the vertex $u$.

\begin{lem}
\label{lem:treestructure}
Let $\alpha \in (1,2]$ and let $\cT_n$ be a $\mu$-GW tree conditioned to have $n$ vertices, where $\mu$ is in the domain of attraction of an $\alpha$-stable law. Let $f(n) = o(n/B_n)$, where $B_n$ verifies \eqref{eq:Bn}. Then, with high probability, no two different $\epsilon n$-nodes of $\cT_n$ are at distance $\leq f(n)$ from each other.
\end{lem}

Let us immediately see how it implies Theorem \ref{thm:discreteconvergencestable}.

\begin{proof}[Proof of Theorem \ref{thm:discreteconvergencestable}]

We give the main ideas of the proof of this theorem. Assume by Skorokhod theorem that Theorem \ref{thm:duquesne} holds almost surely. By Theorem \ref{thm:discreteconvergence}, the only thing that we have to prove is that, almost surely,
\begin{equation}
\label{eq:equalitylamin}
\bL_\infty^{(\alpha)} = \underset{n \rightarrow \infty}{\lim} \bL_{n,+\infty}.
\end{equation}

First, it is clear that $\bL_\infty^{(\alpha)} \subset \underset{n \rightarrow \infty}{\lim} \bL_{n,+\infty}$. Indeed, by Theorem \ref{thm:discreteconvergence} and Proposition \ref{prop:properties} (ii) applied to $(\bL_c^{(\alpha)})_{c \in [0,\infty]}$, for any $\epsilon > 0$ there exists $c(\epsilon)$ such that, with high probability as $n \rightarrow \infty$, $d_H(\bL_{n,c(\epsilon)},\bL_\infty^{(\alpha)})\leq \epsilon$.

We now have to prove the reverse inclusion, that is, $\underset{n \rightarrow \infty}{\lim} \bL_{n,+\infty} \subset \bL_\infty^{(\alpha)}$. For this, take a chord of $(\underset{n \rightarrow \infty}{\lim} \bL_{n,+\infty}) \backslash \bL_\infty^{(\alpha)}$, of length larger than $\epsilon$. This chord has to be drawn inside a face of $\bL_\infty^{(\alpha)}$. In the discrete setting, this corresponds to the existence of $\epsilon>0$ such that, for $n$ large enough, there exists $x$ in $\cT_n$ which is an $\epsilon n$-node, and one of its ancestors $y$ which is an $\epsilon n$-node as well, such that $d(x,y) = o(n/B_n)$. By Lemma \ref{lem:treestructure}, with high probability this does not happen. The result follows.
\end{proof}

\begin{proof}[Proof of Lemma \ref{lem:treestructure}]
The main idea of the proof is to use  the independence between disjoint subtrees of the Galton-Watson tree $\cT_n$, conditionally to their sizes. Define $J_{\epsilon,n}$ the following event: 

$J_{\epsilon, n}$: "there exist $x,y \in \cT_n$ both $\epsilon n$-nodes, such that $x$ is an ancestor of $y$ and $d(x,y) \leq f(n)$". 
We will prove that, for any $\epsilon>0$, $\P(J_{\epsilon,n}) \rightarrow 0$ as $n \rightarrow 0$. Note that we impose here that $x$ is an ancestor of $y$. In order to get rid of it, remark that, if two different $\epsilon n$-nodes $x$, $y$ in $\cT_n$ are at distance less than $f(n)$, then their common ancestor is still an $\epsilon n$-node at distance less than $f(n)$ from any of them, and $J_{\epsilon, n}$ holds.

In what follows, $X$ and $U$ denote two i.i.d. uniform variables on the set of vertices of $\cT_n$, and $F_a(\cT_n)$ denotes the set of $a$-nodes in $\cT_n$. Finally, $K_x(\cT_n)$ denotes the set of children of $x$ in $\cT_n$. 

Remark that we have the inequality
\begin{align*}
\P\left(J_{\epsilon,n}\right) &\leq \E \Big[ \sum_{x \in \cT_n} \mathds{1}[x \in F_{\epsilon n}(\cT_n)] \quad  \sum_{u \in K_x(\cT_n)} \mathds{1}[|\theta_u(\cT_n)| \geq \epsilon n] \\
& \qquad \qquad \qquad \qquad \qquad \times \mathds{1}[\exists y \in \theta_u(\cT_n) \cap F_{\epsilon n}(\cT_n),d(x,y) \leq f(n)] \Big]\\
&= n^{2} \E \Big[\mathds{1}[X \in F_{\epsilon n}(\cT_n)] \quad \mathds{1}[U \in K_X(\cT_n),|\theta_U(\cT_n)| \geq \epsilon n] \\
& \qquad \qquad \qquad \qquad \qquad \times \mathds{1}[\exists y \in \theta_U(\cT_n) \cap F_{\epsilon n}(\cT_n),d(X,y) \leq f(n)] \Big]\\
&= n^{2} \P\left(X \in F_{\epsilon n}(\cT_n), U \in K_X(\cT_n), |\theta_U(\cT_n)|\geq \epsilon n \right) \times \\
& \qquad \P \left(\exists y \in \theta_U(\cT_n) \cap F_{\epsilon n}(\cT_n),d(X,y) \leq f(n) | X \in F_{\epsilon n}(\cT_n), U \in K_X(\cT_n), |\theta_U(\cT_n)|\geq \epsilon n \right)
\end{align*}
The first probability is bounded from above by $(\epsilon n)^{-2}$. Indeed, in a tree of size $n$, there are at most $1/\epsilon$ $\epsilon n$-nodes, and among the children of any vertex at most $1/\epsilon$ are the root of a subtree of size larger than $\epsilon n$ (note that these considerations are deterministic).
On the other hand, remark that the second probability is bounded from above by
\begin{align*}
\sup_{A \geq \epsilon n} \P\left(\exists y \in \theta_U(\cT_n) \cap F_{\epsilon n}(\cT_n),d(X,y) \leq f(n) \big| X \in F_{\epsilon n}(\cT_n), U \in K_X(\cT_n), |\theta_U(\cT_n)|=A \right)
\end{align*}
which, since we condition $U$ to be a child of $X$, is equal to
\begin{align*}
\sup_{A \geq \epsilon n} \P\left(\exists y \in \theta_U(\cT_n) \cap F_{\epsilon n}(\cT_n),d(U,y) \leq f(n)-1 \big| X \in F_{\epsilon n}(\cT_n), U \in K_X(\cT_n), |\theta_U(\cT_n)|=A \right).
\end{align*}
This way we get rid of one dependency in $X$. Then, by usual independence properties of Galton-Watson trees, we obtain
\begin{align*}
\P\left(J_{\epsilon,n}\right) &\leq \epsilon^{-2} \sup_{A \geq \epsilon n} \P \left( \exists y \in \theta_U(\cT_n) \cap F_{\epsilon n}(\cT_n),d(U,y) \leq f(n)-1 \big|  |\theta_U(\cT_n)|=A \right)\\
&= \epsilon^{-2} \sup_{A \geq \epsilon n} \P \left( \exists y \in F_{\epsilon n}(\cT_A),d(\emptyset,y) \leq f(n)-1 \right)
\end{align*}
where $\cT_A$ is a $\mu$-GW tree with $A$ vertices, and $\emptyset$ denotes its root. But, by Theorem \ref{thm:duquesne}, $\sup_{A \geq \epsilon n} \P ( \exists y \in F_{\epsilon n}(\cT_A),d(\emptyset,y) \leq f(n)-1) \rightarrow 0$ as $n \rightarrow \infty$, using the assumption that $f(n)=o(n/B_n)$. Finally, this leads to:
\begin{align*}
\P\left(J_{\epsilon,n}\right) \underset{n \rightarrow \infty}{\rightarrow} 0.
\end{align*}
The result follows.
\end{proof}

\section{Application to minimal factorizations of the cycle}
\label{sec:permutations}

In this section, we consider an application of Theorem~\ref{thm:discreteconvergencestable} to typical minimal factorizations of the $n$-cycle and prove Theorem~\ref{thm:dicretecv2}. We start by defining the so-called Goulden-Yong bijection, which maps minimal factorizations to trees. Then we conclude the proof of Theorem~\ref{thm:dicretecv2}, by studying new laminations obtained from discrete trees by only marking its vertices.

\subsection*{Notations of Section \ref{sec:permutations}}

\renewcommand{\arraystretch}{1.6}
\begin{center}
\begin{tabular}{|c|c|}
\hline
$F$ & minimal factorization of the cycle\\
\hline
$\cC(F)$ & chord configuration associated to $F$\\
\hline
$T(F)$ & dual tree of $\cC(F)$\\
\hline
$\tilde{T}$ & canonical embedding of a labelled non-plane tree $T$\\
\hline
$t^{(n)}$ & uniform minimal factorization of the $n$-cycle\\
\hline
$\cL_c^{(n)}$ & $\cC(t^{(n)})$ restricted to the first $\lfloor c \sqrt{n} \rfloor$ transpositions of $t^{(n)}$.\\
\hline
$\bL(T)$ & lamination obtained from a tree $T$ by drawing chords only at the level of vertices.\\
\hline
$\bL_u(T)$ & sublamination of $\bL(T)$ corresponding to the first $\lfloor u \rfloor$ vertices of a labelled tree $T$\\
\hline 
\end{tabular}
\end{center}

\subsection{Minimal factorizations}

We start by a study of the class of minimal factorizations: recall that $\mathfrak{M}_n$ is the set of minimal factorizations of the $n$-cycle, namely
\begin{align*}
\mathfrak{M}_n := \left\{ (t_1, \, ..., \, t_{n-1}) \in \mathfrak{T}_n^{n-1}, t_1...t_{n-1}=(1 \, ... \,n) \right\}.
\end{align*}
Recall that, by convention, we apply the transpositions from the left to the right, in the sense that the notation $t_1 t_2$ corresponds to $t_2 \circ t_1$.

Féray and Kortchemski are interested in the properties of a uniform element of $\mathfrak{M}_n$ (see \cite{FK17, FK18}), which we will denote by $t^{(n)} := (t_1^{(n)},...,t_{n-1}^{(n)})$. The starting point of \cite{FK17}, taken from \cite{GY02}, is a geometric coding of $t^{(n)}$ by a random lamination-valued process $(\cL_c^{(n)})_{c \in [0,+\infty]}$: to a transposition $(a b)$ with $a,b \in \llbracket 1,n \rrbracket$, they associate the chord $[ e^{-2i\pi a/n}, e^{-2i \pi b/n} ]$ and, for $c \geq 0$ fixed, they define the random lamination $\cL_c^{(n)}$ as the union of the unit circle and the chords associated to the first $\lfloor c \sqrt{n} \rfloor$ transpositions of $t^{(n)}$: $t_1^{(n)}, \,..., \, t_{\lfloor c \sqrt{n} \rfloor}^{(n)}$ (taking all chords if $c \sqrt{n} \geq n-1$). We furthermore denote by $\cL_\infty^{(n)}$ the union of the unit disk and all the $(n-1)$ chords associated to the factors of $t^{(n)}$. It turns out that these laminations are closely related to the laminations $(\mathbb{L}^{(2)}_c)_{c \in [0,+\infty]}$. 

Féray and Kortchemski prove the following $1$-dimensional convergence, at $c$ fixed:

\begin{thm}[Féray \& Kortchemski \cite{FK17}]
\label{thm:permutationlamination}

Fix $c \in \R_+ \cup \{ +\infty \}$. There exists a lamination $\cL_c$ such that in distribution, for the Hausdorff distance,
\begin{align*}
\cL_c^{(n)} \overset{(d)}{\underset{n \rightarrow \infty}{\rightarrow}} \cL_c.
\end{align*}
\end{thm}

We extend this result and get the functional convergence of the lamination-valued process, which was left open in \cite{FK17}, proving in addition that $(\cL_c)_{0 \leq c \leq \infty} = (\bL_c^{(2)})_{0 \leq c \leq \infty}$ in distribution (Theorem~\ref{thm:dicretecv2}). Let us briefly explain the structure of the proof of Theorem~\ref{thm:dicretecv2}. It is based on two ingredients. The first one is the so-called Goulden-Yong bijection (presented in Section \ref{ssec:GY}), which yields an explicit bijection between $\mathfrak{M}_n$ and a  subset of plane trees with $n$ labelled vertices. The labellings have constraints, namely, the root is the vertex with label $1$ and the labels of a vertex and of its children are sorted in clockwise decreasing order (we call this condition $(C_\Delta)$, see Fig.~\ref{fig:gy}, middle-right for an example). The second one is the introduction of a discrete analogue of the construction given in Section \ref{sec:construction}, where one only marks vertices of the tree instead of all its points. This allows us to obtain in Section \ref{ssec:discrete} an analogue of Theorem~\ref{thm:discreteconvergencestable}, where the lamination-valued processes are obtained from plane trees with a uniform labelling. In order to combine these two ingredients, we lift the constraints on the labellings which appear in the Goulden-Yong bijection by using a shuffling argument based on two operations, presented in Section \ref{ssec:shuffling}.

\subsection{The Goulden-Yong bijection}
\label{ssec:GY}

The Goulden-Yong bijection (see \cite{GY02}) allows us to translate results on random trees into results on minimal factorizations. Let us first explain what this bijection consists in. See Fig.~\ref{fig:gy} for an example on an element of $\kM_9$.

\paragraph*{Step $1$}
Let $F := (t_1, ..., t_{n-1}) \in \mathfrak{M}_n$. For a factor $t_i \coloneqq (a_i,b_i)$, draw a chord between $e^{-2i\pi a_i/n}$ and $e^{-2i\pi b_i/n}$, and give the label $(i+1)$ to this chord. Doing this for the $(n-1)$ transpositions of $F$ gives a compact subset $\cC(F)$ of the disk, made of the unit circle and of $(n-1)$ chords labelled from $2$ to $n$. It appears (see \cite[Theorem $2.2$]{GY02} for further details) that these chords do not intersect - except possibly at their endpoints - and form a tree. Furthermore, the labels of the chords that share an endpoint are sorted in increasing clockwise order around this endpoint (we call this condition  $(C_\Delta)$ as well; see Fig.~\ref{fig:gy}, top-left for an example). Remark in particular that, forgetting about the labels, $\cC(t^{(n)}) = \cL^{(n)}_{\infty}$. 

\paragraph*{Step 2}
Now, draw the dual tree of $\cC(F)$ the following way: put a dual vertex inside each connected component of the complement of $\cC(F)$ in the unit disk, and put a dual edge between two dual vertices if the corresponding connected components share a primal chord as a border. Then, give the label $1$ to the dual vertex whose connected component contains the points $1$ and $e^{-2i\pi/n}$ (this dual vertex exists and is unique by \cite[Proposition $2.3$]{GY02}). The set of dual edges then forms a tree, where each dual edge is given the label of the primal chord that it crosses. Finally, for each dual edge, find the unique path in this dual tree from this edge to the dual vertex $1$ and "slide" the label of the edge to its endpoint further from $1$. This finally provides a plane labelled tree which we denote by $\tilde{T}(F)$. It notably verifies condition $(C_\Delta)$: its root is labelled $1$, and, for any vertex of $\tilde{T}(F)$, its label and the labels of its children are sorted in decreasing clockwise order (see Fig. \ref{fig:gy}, middle-right). Furthermore, forgetting about the planar structure of $\tilde{T}(F)$, we obtain a non plane tree with $n$ labelled vertices, which we denote by $T(F)$.

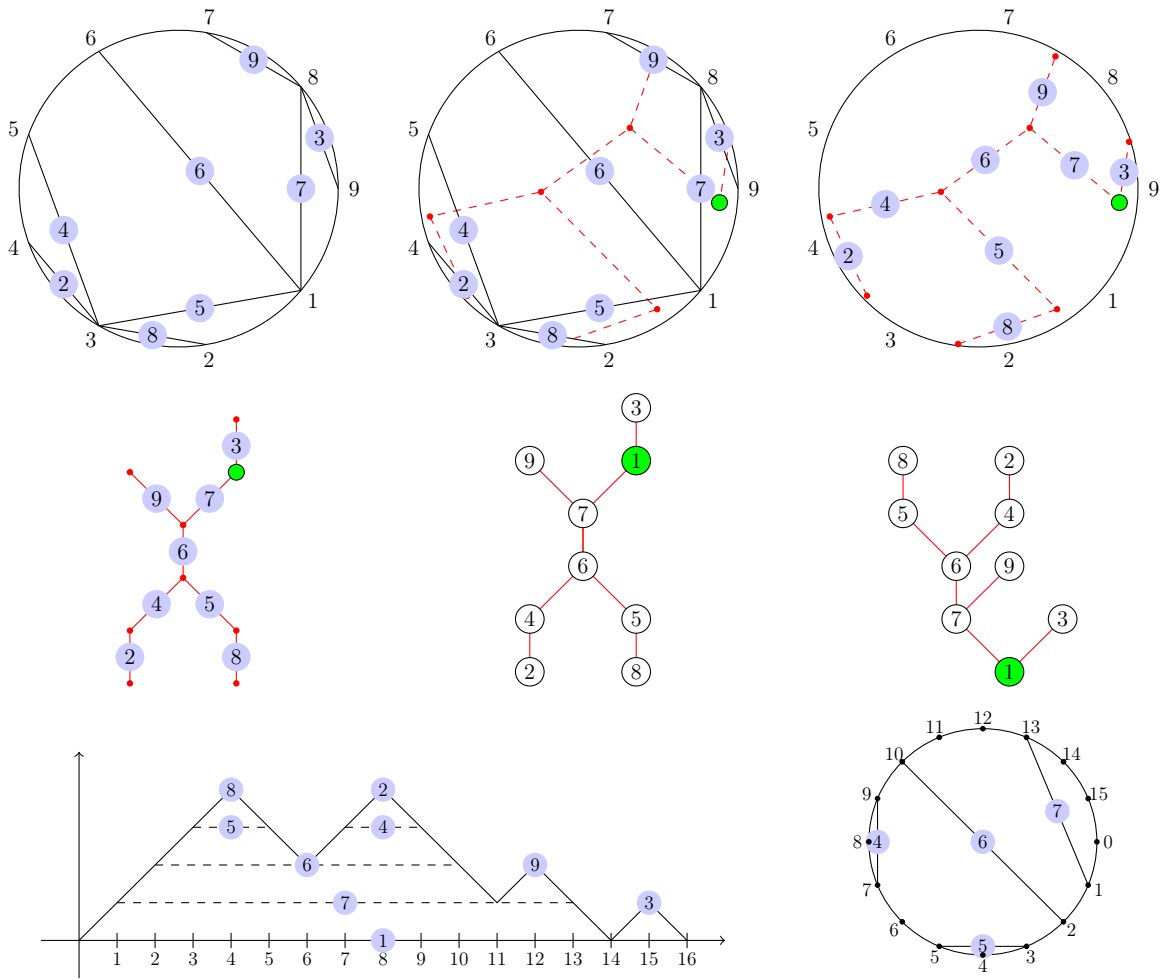
\begin{figure}
\center
\caption{The Goulden-Young mapping, applied to $F := (34)(89)(35)(13)(16)(18)(23)(78) \in \kM_9$. Condition $(C_\Delta)$ is verified for the lamination $\cC(F)$ (top-left) and the tree $\tilde{T}(F)$ (middle-right). At the bottom, the contour function of $\tilde{T}(F)$ and the lamination $\bL (\tilde{T}(F))$, where a labelled chord corresponds to the vertex with the same label. We do not represent the chords of length $0$ in order not to overload the picture.}
\label{fig:gy}
\begin{tabular}{c c c}
\begin{tikzpicture}[scale=0.7, every node/.style={scale=0.7}, rotate=-40]
\draw (0,0) circle (3);
\foreach \i in {1,...,9}
{
\draw[auto=right] ({3.3*cos(-(\i-1)*360/9)},{3.3*sin(-(\i-1)*360/9)}) node{\i};
}
\draw ({3*cos(-2*360/9)},{3*sin(-2*360/9)}) -- ({3*cos(-3*360/9)},{3*sin(-3*360/9)}) node[circle,midway,fill=blue!20, inner sep=2pt]{2};
\draw ({3*cos(2*360/9)},{3*sin(2*360/9)}) -- ({3*cos(1*360/9)},{3*sin(1*360/9)}) node[circle,midway,fill=blue!20, inner sep=2pt]{3};
\draw ({3*cos(-2*360/9)},{3*sin(-2*360/9)}) -- ({3*cos(-4*360/9)},{3*sin(-4*360/9)}) node[circle,midway,fill=blue!20, inner sep=2pt]{4};
\draw ({3*cos(-2*360/9)},{3*sin(-2*360/9)}) -- ({3*cos(0*360/9)},{3*sin(0*360/9)}) node[circle,midway,fill=blue!20, inner sep=2pt]{5};
\draw ({3*cos(-5*360/9)},{3*sin(-5*360/9)}) -- ({3*cos(0*360/9)},{3*sin(0*360/9)}) node[circle,midway,fill=blue!20, inner sep=2pt]{6};
\draw ({3*cos(-7*360/9)},{3*sin(-7*360/9)}) -- ({3*cos(0*360/9)},{3*sin(0*360/9)}) node[circle,midway,fill=blue!20, inner sep=2pt]{7};
\draw ({3*cos(-2*360/9)},{3*sin(-2*360/9)}) -- ({3*cos(-1*360/9)},{3*sin(-1*360/9)}) node[circle,midway,fill=blue!20, inner sep=2pt]{8};
\draw ({3*cos(-6*360/9)},{3*sin(-6*360/9)}) -- ({3*cos(-7*360/9)},{3*sin(-7*360/9)}) node[circle,midway,fill=blue!20, inner sep=2pt]{9};
\end{tikzpicture}
&
\begin{tikzpicture}[scale=0.7, every node/.style={scale=0.7}, rotate=-40]
\draw (0,0) circle (3);
\foreach \i in {1,...,9}
{
\draw[auto=right] ({3.3*cos(-(\i-1)*360/9)},{3.3*sin(-(\i-1)*360/9)}) node{\i};
}
\draw[red,fill=red] (-.5,-.5) circle (.05);
\draw[red,fill=red] (0,1.5) circle (.05);
\draw[red,fill=red] (-.5,2.85) circle (.05);
\draw[red,fill=red] (1.6,-2.5) circle (.05);
\draw[red,fill=red] (1.6,2.5) circle (.05);
\draw[red,fill=red] (2.6,-.8) circle (.05);
\draw[red,fill=red] (-1.8,-2.2) circle (.05);
\draw[red,fill=red] (-.3,-2.9) circle (.05);
\draw[dashed,red] (-.5,-.5) -- (0,1.5) -- (-.5,2.85);
\draw[dashed,red] (1.6,-2.5) -- (2.6,-.8) --(-.5,-.5) -- (-1.8,-2.2) -- (-.3,-2.9);
\draw[dashed,red] (0,1.5) -- (2.2,1.5) -- (1.6,2.5);

\draw[fill=green] (2.2,1.5) circle (.15);
\draw ({3*cos(-2*360/9)},{3*sin(-2*360/9)}) -- ({3*cos(-3*360/9)},{3*sin(-3*360/9)}) node[circle,midway,fill=blue!20, inner sep=2pt]{2};
\draw ({3*cos(2*360/9)},{3*sin(2*360/9)}) -- ({3*cos(1*360/9)},{3*sin(1*360/9)}) node[circle,midway,fill=blue!20, inner sep=2pt]{3};
\draw ({3*cos(-2*360/9)},{3*sin(-2*360/9)}) -- ({3*cos(-4*360/9)},{3*sin(-4*360/9)}) node[circle,midway,fill=blue!20, inner sep=2pt]{4};
\draw ({3*cos(-2*360/9)},{3*sin(-2*360/9)}) -- ({3*cos(0*360/9)},{3*sin(0*360/9)}) node[circle,midway,fill=blue!20, inner sep=2pt]{5};
\draw ({3*cos(-5*360/9)},{3*sin(-5*360/9)}) -- ({3*cos(0*360/9)},{3*sin(0*360/9)}) node[circle,midway,fill=blue!20, inner sep=2pt]{6};
\draw ({3*cos(-7*360/9)},{3*sin(-7*360/9)}) -- ({3*cos(0*360/9)},{3*sin(0*360/9)}) node[circle,midway,fill=blue!20, inner sep=2pt]{7};
\draw ({3*cos(-2*360/9)},{3*sin(-2*360/9)}) -- ({3*cos(-1*360/9)},{3*sin(-1*360/9)}) node[circle,midway,fill=blue!20, inner sep=2pt]{8};
\draw ({3*cos(-6*360/9)},{3*sin(-6*360/9)}) -- ({3*cos(-7*360/9)},{3*sin(-7*360/9)}) node[circle,midway,fill=blue!20, inner sep=2pt]{9};
\end{tikzpicture}
&
\begin{tikzpicture}[scale=0.7, every node/.style={scale=0.7}, rotate=-40]
\draw (0,0) circle (3);
\foreach \i in {1,...,9}
{
\draw[auto=right] ({3.3*cos(-(\i-1)*360/9)},{3.3*sin(-(\i-1)*360/9)}) node{\i};
}
\draw[red,fill=red] (-.5,-.5) circle (.05);
\draw[red,fill=red] (0,1.5) circle (.05);
\draw[red,fill=red] (-.5,2.85) circle (.05);
\draw[red,fill=red] (1.6,-2.5) circle (.05);
\draw[red,fill=red] (1.6,2.5) circle (.05);
\draw[red,fill=red] (2.6,-.8) circle (.05);
\draw[red,fill=red] (-1.8,-2.2) circle (.05);
\draw[red,fill=red] (-.3,-2.9) circle (.05);
\draw[dashed,red] (-.5,-.5) -- (0,1.5) node[circle, black,midway,fill=blue!20, inner sep=2pt]{6} -- (-.5,2.85) node[circle,black,midway,fill=blue!20, inner sep=2pt]{9};
\draw[dashed,red] (1.6,-2.5) -- (2.6,-.8) node[circle, black,midway,fill=blue!20, inner sep=2pt]{8} --(-.5,-.5) node[circle, black,midway,fill=blue!20, inner sep=2pt]{5} -- (-1.8,-2.2) node[circle, black,midway,fill=blue!20, inner sep=2pt]{4} -- (-.3,-2.9) node[circle, black,midway,fill=blue!20, inner sep=2pt]{2};
\draw[dashed,red] (0,1.5) -- (2.2,1.5) node[circle, black,midway,fill=blue!20, inner sep=2pt]{7} -- (1.6,2.5) node[circle, black,midway,fill=blue!20, inner sep=2pt]{3};
\draw[fill=green] (2.2,1.5) circle (.15);
\end{tikzpicture}
\\
\begin{tikzpicture}[scale=0.7, every node/.style={scale=0.7}]
\draw[red] (-1,-2) -- (-1,-1) node[black,circle,midway,fill=blue!20, inner sep=2pt]{2} -- (0,0) node[black,circle,midway,fill=blue!20, inner sep=2pt]{4} -- (1,-1) node[black,circle,midway,fill=blue!20, inner sep=2pt]{5} -- (1,-2) node[black,circle,midway,fill=blue!20, inner sep=2pt]{8};
\draw[red] (0,0) -- (0,1) node[black,circle,midway,fill=blue!20, inner sep=2pt]{6} -- (-1,2) node[black,circle,midway,fill=blue!20, inner sep=2pt]{9};
\draw[red] (0,1) -- (1,2) node[black,circle,midway,fill=blue!20, inner sep=2pt]{7} -- (1,3) node[black,circle,midway,fill=blue!20, inner sep=2pt]{3};
\draw[fill=green] (1,2) circle (.15);
\draw[red,fill=red] (-1,-2) circle (.05);
\draw[red,fill=red] (-1,-1) circle (.05);
\draw[red,fill=red] (0,0) circle (.05);
\draw[red,fill=red] (1,-1) circle (.05);
\draw[red,fill=red] (1,-2) circle (.05);
\draw[red,fill=red] (0,1) circle (.05);
\draw[red,fill=red] (-1,2) circle (.05);
\draw[red,fill=red] (1,3) circle (.05);
\end{tikzpicture}
&
\begin{tikzpicture}[scale=0.7, every node/.style={scale=0.7}]
\draw[red] (-1,-2) -- (-1,-1) -- (0,0) -- (0,1) -- (-1,2);
\draw[red] (1,-2) -- (1,-1) -- (0,0) -- (0,1) -- (1,2) -- (1,3);
\draw (0,0) node[circle,fill=white, draw=black, inner sep=2pt]{6};
\draw (0,1) node[circle,fill=white, draw=black, inner sep=2pt]{7};
\draw (1,2) node[circle,fill=green, draw=black, inner sep=2pt]{1};
\draw (1,3) node[circle,fill=white, draw=black, inner sep=2pt]{3};
\draw (-1,2) node[circle,fill=white, draw=black, inner sep=2pt]{9};
\draw (-1,-1) node[circle,fill=white, draw=black, inner sep=2pt]{4};
\draw (-1,-2) node[circle,fill=white, draw=black, inner sep=2pt]{2};
\draw (1,-1) node[circle,fill=white, draw=black, inner sep=2pt]{5};
\draw (1,-2) node[circle,fill=white, draw=black, inner sep=2pt]{8};
\end{tikzpicture}
&
\begin{tikzpicture}[scale=0.7, every node/.style={scale=0.7}]
\draw[red] (1,1) -- (0,0) -- (-1,1) -- (0,2);
\draw[red] (-1,1) -- (-1,2) -- (0,3) -- (0,4);
\draw[red] (-1,2) -- (-2,3) -- (-2,4);
\draw (-1,2) node[circle,fill=white, draw=black, inner sep=2pt]{6};
\draw (-1,1) node[circle,fill=white, draw=black, inner sep=2pt]{7};
\draw (0,0) node[circle,fill=green, draw=black, inner sep=2pt]{1};
\draw (1,1) node[circle,fill=white, draw=black, inner sep=2pt]{3};
\draw (0,2) node[circle,fill=white, draw=black, inner sep=2pt]{9};
\draw (0,3) node[circle,fill=white, draw=black, inner sep=2pt]{4};
\draw (0,4) node[circle,fill=white, draw=black, inner sep=2pt]{2};
\draw (-2,3) node[circle,fill=white, draw=black, inner sep=2pt]{5};
\draw (-2,4) node[circle,fill=white, draw=black, inner sep=2pt]{8};
\end{tikzpicture}
\\
\multicolumn{2}{c}{
\begin{tikzpicture}[scale=.5,every node/.style={scale=0.6}]
\draw[->] (0,-1) -- (0,5);
\draw[->] (-1,0) -- (17,0);
\draw (0,0) -- (4,4) -- (6,2) -- (8,4) -- (11,1) -- (12,2) -- (14,0) -- (15,1) -- (16,0);
\draw[dashed] (1,1) -- (13,1) node[black,circle,midway,fill=blue!20, inner sep=2pt]{7} (2,2) -- (10,2) node[black,circle,midway,fill=blue!20, inner sep=2pt]{6} (3,3) -- (5,3) node[black,circle,midway,fill=blue!20, inner sep=2pt]{5} (7,3) -- (9,3) node[black,circle,midway,fill=blue!20, inner sep=2pt]{4};
\draw (1,-.2) -- (1,.2);
\draw (2,-.2) -- (2,.2);
\draw (3,-.2) -- (3,.2);
\draw (4,-.2) -- (4,.2);
\draw (5,-.2) -- (5,.2);
\draw (6,-.2) -- (6,.2);
\draw (7,-.2) -- (7,.2);
\draw (8,-.2) -- (8,.2);
\draw (9,-.2) -- (9,.2);
\draw (10,-.2) -- (10,.2);
\draw (11,-.2) -- (11,.2);
\draw (12,-.2) -- (12,.2);
\draw (13,-.2) -- (13,.2);
\draw (14,-.2) -- (14,.2);
\draw (15,-.2) -- (15,.2);
\draw (16,-.2) -- (16,.2);
\draw (1,-.5) node{1};
\draw (2,-.5) node{2};
\draw (3,-.5) node{3};
\draw (4,-.5) node{4};
\draw (5,-.5) node{5};
\draw (6,-.5) node{6};
\draw (7,-.5) node{7};
\draw (8,-.5) node{8};
\draw (9,-.5) node{9};
\draw (10,-.5) node{10};
\draw (11,-.5) node{11};
\draw (12,-.5) node{12};
\draw (13,-.5) node{13};
\draw (14,-.5) node{14};
\draw (15,-.5) node{15};
\draw (16,-.5) node{16};
\draw (4,4) node[black,circle,fill=blue!20, inner sep=2pt]{8};
\draw (8,4) node[black,circle,fill=blue!20, inner sep=2pt]{2};
\draw (12,2) node[black,circle,fill=blue!20, inner sep=2pt]{9};
\draw (15,1) node[black,circle,fill=blue!20, inner sep=2pt]{3};
\draw (8,0) node[black,circle,fill=blue!20, inner sep=2pt]{1};
\end{tikzpicture}
}
&
\begin{tikzpicture}[scale=1.5,every node/.style={scale=0.6}]
\draw (0,0) circle (1);
\draw ({cos(-360/16)},{sin(-360/16)}) -- ({cos(-13*360/16)},{sin(-13*360/16)}) node[black,circle,midway,fill=blue!20, inner sep=2pt]{7} ;
\draw ({cos(-2*360/16)},{sin(-2*360/16)}) -- ({cos(-10*360/16)},{sin(-10*360/16)}) node[black,circle,midway,fill=blue!20, inner sep=2pt]{6};
\draw ({cos(-3*360/16)},{sin(-3*360/16)}) -- ({cos(-5*360/16)},{sin(-5*360/16)}) node[black,circle,midway,fill=blue!20, inner sep=2pt]{5};
\draw ({cos(-7*360/16)},{sin(-7*360/16)}) -- ({cos(-9*360/16)},{sin(-9*360/16)}) node[black,circle,midway,fill=blue!20, inner sep=2pt]{4};
\draw[fill] ({cos(-4*360/16)},{sin(-4*360/16)}) circle (.02);
\draw[fill] ({cos(-8*360/16)},{sin(-8*360/16)}) circle (.02);
\draw[fill] ({cos(-12*360/16)},{sin(-12*360/16)}) circle (.02);
\draw[fill] ({cos(-15*360/16)},{sin(-15*360/16)}) circle (.02);
\draw[fill] ({cos(-16*360/16)},{sin(-16*360/16)}) circle (.02);
\draw[fill] ({cos(-1*360/16)},{sin(-1*360/16)}) circle (.02);
\draw[fill] ({cos(-2*360/16)},{sin(-2*360/16)}) circle (.02);\draw[fill] ({cos(-3*360/16)},{sin(-3*360/16)}) circle (.02);\draw[fill] ({cos(-5*360/16)},{sin(-5*360/16)}) circle (.02);
\draw[fill] ({cos(-6*360/16)},{sin(-6*360/16)}) circle (.02);\draw[fill] ({cos(-7*360/16)},{sin(-7*360/16)}) circle (.02);
\draw[fill] ({cos(-9*360/16)},{sin(-9*360/16)}) circle (.02);
\draw[fill] ({cos(-10*360/16)},{sin(-10*360/16)}) circle (.02);
\draw[fill] ({cos(-11*360/16)},{sin(-11*360/16)}) circle (.02);
\draw[fill] ({cos(-13*360/16)},{sin(-13*360/16)}) circle (.02);
\draw[fill] ({cos(-14*360/16)},{sin(-14*360/16)}) circle (.02);
\foreach \i in {0,...,15}
{
\draw[auto=right] ({1.1*cos(-(\i)*360/16)},{1.1*sin(-(\i)*360/16)}) node{\i};
}
\end{tikzpicture}
\end{tabular}
\end{figure}

Denote by $\kU_n$ the set of non plane trees with $n$ vertices labelled from $1$ to $n$. A complete proof of the following proposition can be found in \cite{GY02}:

\begin{prop}
The Goulden-Yong map $F \rightarrow T(F)$ is a bijection between $\mathfrak{M}_n$ and $\mathfrak{U}_n$.
\end{prop}

As a corollary, $F \rightarrow \tilde{T}(F)$ is a bijection between $\kM_n$ and the set of plane trees with $n$ labelled vertices verifying condition $(C_\Delta)$.

\subsection{A discrete lamination-valued process coded by a discrete tree}
\label{ssec:discrete}

The construction of the process $(\bL_c(f))_{c \in [0,\infty]}$ given in Section \ref{ssec:excursions} is notably valid when $f$ is the (renormalized) contour function of a discrete tree. It consists in throwing points on the skeleton of these trees and then associating a chord to each of these cutpoints. Here, the lamination $\cC(F)$ associated to a minimal factorization $F$ is of a different type, since each of its chords corresponds to a \textit{vertex} of the tree $\tilde{T}(F)$ (namely, the vertex which gets the label of the chord) and not a point thrown uniformly at random on its skeleton. Furthermore, these chords appear at integer times, and not at random times as in Section \ref{sec:construction}. Nevertheless, it happens that laminations of both types can be related to each other, as stated in Proposition \ref{prop:discrcontin} below: in view of future use, we explain how to associate a discrete lamination-valued process to a labelled plane tree, and show that, roughly speaking, this process is close to the one obtained from a Poisson point process under its contour function (in the sense of Section \ref{ssec:excursions}).

Fix a plane tree $T$ with $n$ vertices. For every vertex $u \in T$, denote by $g_{u}$ (resp. $d_{u}$) the first time (resp. the last time) that the contour function of $T$ visits $u$, and let $c_{u}(T)=[e^{-2i \pi g_u/2n},e^{-2i\pi d_u/2n}]$ be the associated chord in $\overline{\mathbb{D}}$. We then set

$$\bL(T)= \mathbb{S}^1 \cup \bigcup_{u \in T} c_{u}(T).$$
where the union is taken over the set of vertices of $T$. Notably, the set of chords of $\bL(T)$ (which may have length $0$) is in bijection with the set of vertices of $T$. Now, we construct a random discrete lamination-valued process $(\bL_s(T))_{s \in [0,\infty]}$ as follows. Let $U_{1}$ be the root of $T$, and let $U_{2}, \ldots,U_{n}$ be a random uniform permutation of the other vertices of $T$. Then, for $s \geq 0$, set

$$\bL_s(T)= \mathbb{S}^1 \cup \bigcup_{i=1}^{ \min(\lfloor s \rfloor,n)} c_{U_{i}}(T),$$
which is roughly speaking the sublamination of $\bL(T)$ obtained by drawing the chords associated to the ``first'' $ \lfloor s \rfloor$ vertices of $T$.

\bigskip

Recall from Section \ref{ssec:excursions}  the notation $(\bL_c(f))_{c \in [0,\infty]}$ for the lamination-valued process obtained from a Poisson point process in the epigraph of a continuous excursion-type function $f$.  We denote by $(\bL_c(C(T)))_{c \in [0,\infty]}$ the lamination-valued process obtained in this way by considering the time-scaled contour function of $T$ on $[0,1]$: $t \rightarrow C_{2|T|t}(T)$. Roughly speaking, $(\bL_s(T))_{s \in [0,\infty]}$ is a discrete version of  $(\bL_c(C(T)))_{c \in [0,\infty]}$, where one only considers cuts on vertices. The following result shows that these two lamination-valued processes  are close in a certain sense, after suitable time-changes, when applied to Galton-Watson trees.

\begin{prop}
\label{prop:discrcontin}

Let $\mu$ be a critical distribution in the domain of attraction of an $\alpha$-stable law, and $\cT_n$ a $\mu$-GW tree with $n$ vertices. Then there exists a coupling between $(\bL_c(\tilde{C}(\cT_n)))_{c \in [0,\infty]}$ and $(\bL_s(\cT_n))_{s \in [0,\infty]}$ such that, with high probability, as $n$ tends to $\infty$:

\begin{align*}
d_{Sk} \left( \left( \bL_c(\tilde{C}(\cT_n)) \right)_{c \in [0,\infty]}, \left( \bL_{cB_n}(\cT_n) \right)_{c \in [0,\infty]} \right) = o(1).
\end{align*}
where $d_{Sk}$ denotes the Skorokhod $J_{1}$ distance on $\D([0,\infty],\bL(\overline{\D}))$ and the $o(1)$ does only depend on $n$ and not on the (random) tree $\cT_n$.
\end{prop}

\begin{proof}
The idea of the proof is to use concentration inequalities to show that, under a suitable coupling, chords appear roughly at the same time and place in both processes. To this end, we study the underlying point processes on the tree $\cT_n$.
For convenience, we use the notation $\bL_{n,c}$ instead of $\bL_c(\tilde{C}(\cT_n))$. Let us first explain the proper coupling between these lamination-valued processes. To this end, define the process $(\tilde{\bL}_{n,c})_{c \in [0,\infty]}$ as follows: remark that, taking the notations of Section \ref{ssec:excursions}, $T(\tilde{C}(\cT_n)) = \cT_n/B_n$. Therefore, by Section \ref{ssec:excursions} again, $\bL_{n,c}$ is obtained from a Poisson point process $\cP_c(\cT_n)$ on $\cT_n$, of intensity $(cB_n/n) d\ell$. For any $c \geq 0$, to each point $u \in \cP_{c}(\cT_n)$, associate the vertex $p(u)$ of $\cT_n$ such that $u$ is in the edge between the vertex $p(u)$ and its parent (if $u$ is a vertex, say that $p(u)=u$). Denote by $\tilde{\cP}_{c}(\cT_n)$ the set of all vertices of $\cT_n$ that are of the form $p(u)$ for some $u \in \cP_{c}(\cT_n)$, and denote finally by $\tilde{\bL}_{n,c}$ the lamination obtained by drawing the chords corresponding to all points of $\tilde{\cP}_{c}(\cT_n)$.
It is clear that 
\begin{equation}
\label{eq:skorlam}
d_{Sk} \left((\tilde{\bL}_{n,c})_{0 \leq c \leq \infty},(\bL_{n,c})_{0 \leq c \leq \infty} \right) \leq \frac{2\pi}{n},
\end{equation}
which tends to $0$ as $n$ grows. Hence we only have to find a proper coupling between the processes $(\tilde{\bL}_{n,c})$ and $(\bL_{cB_n}(\cT_n))$.

To this end, let us precisely compare the times at which points appear in the processes $\cP(\cT_n)$ and $\tilde{\cP}(\cT_n)$.
Since $\cT_n$ has finite length measure $n-1$, almost surely no two points appear at the same time in the process $(\tilde{\cP}_{c}(\cT_n))_{c \geq 0}$. Therefore, this process induces an order on the set of non-root vertices of $\cT_n$, according to the first time that they appear in the process. For $x \geq 0$, denote by $\tau(x)$ the minimum $c \geq 0$ such that $|\tilde{\cP}_{c}(\cT_n)| \geq x$. The order of arrival of the vertices of $\cT_n$ in $(\tilde{\cP}_{c}(\cT_n))_{c \geq 0}$ is uniform among all possible permutations of the non-root vertices, which induces a coupling between $(\tilde{\bL}_{n,c})_{0 \leq c \leq \infty}$ and $(\bL_{cB_n}(\cT_n) )_{0 \leq c \leq \infty}$ such that, for all $c \geq 0$, 
\begin{align*}
\bL_{cB_n}(\cT_n) = \tilde{\bL}_{n, \tau(cB_n)}.
\end{align*} 
Specifically, a chord appears at time $k$ in $(\bL_{u}(\cT_n))_{u \geq 0}$ if it is the chord associated to the $k$-th vertex of $\cT_n$ to get a point of $\cP(\cT_n)$ on the edge between it and its parent.

Now we have to prove that these coupled processes $(\tilde{\bL}_{n,c})$ and $(\bL_{cB_n}(\cT_n))$ are close. We prove in a first time that they are close up to a time $c = f(n) \coloneqq n^{1/2-1/2\alpha}$, and then show that both processes do not change much after this time, as they are already close to their final value.

To prove that they are close up to time $f(n)$, by classical properties of the $J_1$ Skorokhod topology (see \cite[$VI$, Theorem $1.14$]{JS13}), the only thing that we need to show is that the points roughly appear at the same time in both processes. More precisely, uniformly for $c \leq f(n)$, 
\begin{equation}
\label{eq:tau}
|\tau(cB_n) - c| = o(1)
\end{equation} 
with high probability. We prove this result later in this paragraph.
In a second time, assuming that \eqref{eq:tau} holds, we claim that the processes stay close after time $f(n)$. The idea is to use the convergence of the dicrete lamination-valued process to $(\bL_c^{(\alpha)})_{c \in [0,\infty]}$. Assume by Skorokhod theorem that the convergence of Theorem \ref{thm:duquesne} holds almost surely. Then, for $k \geq 1$, let $c_k>0$ such that $d_H(\bL_{c_k}^{(\alpha)}, \bL_\infty^{(\alpha)}) < 1/k$ with probability greater than $1-2^{-k}$. Such a $c_k$ exists by Proposition \ref{prop:properties} (ii). Then, putting together Theorem \ref{thm:discreteconvergencestable}, \eqref{eq:skorlam} and \eqref{eq:tau}, there exists $M_k$ verifying $M_k^{1/2-1/2\alpha} \geq c_k$ such that, for any $n \geq M_k$, $d_H(\bL_{c_k B_n}(\cT_n),\bL_{c_k}^{(\alpha)}) < 1/k$ with probability greater than $1-2^{-k}$. Hence, for any subsequence $(n_k)_{k \geq 1}$ such that, for all $k$, $n_k \geq M_k$, the following holds in almost surely:
\begin{align*}
\bL_\infty^{(\alpha)} = \underset{k \rightarrow \infty}{\lim} \bL_{c_k B_{n_k}}\left(\cT_{n_k}\right) \subset \underset{k \rightarrow \infty}{\lim} \bL_{f(n_k) B_{n_k}}\left(\cT_{n_k}\right).
\end{align*}
The reverse inclusion is clear by Theorem \ref{thm:discreteconvergencestable}. This implies that $d_H(\bL_{f(n) B_n}(\cT_n), \bL_\infty^{(\alpha)})$ converges to $0$ almost surely. Therefore,
\begin{align*}
d_H\left(\bL_{f(n) B_n}(\cT_n), \bL_\infty(\cT_n)\right) \overset{\P}{\rightarrow} 0.
\end{align*}
Since, for any $c \geq 0$, $\bL_{cB_n}(\cT_n)$ and $\tilde{\bL}_{n,c}$ are included in $\bL_\infty(\cT_n)$, this implies Proposition \ref{prop:discrcontin}.

Now we prove \eqref{eq:tau}. First, note that the distribution of the sequence of variables $(\tau(u))_{u \geq 0}$ is independent of $\cT_n$, and only depends on $n$. Thus, the study of these variables boils down to a coupon collector problem, where coupons are vertices of the tree. Set $g(n) \coloneqq n^{1/8 + 5/8\alpha}$, so that $\sqrt{f(n) B_n} \ll g(n) \ll B_n$. In order to prove \eqref{eq:tau}, we show two things:

(i) uniformly in $u \leq f(n) B_n$, with high probability $|\cP_{\tau(u)}(\cT_n)| \leq |\tilde{\cP}_{\tau(u)}(\cT_n)|+g(n)$. In other terms, uniformly in $u \leq f(n) B_n$, we need to throw at most $u+g(n)$ points on the edges of $\cT_n$ before $u$ different vertices appear in the process $(\tilde{\cP}_{c}(\cT_n))_{c \geq 0}$.
 
(ii) uniformly in $k \in \llbracket 0, f(n) B_n + g(n) \rrbracket$, $|\cP_{k/B_n}(\cT_n)|=k+o(B_n)$.

Roughly speaking, if (i) holds, then, since $g(n)=o(B_n)$, the number of points that appear in $\cP(\cT_n)$ on an edge where there was already an other point is negligible compared to $B_n$. Hence, if (ii) also holds, the $\lfloor cB_n \rfloor$-th point appears at time $c+o(1)$, and \eqref{eq:tau} follows.

\emph{Proof of (i)} 
By analogy with the coupon collector problem, let $q_x$ be the number of points that we have to throw on the edges of $\cT_n$ so that $x$ vertices appear in $\tilde{\cP}(\cT_n)$ (this is the number of coupons that we have to buy in order to get $x$ different ones). Remark immediately that $q_x \geq x$ for all $x$. Then, a direct application of Bienaymé-Tchebytchev inequality tells us that $q_{f(n) B_n}$ verifies
\begin{align*}
\P \left( |q_{f(n) B_n} - f(n) B_n| \geq g(n) \right) \underset{n \rightarrow \infty}{\rightarrow} 0,
\end{align*}
using the fact that $g(n)^2 \ll f(n) B_n$. This means that, among the first $f(n) B_n + g(n)$ points that have appeared in $\cP(\cT_n)$, at most $g(n)$ have appeared on an edge where there was already a point. Therefore, at any time $u \leq \tau(f(n) B_n)$, there cannot be more that $g(n)$ such points, which implies (i).

\emph{Proof of (ii)}.
Remark that the variables $X_i \coloneqq |\cP_{((i+1)/B_n) n/(n-1)}(\cT_n)|-|\cP_{(i/B_n)n/(n-1)}(\cT_n)|$ for $i \in \Z_+$ are i.i.d. Poisson variables of parameter $1$. This factor $B_n n /  (n-1)$ comes from the fact that $\cP_c(\cT_n)$ is a Poisson point process of intensity $cB_n/n d\ell$ on $\cT_n$, knowing that $\cT_n$ has total length $\ell(\cT_n)=n-1$.

An application of the so-called local limit theorem (see \cite[Theorem $4.2.1$]{IL71} for a statement and proof) shows that, with high probability,
\begin{align*}
B_n^{-1} \underset{0 \leq k \leq A(n)}{\sup}  \sum_{i \leq k} \left( X_{(i/B_n)n/(n-1)} -1 \right) \underset{n \rightarrow \infty}{\rightarrow} 0,
\end{align*}
where we have set $A(n)=f(n)B_n+g(n)$. Therefore (ii) holds with high probability.
\end{proof}

Let us now explain how to apply Proposition \ref{prop:discrcontin} in our framework: since in $\cC(F)$ each chord corresponds to a vertex of $\tilde{T}(F)$, we use the construction above to exhibit a discrete version of $\bL(C(\tilde{T}(F)))$, in which each chord corresponds to a vertex as well. In addition, we prove that, for $F$ a minimal factorization of the $n$-cycle, this discrete dual lamination $\bL(\tilde{T}(F))$ is close to $\cC(F)$. This statement is not straightforward, since two different minimal factorizations may lead to the same discrete lamination (see Fig. \ref{fig:twolamsonetree} for an example).
For $F$ a minimal factorization, we give $\bL( \tilde{T}(F))$ more structure, by labelling its chords the following way: remember that, in the construction of $\bL(T)$, each chord corresponds to a vertex of $T$. Then, for each vertex $x \in \tilde{T}(F)$, give to the corresponding chord in $\bL(\tilde{T}(F))$ the label of $x$. For $n \geq 2$, $F \in \kM_n$ and $2 \leq j \leq n$, denote by $c(j)$ the chord with label $j$ in $\cC(F)$, and by $c'(j)$ the chord of label $j$ in $\bL( \tilde{T}(F))$. Note that there are $(n-1)$ chords in each lamination, if one does not take into account the chord of length $0$ associated to the root of $\tilde{T}(F)$ in $\bL(\tilde{T}(F))$, and that the leaves of $\tilde{T}(F)$ are coded by chords of length $0$ in $\bL(\tilde{T}(F))$. The next lemma bounds the distance between chords with the same label in $\cC(F)$ and $\bL(\tilde{T}(F))$, by a quantity which only depends on the height of $\tilde{T}(F)$.

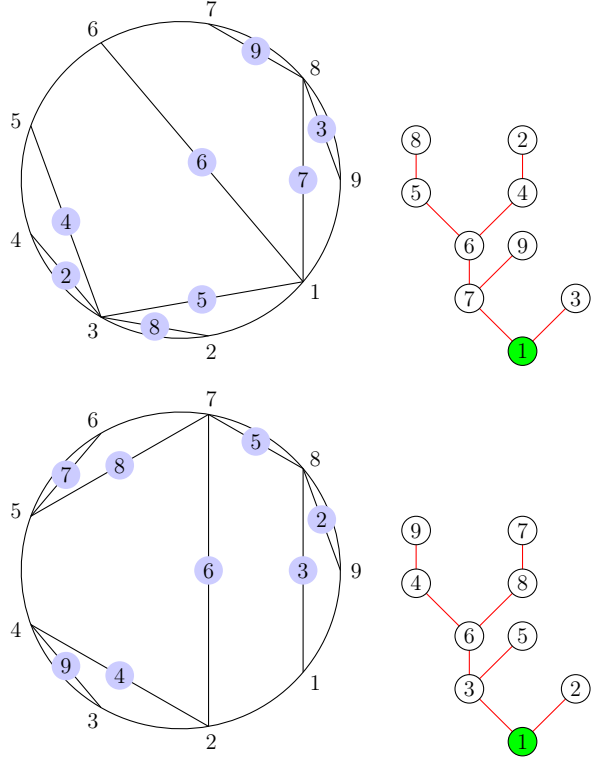
\begin{figure}
\center
\caption{The geometric representation of two different minimal factorizations whose images by the Goulden-Yong map (forgetting about labels) are the same tree.}
\label{fig:twolamsonetree}
\begin{tabular}{c c}
\begin{tikzpicture}[scale=0.7, every node/.style={scale=0.7}, rotate=-40]
\draw (0,0) circle (3);
\foreach \i in {1,...,9}
{
\draw[auto=right] ({3.3*cos(-(\i-1)*360/9)},{3.3*sin(-(\i-1)*360/9)}) node{\i};
}
\draw ({3*cos(-2*360/9)},{3*sin(-2*360/9)}) -- ({3*cos(-3*360/9)},{3*sin(-3*360/9)}) node[circle,midway,fill=blue!20, inner sep=2pt]{2};
\draw ({3*cos(2*360/9)},{3*sin(2*360/9)}) -- ({3*cos(1*360/9)},{3*sin(1*360/9)}) node[circle,midway,fill=blue!20, inner sep=2pt]{3};
\draw ({3*cos(-2*360/9)},{3*sin(-2*360/9)}) -- ({3*cos(-4*360/9)},{3*sin(-4*360/9)}) node[circle,midway,fill=blue!20, inner sep=2pt]{4};
\draw ({3*cos(-2*360/9)},{3*sin(-2*360/9)}) -- ({3*cos(0*360/9)},{3*sin(0*360/9)}) node[circle,midway,fill=blue!20, inner sep=2pt]{5};
\draw ({3*cos(-5*360/9)},{3*sin(-5*360/9)}) -- ({3*cos(0*360/9)},{3*sin(0*360/9)}) node[circle,midway,fill=blue!20, inner sep=2pt]{6};
\draw ({3*cos(-7*360/9)},{3*sin(-7*360/9)}) -- ({3*cos(0*360/9)},{3*sin(0*360/9)}) node[circle,midway,fill=blue!20, inner sep=2pt]{7};
\draw ({3*cos(-2*360/9)},{3*sin(-2*360/9)}) -- ({3*cos(-1*360/9)},{3*sin(-1*360/9)}) node[circle,midway,fill=blue!20, inner sep=2pt]{8};
\draw ({3*cos(-6*360/9)},{3*sin(-6*360/9)}) -- ({3*cos(-7*360/9)},{3*sin(-7*360/9)}) node[circle,midway,fill=blue!20, inner sep=2pt]{9};
\end{tikzpicture} 
&
\begin{tikzpicture}[scale=0.7, every node/.style={scale=0.7}]
\draw[red] (1,1) -- (0,0) -- (-1,1) -- (0,2);
\draw[red] (-1,1) -- (-1,2) -- (0,3) -- (0,4);
\draw[red] (-1,2) -- (-2,3) -- (-2,4);
\draw (-1,2) node[circle,fill=white, draw=black, inner sep=2pt]{6};
\draw (-1,1) node[circle,fill=white, draw=black, inner sep=2pt]{7};
\draw (0,0) node[circle,fill=green, draw=black, inner sep=2pt]{1};
\draw (1,1) node[circle,fill=white, draw=black, inner sep=2pt]{3};
\draw (0,2) node[circle,fill=white, draw=black, inner sep=2pt]{9};
\draw (0,3) node[circle,fill=white, draw=black, inner sep=2pt]{4};
\draw (0,4) node[circle,fill=white, draw=black, inner sep=2pt]{2};
\draw (-2,3) node[circle,fill=white, draw=black, inner sep=2pt]{5};
\draw (-2,4) node[circle,fill=white, draw=black, inner sep=2pt]{8};
\end{tikzpicture}
 \\
\begin{tikzpicture}[scale=0.7, every node/.style={scale=0.7}, rotate=-40]
\draw (0,0) circle (3);
\foreach \i in {1,...,9}
{
\draw[auto=right] ({3.3*cos(-(\i-1)*360/9)},{3.3*sin(-(\i-1)*360/9)}) node{\i};
}
\draw ({3*cos(-4*360/9)},{3*sin(-4*360/9)}) -- ({3*cos(-5*360/9)},{3*sin(-5*360/9)}) node[circle,midway,fill=blue!20, inner sep=2pt]{7};
\draw ({3*cos(2*360/9)},{3*sin(2*360/9)}) -- ({3*cos(1*360/9)},{3*sin(1*360/9)}) node[circle,midway,fill=blue!20, inner sep=2pt]{2};
\draw ({3*cos(-6*360/9)},{3*sin(-6*360/9)}) -- ({3*cos(-4*360/9)},{3*sin(-4*360/9)}) node[circle,midway,fill=blue!20, inner sep=2pt]{8};
\draw ({3*cos(-3*360/9)},{3*sin(-3*360/9)}) -- ({3*cos(-1*360/9)},{3*sin(-1*360/9)}) node[circle,midway,fill=blue!20, inner sep=2pt]{4};
\draw ({3*cos(-6*360/9)},{3*sin(-6*360/9)}) -- ({3*cos(-1*360/9)},{3*sin(-1*360/9)}) node[circle,midway,fill=blue!20, inner sep=2pt]{6};
\draw ({3*cos(-7*360/9)},{3*sin(-7*360/9)}) -- ({3*cos(0*360/9)},{3*sin(0*360/9)}) node[circle,midway,fill=blue!20, inner sep=2pt]{3};
\draw ({3*cos(-2*360/9)},{3*sin(-2*360/9)}) -- ({3*cos(-3*360/9)},{3*sin(-3*360/9)}) node[circle,midway,fill=blue!20, inner sep=2pt]{9};
\draw ({3*cos(-6*360/9)},{3*sin(-6*360/9)}) -- ({3*cos(-7*360/9)},{3*sin(-7*360/9)}) node[circle,midway,fill=blue!20, inner sep=2pt]{5};
\end{tikzpicture} 
&
\begin{tikzpicture}[scale=0.7, every node/.style={scale=0.7}]
\draw[red] (1,1) -- (0,0) -- (-1,1) -- (0,2);
\draw[red] (-1,1) -- (-1,2) -- (0,3) -- (0,4);
\draw[red] (-1,2) -- (-2,3) -- (-2,4);
\draw (-1,2) node[circle,fill=white, draw=black, inner sep=2pt]{6};
\draw (-1,1) node[circle,fill=white, draw=black, inner sep=2pt]{3};
\draw (0,0) node[circle,fill=green, draw=black, inner sep=2pt]{1};
\draw (1,1) node[circle,fill=white, draw=black, inner sep=2pt]{2};
\draw (0,2) node[circle,fill=white, draw=black, inner sep=2pt]{5};
\draw (0,3) node[circle,fill=white, draw=black, inner sep=2pt]{8};
\draw (0,4) node[circle,fill=white, draw=black, inner sep=2pt]{7};
\draw (-2,3) node[circle,fill=white, draw=black, inner sep=2pt]{4};
\draw (-2,4) node[circle,fill=white, draw=black, inner sep=2pt]{9};
\end{tikzpicture}
\end{tabular}
\end{figure}

\begin{lem}
\label{lem:bothlaminations}
As $n \rightarrow \infty$, uniformly for $F \in \kM_n$,
\begin{align*}
\underset{2 \leq j \leq n}{\sup}  d_H (c(j), c'(j)) \leq  2\pi \frac{H(\tilde{T}(F))}{n} + o(1)
\end{align*}
\end{lem}

\begin{proof}
Take $2 \leq j \leq n$ and $F \in \kM_n$.
Let $x(j)$ be the vertex of label $j$ in $\tilde{T}(F)$. $x(j)$ induces a natural partition of the vertices of the tree into three sets: $S'_1(j)$, the set of vertices that are visited by the contour exploration before $x(j)$; $S'_2(j)$ the set of vertices of the subtree rooted in $x(j)$; $S'_3(j)$ the set of vertices that are visited by the contour exploration for the first time after $x(j)$ has been visited for the last time. See an example on Fig. \ref{fig:si}, left. The three connected components of the circle delimited by $c'(j)$ (that is, by $1$ and the endpoints of the chord) have respective arc lengths $2\pi |S'_1(j)|/n +o(1)$, $2\pi |S'_2(j)|/n+o(1)$, $2\pi |S'_3(j)|/n +o(1)$, the $o(1)$ being uniform in $j$ as $n \rightarrow \infty$.

\begin{figure}
\caption{Representation of the two different partitions of the set of vertices of the tree $\tilde{T}(F)$ associated to a minimal factorization $F$, according to the vertex of label $4$. In the middle, $\cC(F)$.}
\label{fig:si}
\center
\begin{tabular}{c c c}
\begin{tikzpicture}[scale=1, every node/.style={scale=0.7}]
\draw[red] (1,1) -- (0,0) -- (-1,1) -- (0,2);
\draw[red] (-1,1) -- (-1,2) -- (0,3) -- (0,4);
\draw[red] (-1,2) -- (-2,3) -- (-2,4);
\draw (-1,2) node[circle,fill=white, draw=black, inner sep=2pt]{6};
\draw (-1,1) node[circle,fill=white, draw=black, inner sep=2pt]{7};
\draw (0,0) node[circle,fill=green, draw=black, inner sep=2pt]{1};
\draw (1,1) node[circle,fill=white, draw=black, inner sep=2pt]{3};
\draw (0,2) node[circle,fill=white, draw=black, inner sep=2pt]{9};
\draw (0,3) node[circle,fill=white, draw=black, inner sep=2pt]{4};
\draw (0,4) node[circle,fill=white, draw=black, inner sep=2pt]{2};
\draw (-2,3) node[circle,fill=white, draw=black, inner sep=2pt]{5};
\draw (-2,4) node[circle,fill=white, draw=black, inner sep=2pt]{8};
\draw[dashed] (1,.5) -- (1.5,1) -- (0,2.5) -- (-.5,2) -- cycle;
\draw[dashed] (-.5,2.5) -- (.5,2.5) -- (.5,4.5) -- (-.5,4.5) -- cycle;
\draw[dashed] (-2,4.5) -- (.5,0) -- (-.5,-1) -- (-3,3.5) -- cycle;
\draw (0,5) node{$S'_2(4)$};
\draw (1,2) node{$S'_3(4)$};
\draw (-2,1) node{$S'_1(4)$};
\end{tikzpicture}
&
\begin{tikzpicture}[scale=0.7, every node/.style={scale=0.7}, rotate=-40]
\draw (0,0) circle (3);
\foreach \i in {1,...,9}
{
\draw[auto=right] ({3.3*cos(-(\i-1)*360/9)},{3.3*sin(-(\i-1)*360/9)}) node{\i};
}
\draw ({3*cos(-2*360/9)},{3*sin(-2*360/9)}) -- ({3*cos(-3*360/9)},{3*sin(-3*360/9)}) node[circle,midway,fill=blue!20, inner sep=2pt]{2};
\draw ({3*cos(2*360/9)},{3*sin(2*360/9)}) -- ({3*cos(1*360/9)},{3*sin(1*360/9)}) node[circle,midway,fill=blue!20, inner sep=2pt]{3};
\draw ({3*cos(-2*360/9)},{3*sin(-2*360/9)}) -- ({3*cos(-4*360/9)},{3*sin(-4*360/9)}) node[circle,midway,fill=blue!20, inner sep=2pt]{4};
\draw ({3*cos(-2*360/9)},{3*sin(-2*360/9)}) -- ({3*cos(0*360/9)},{3*sin(0*360/9)}) node[circle,midway,fill=blue!20, inner sep=2pt]{5};
\draw ({3*cos(-5*360/9)},{3*sin(-5*360/9)}) -- ({3*cos(0*360/9)},{3*sin(0*360/9)}) node[circle,midway,fill=blue!20, inner sep=2pt]{6};
\draw ({3*cos(-7*360/9)},{3*sin(-7*360/9)}) -- ({3*cos(0*360/9)},{3*sin(0*360/9)}) node[circle,midway,fill=blue!20, inner sep=2pt]{7};
\draw ({3*cos(-2*360/9)},{3*sin(-2*360/9)}) -- ({3*cos(-1*360/9)},{3*sin(-1*360/9)}) node[circle,midway,fill=blue!20, inner sep=2pt]{8};
\draw ({3*cos(-6*360/9)},{3*sin(-6*360/9)}) -- ({3*cos(-7*360/9)},{3*sin(-7*360/9)}) node[circle,midway,fill=blue!20, inner sep=2pt]{9};
\end{tikzpicture}
&
\begin{tikzpicture}[scale=1, every node/.style={scale=0.7}]
\draw[red] (1,1) -- (0,0) -- (-1,1) -- (0,2);
\draw[red] (-1,1) -- (-1,2) -- (0,3) -- (0,4);
\draw[red] (-1,2) -- (-2,3) -- (-2,4);
\draw (-1,2) node[circle,fill=white, draw=black, inner sep=2pt]{6};
\draw (-1,1) node[circle,fill=white, draw=black, inner sep=2pt]{7};
\draw (0,0) node[circle,fill=green, draw=black, inner sep=2pt]{1};
\draw (1,1) node[circle,fill=white, draw=black, inner sep=2pt]{3};
\draw (0,2) node[circle,fill=white, draw=black, inner sep=2pt]{9};
\draw (0,3) node[circle,fill=white, draw=black, inner sep=2pt]{4};
\draw (0,4) node[circle,fill=white, draw=black, inner sep=2pt]{2};
\draw (-2,3) node[circle,fill=white, draw=black, inner sep=2pt]{5};
\draw (-2,4) node[circle,fill=white, draw=black, inner sep=2pt]{8};
\draw[dashed] (1,.5) -- (1.5,1) -- (0,2.5) -- (-.5,2) -- cycle;
\draw[dashed] (-.5,2.5) -- (.5,2.5) -- (.5,4.5) -- (-.5,4.5) -- cycle;
\draw[dashed] (-2.5,2.5) -- (-1.5,2.5) -- (-1.5,4.5) -- (-2.5,4.5) -- cycle;
\draw[dashed] (.5,0) -- (-1,2.5) -- (-2,1.5) -- (-.5,-1) -- cycle;
\draw (0,5) node{$S_2(x(4))$};
\draw (1,2) node{$S_3(x(4))$};
\draw (-2,5) node{$S_1(x(4))$};
\draw (-2,.5) node{$E(x(4))$} ;
\end{tikzpicture}
\end{tabular}
\end{figure}
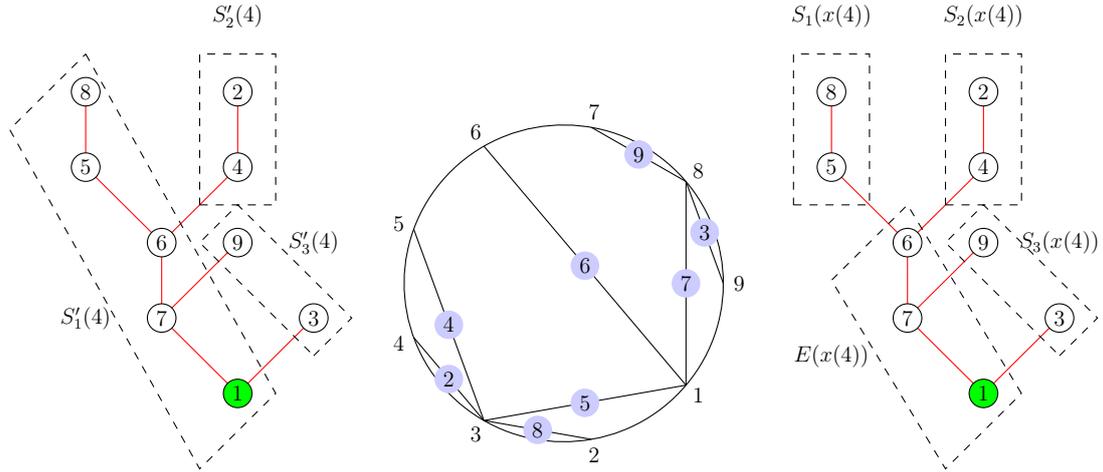

Now, let us focus on the corresponding chord $c(j)$ in $\cC(F)$, and note that it is not given by the position of the chord $c'(j)$. As an example, in Fig. \ref{fig:twolamsonetree}, the vertex with label $6$ is at the same place in both trees, while the chord $c(j)$ is not at the same place in both laminations. Denote by $(a \, b)$ the transposition corresponding to $c(j)$, with $1 \leq a < b \leq n$. Note that the length of $c(j)$ can be directly seen on the unlabelled tree $\tilde{T}(F)$, but the position of its endpoints on the circle depends on the labels of the other vertices, and therefore on its embedding in the disk. We now split the circle into four components, which correspond to a partition the set of vertices of $\tilde{T}(F)$ into four parts: $S_1(j)$ the set of vertices of $\tilde{T}(F)$ whose corresponding chord has its endpoints between $1$ and $a$ ($1$ included); $S_2(j)$ the set of vertices of $\tilde{T}(F)$ whose corresponding chord has its endpoints between $a$ and $b$; $S_3(j)$ the set of vertices of $\tilde{T}(F)$ whose corresponding chord has its endpoints between $b$ and $n$ ($n$ included); $E(x(j))$ the set of ancestors of $x(j)$ ($x(j)$ excluded). One can check that $S'_1(j) = S_1(j) \cup E(x(j))$, $S'_2(j)=S_2(j)$, $S'_3(j)=S_3(j)$. See Fig. \ref{fig:si}, right. Therefore, the distance between $c(j)$ and $c'(j)$ only depends on the labels of the other vertices, and is bounded by $2\pi \frac{|E(x(j))|}{n} + o(1)$. The result follows, since the size of the ancestral line of $x(j)$ is at most $H(\tilde{T}(F))$.
\end{proof}
Lemma \ref{lem:bothlaminations} not only proves that $\cC(F)$ and $\bL(\tilde{T}(F))$ are close, but in addition that they are close chord by chord. This will allow us to bound the distance between the underlying processes of laminations.

The last part of this section is devoted to the study of the set $\kU_n$ of non plane labelled trees. Indeed, the Goulden-Yong mapping allows us to translate results about trees into results about minimal factorizations. For $T$ an element of $\kU_n$ (that is, a non plane tree with $n$ vertices labelled from $1$ to $n$), one can associate exactly one plane rooted tree verifying condition $(C_\Delta)$. We denote by $\tilde{T}$ this canonical embedding of $T$, so that it matches the notations of Section \ref{ssec:GY}. In what follows, $U_n$ is an element of $\kU_n$ taken uniformly at random, and $\tilde{U}_n$ its canonical embedding on the plane.

Our first result concerns the distribution of the tree $\tilde{U}_n$, when one does not care about labels. A proof can be found in \cite[Example $10.2$]{Jan12}.

\begin{lem}
\label{lem:poisson}
Let $n \geq 1$. Then $\tilde{U}_n$, forgetting about the labels, has the law of a $\mu$-GW tree conditioned to have $n$ vertices, where $\mu$ is the Poisson distribution of parameter $1$. 
\end{lem}
This describes the structure of the unlabelled tree $\tilde{T}(t^{(n)})$. We now investigate the constraints that we have on the labelling (condition $(C_\Delta)$). 

\subsection{A shuffling operation on vertices}
\label{ssec:shuffling}

We prove here the important Theorem~\ref{thm:dicretecv2}. To this end, we define an operation on finite trees, which randomly shuffles the labels of its vertices without changing much the overall structure of the tree and the associated lamination. We use it to prove that the lamination $\cL_c^{(n)}$ is close in distribution to $\bL_{c \sqrt{n}}(\cT_n)$ uniformly in $c$, for $\cT_n$ a given Galton-Watson tree (which we will describe) conditioned by its number of vertices. This allows us to use Theorem~\ref{thm:discreteconvergencestable}.

Let us explain the main idea of the shuffling argument. The goal is to lift the constraint on the labels (condition $(C_\Delta)$) without changing much the structure of the tree. To this end, an idea would be to uniformly shuffle the labels of the children of each vertex. But consider a large chord of $\bL_{c \sqrt{n}}(\tilde{U}_n)$, hence corresponding to a vertex $u$ with label $\ell \leq c \sqrt{n}$ in $\tilde{U}_n$ with a large subtree on top of it. If one shuffles the labels uniformly at random among children of its parent, the label $\ell$ could be given to another vertex with a small descendance, resulting in a small chord. The associated lamination would then be far from $\bL_{c \sqrt{n}}(\tilde{U}_n)$. In order to keep the descendance fixed, one could try to shuffle the labels uniformly at random among children of all vertices, also keeping the subtrees on top of them. But then, large subtrees could be swapped at branching points, so that the associated laminations would also be far from each other. The idea is to combine these two operations.

\begin{defi}
\label{def}
Let $T$ be a plane tree with $n$ vertices labelled from $1$ to $n$, rooted at the vertex of label $1$, and let $K \leq n$. We define the shuffled tree $T^{(K)}$ as follows: starting from the root of $T$, we perform one of the following two operations on the vertices of $T$. For consistency, we impose that the operation shall be performed on a vertex before being performed on its children.
\begin{itemize}
\item Operation $1$: for a vertex such that the labels of its children are all $> K$, we uniformly shuffle these labels (without shuffling the corresponding subtrees).
\item Operation $2$: for a vertex such that at least one of its children has a label $\leq K$, we uniformly shuffle these labelled vertices and keep the subtrees on top of each of these children.
\end{itemize}
\end{defi}
See Figure~\ref{fig:shuffling} for an example. Note that this operation induces a transformation of the lamination $\bL(T)$ associated to $T$.

\begin{figure}
\center
\caption{Examples of the shuffling operation. The operation is different in both cases, since in the second case the vertex labelled $9$ has a child with label $4 \leq K$.}
\label{fig:shuffling}
\begin{tabular}{c c}
\begin{tikzpicture}[scale=.7, every node/.style={scale=.7}]
\draw (-2,3) -- (0,1.5) -- (-1,3);
\draw (0,3) -- (0,1.5) -- (1,3);
\draw (0,1.5) -- (2,3);
\draw (-2,3.3) circle (.3) node{7};
\draw (-1,3.3) circle (.3) node{4};
\draw (0,3.3) circle (.3) node{28};
\draw (2,3.3) circle (.3) node{12};
\draw (1,3.3) circle (.3) node{16};
\draw (0,1.2) circle (.3) node{9};
\draw (-2,3.6) -- (-2.3,5) -- (-1.7,5) -- cycle;
\draw (-1,3.6) -- (-1.3,4.5) -- (-.7,4.5) -- cycle;
\draw (0,3.6) -- (-.3,5.5) -- (.3,5.5) -- cycle;
\draw (1,3.6) -- (.7,4) -- (1.3,4) -- cycle;
\draw (2,3.6) -- (1.7,6) -- (2.3,6) -- cycle;
\draw (-2,5.5) node{$T_1$};
\draw (-1,5) node{$T_2$};
\draw (0,6) node{$T_3$};
\draw (1,4.5) node{$T_4$};
\draw (2,6.5) node{$T_5$};
\end{tikzpicture}
&
\begin{tikzpicture}[scale=.7, every node/.style={scale=.7}]
\draw (-2,3) -- (0,1.5) -- (-1,3);
\draw (0,3) -- (0,1.5) -- (1,3);
\draw (0,1.5) -- (2,3);
\draw (-2,3.3) circle (.3) node{28};
\draw (-1,3.3) circle (.3) node{4};
\draw (0,3.3) circle (.3) node{16};
\draw (2,3.3) circle (.3) node{7};
\draw (1,3.3) circle (.3) node{12};
\draw (0,1.2) circle (.3) node{9};
\draw (-2,3.6) -- (-2.3,5) -- (-1.7,5) -- cycle;
\draw (-1,3.6) -- (-1.3,4.5) -- (-.7,4.5) -- cycle;
\draw (0,3.6) -- (-.3,5.5) -- (.3,5.5) -- cycle;
\draw (1,3.6) -- (.7,4) -- (1.3,4) -- cycle;
\draw (2,3.6) -- (1.7,6) -- (2.3,6) -- cycle;
\draw (-2,5.5) node{$T_1$};
\draw (-1,5) node{$T_2$};
\draw (0,6) node{$T_3$};
\draw (1,4.5) node{$T_4$};
\draw (2,6.5) node{$T_5$};
\end{tikzpicture}
\\
\multicolumn{2}{c}{(a) Shuffling of a labelled plane tree when $K=3$: Operation $1$ is performed}\\

\begin{tikzpicture}[scale=.7, every node/.style={scale=.7}]
\draw (-2,3) -- (0,1.5) -- (-1,3);
\draw (0,3) -- (0,1.5) -- (1,3);
\draw (0,1.5) -- (2,3);
\draw (-2,3.3) circle (.3) node{7};
\draw (-1,3.3) circle (.3) node{4};
\draw (0,3.3) circle (.3) node{28};
\draw (2,3.3) circle (.3) node{12};
\draw (1,3.3) circle (.3) node{16};
\draw (0,1.2) circle (.3) node{9};
\draw (-2,3.6) -- (-2.3,5) -- (-1.7,5) -- cycle;
\draw (-1,3.6) -- (-1.3,4.5) -- (-.7,4.5) -- cycle;
\draw (0,3.6) -- (-.3,5.5) -- (.3,5.5) -- cycle;
\draw (1,3.6) -- (.7,4) -- (1.3,4) -- cycle;
\draw (2,3.6) -- (1.7,6) -- (2.3,6) -- cycle;
\draw (-2,5.5) node{$T_1$};
\draw (-1,5) node{$T_2$};
\draw (0,6) node{$T_3$};
\draw (1,4.5) node{$T_4$};
\draw (2,6.5) node{$T_5$};
\end{tikzpicture}
&
\begin{tikzpicture}[scale=.7, every node/.style={scale=.7}]
\draw (-2,3) -- (0,1.5) -- (-1,3);
\draw (0,3) -- (0,1.5) -- (1,3);
\draw (0,1.5) -- (2,3);
\draw (-2,3.3) circle (.3) node{12};
\draw (-1,3.3) circle (.3) node{28};
\draw (0,3.3) circle (.3) node{4};
\draw (2,3.3) circle (.3) node{16};
\draw (1,3.3) circle (.3) node{7};
\draw (0,1.2) circle (.3) node{9};
\draw (1,3.6) -- (.7,5) -- (1.3,5) -- cycle;
\draw (0,3.6) -- (-.3,4.5) -- (.3,4.5) -- cycle;
\draw (-1,3.6) -- (-1.3,5.5) -- (-.7,5.5) -- cycle;
\draw (2,3.6) -- (2.3,4) -- (1.7,4) -- cycle;
\draw (-2,3.6) -- (-1.7,6) -- (-2.3,6) -- cycle;
\draw (1,5.3) node{$T_1$};
\draw (0,4.8) node{$T_2$};
\draw (-1,5.8) node{$T_3$};
\draw (2,4.3) node{$T_4$};
\draw (-2,6.3) node{$T_5$};
\end{tikzpicture}
\\
\multicolumn{2}{c}{(b) Shuffling of the same tree when $K=5$: Operation $2$ is performed}
\end{tabular}
\end{figure}

The main interest of this shuffling is that, for any $K$, $\tilde{U}_n^{(K)}$ has the law of a $Po(1)$-GW tree conditioned to have $n$ vertices, where the root has label $1$ and the other vertices are uniformly labelled from $2$ to $n$. The challenge, in our case, is to find a suitable $K$.

In addition, for $T$ a plane tree with labelled vertices and $u \geq 0$, denote $\bL_u(T)$ the sublamination of $\bL(T)$ made only of the chords that correspond to vertices of label $\leq u$. This extends the notation of Section \ref{ssec:discrete} to a labelled tree (remember that, in Section \ref{ssec:discrete}, we start from an unlabelled tree and label its non-root vertices uniformly at random from $2$ to $|T|$). Notably, for $u \geq |T|$, $\bL_u(T)=\bL(T)$.

\begin{lem}
\label{lem:changedtree}
Let $U_n$ be a uniform element of $\mathfrak{U}_n$. Then, for any sequence $(K_n)_{n \in \Z_+}$ such that $\frac{K_n}{n} \underset{n \rightarrow \infty}{\rightarrow} 0$, as $n \rightarrow \infty$, in probability,
\begin{align*}
d_{Sk}\left( \left(\bL_{u \sqrt{n}}\left( \tilde{U}_n \right)\right)_{0 \leq u \leq K_n/\sqrt{n}}, \left(\bL_{u \sqrt{n}}\left( \tilde{U}_n^{(K_n)} \right) \right)_{0 \leq u \leq K_n/\sqrt{n}} \right) \overset{\P}{\rightarrow} 0,
\end{align*}
where $d_{Sk}$ denotes the Skorokhod distance between these processes.

If, in addition, $\frac{K_n}{\sqrt{n}} \rightarrow \infty$, then
\begin{align*}
d_{Sk}\left( \left(\bL_{u \sqrt{n}}\left( \tilde{U}_n \right)\right)_{0 \leq u \leq \infty}, \left(\bL_{u \sqrt{n}}\left( \tilde{U}_n^{(K_n)} \right) \right)_{0 \leq u \leq \infty} \right) \overset{\P}{\rightarrow} 0.
\end{align*}
\end{lem}

The proof of this lemma, postponed to Section \ref{ssec:proof}, relies on the study of what we call $a$-branching points, for $a \in \Z_+$. For $a>0$ and $T$ a tree, we say that a vertex $u \in T$ is an $a$-branching point if at least two of its children have subtrees of size $\geq a$. Note that this is a particular case of $a$-nodes defined in Section \ref{sec:discrete}. In order to prove Lemma \ref{lem:changedtree}, we show in Section \ref{ssec:proof} that with high probability Operation $2$ is not performed on any $\epsilon n$-branching point for fixed $\epsilon>0$, and then show that it ensures that the lamination-valued processes stay close to each other.

\begin{rk}
We do not have that $d_{Sk} \left( \left(\bL_{u \sqrt{n}}\left( \tilde{U}_n \right)\right)_{0 \leq u \leq \infty}, \left(\bL_{u\sqrt{n}}\left( \tilde{U}_n^{(K_n)} \right) \right)_{0 \leq u \leq \infty} \right) \overset{\P}{\rightarrow} 0$ in all cases, and the second assumption is needed. Indeed, if $K_n=0$, we perform Operation $1$ on all vertices, and the labels of the chords of size $\geq \epsilon$, for $\epsilon$ small enough, might not appear in the process in the same order. On the other hand, the first assumption is needed as well: if $K_n=n$, then Operation $2$ is performed on all vertices, and in particular on $\epsilon n$-branching points. Hence, the large subtrees rooted in children of a given $\epsilon n$-branching point might be interchanged, which leads to a completely different lamination-valued process.
\end{rk}

We are now ready to prove Theorem~\ref{thm:dicretecv2}.

\begin{proof}[Proof of Theorem~\ref{thm:dicretecv2} from Lemma~\ref{lem:changedtree}]
Recall that $t^{(n)}$ denotes a uniform element of $\kM_n$. 
First, we know by Lemma~\ref{lem:poisson} that $\tilde{T}(t^{(n)})$ - forgetting about the labels - is distributed as a $Po(1)$-GW tree. Hence, by Theorem~\ref{thm:duquesne}, $\frac{H(\tilde{T}(t^{(n)}))}{n^{3/4}} \underset{n \rightarrow \infty}{\rightarrow} 0$ in probability. Lemma~\ref{lem:bothlaminations} therefore implies that

\begin{align*}
d_{Sk} \left( \left( \cL^{(n)}_{c} \right)_{0 \leq c \leq \infty}, \left( \bL_{c \sqrt{n}}\left(\tilde{T}(t^{(n)}) \right) \right)_{0 \leq c \leq \infty} \right) \overset{\P}{\rightarrow} 0.
\end{align*} 
On the other hand, let $K_n$ be a sequence of integers such that $\sqrt{n} \ll K_n \ll n$ and recall that only the first $\lfloor c \sqrt{n} \rfloor$ factors of $t^{(n)}$ are represented in $\cL^{(n)}_c$. By Lemma~\ref{lem:changedtree}, as $n \rightarrow \infty$, in probability:
\begin{align*}
d_{Sk} \left(\left( \bL_{c \sqrt{n}} (\tilde{U}_n)\right)_{0 \leq c \leq \infty}, \left( \bL_{c \sqrt{n}} (\tilde{U}_n^{(K_n)}) \right)_{0 \leq c \leq \infty}\right) \overset{\P}{\rightarrow} 0.
\end{align*}

The last step is to prove that $(\bL_{c \sqrt{n}}(\tilde{U}_n^{(K_n)}))_{0 \leq c \leq \infty}$ converges in distribution towards $(\bL_c^{(2)})_{0 \leq c \leq \infty}$. This is a direct consequence of Proposition~\ref{prop:discrcontin} and  Theorem~\ref{thm:discreteconvergencestable}. Indeed, we have already mentioned that $\tilde{U}_n^{(K_n)}$ is distributed as a $Po(1)$-GW tree conditioned to have $n$ vertices labelled from $1$ to $n$, the root having label $1$ and the label of the other vertices being uniformly distributed from $2$ to $n$. This gives the result.
\end{proof}

\subsection{Proof of the technical lemma}
\label{ssec:proof}

This part of the section is devoted to the proof of the technical lemma~\ref{lem:changedtree}, which provides information on $(\bL_u(\tilde{U}_n))_{u \geq 0}$. Before diving into the proof, we present a powerful tool in the study of finite trees, the so-called local limit theorem, which provides good asymptotics on the behaviour of random walks. We provide here two versions of this theorem, the first one concerning general random walks and the second one concerning its application to the size of GW trees (see \cite[Theorem $4.2.1$]{IL71} for details and proofs). 

\begin{thm}[Local limit theorem]
\label{llt}
Let $\alpha \in (1,2]$, $\mu$ a critical distribution on $\Z_+$ in the domain of attraction of an $\alpha$-stable law, and $(B_n)_{n \geq 1}$ verifying \eqref{eq:Bn}.
Let $q_1$ be the density of $Y_1^{(\alpha)}$, where we recall that $Y^{(\alpha)}$ is the $\alpha$-stable Lévy process. Then
\begin{enumerate}
\item[(i)]
Let $(X_i)_{i \geq 1}$ be a sequence of i.i.d. variables taking their values in $\Z_+ \cup \{-1\}$, of law $\mu(\cdot + 1)$. Then
\begin{align*}
\underset{k \in \Z}{\sup} \left| B_n \P \left( \sum\limits_{i=1}^n X_i = k \right) - q_1 \left( \frac{k}{B_n} \right)  \right| \underset{n \rightarrow \infty}{\rightarrow} 0.
\end{align*}
\item[(ii)] Let $\cT$ denote a $\mu$-GW tree. Then, as $n \rightarrow \infty$,
\begin{align*}
\P \left( |\cT| = n \right) \sim n^{-1-1/\alpha} \ell(n)
\end{align*}
\end{enumerate}
where $\ell$ is a slowly varying function depending on $\mu$.
\end{thm}
In particular, an important fact is that $\P(|\cT|=n)^{-1}$ grows more slowly that some polynomial in $n$. Although $q_1$ has no closed expression for $\alpha<2$, (i) can be rewritten when $\mu$ has finite variance $\sigma^2$:
\begin{align*}
\underset{k \in \Z}{\sup} \left| \sqrt{2\pi \sigma^2 n} \P \left( \sum\limits_{i=1}^n X_i = k \right) - \exp \left( -\frac{k^2}{2\sigma^2 n} \right)  \right| \underset{n \rightarrow \infty}{\rightarrow} 0.
\end{align*}

This local limit theorem allows us to understand the structure of the tree $\tilde{U}_n$. We start by setting some notations: for $a \in \Z_+$, denote by $E_a(T)$ the set of $a$-branching points of $T$ and $N_a(T) = \sum_{u \in E_a(T)} k_u(T)$ the number of vertices of $T$ that are \textit{children} of an $a$-branching point. It is straightforward by induction on $|T|$ that, for any $\epsilon >0$ and any finite tree $T$,  $\left| E_{\epsilon |T|}(T) \right| \leq \lfloor \frac{1}{\epsilon} \rfloor$. The following lemma estimates the quantity $N_{\epsilon n}(\tilde{U}_n)$ for fixed $\epsilon > 0$, and may be of independent interest.

\begin{lem}
\label{lem:branchingpoints}
Fix $\epsilon>0$. For $i \geq 1$, let $U_i$ be a uniform element of $\kU_i$, and $\tilde{U}_i$ its canonical embedding in the plane. Then the following two estimates hold:
\begin{enumerate}
\item[(i)] 
There exists a nonincreasing function $C_1$ of $\epsilon$ such that uniformly for $i \geq 2 \epsilon n$,
\begin{align*}
\E \left[ k_\emptyset(\tilde{U}_i) \mathds{1}_{\emptyset \in E_{\epsilon n}(\tilde{U}_i)} \right] \leq C_1(\epsilon) n^{-1/2}.
\end{align*}

\item[(ii)]
Let $f: \Z_+ \rightarrow \R_+$. Let $A_n$ be the event that $H(\tilde{U}_n) \leq f(n) \sqrt{n}$. Then,
\begin{align*}
\E \left[ N_{\epsilon n}(\tilde{U}_n) \mathds{1}_{A_n} \right] \leq C_2(\epsilon) f(n)
\end{align*}
for some constant $C_2(\epsilon)$.
\end{enumerate}
\end{lem}

Note that these results do not in fact depend on the embedding of $U_n$ in the plane. It notably relies on Lemma \ref{lem:poisson} and the local limit theorem. Let us first see how it implies Lemma \ref{lem:changedtree}.

\begin{proof}[Proof of Lemma~\ref{lem:changedtree}]
Let us first explain the main idea of this proof. On one hand, if $K_n=o(n)$, it is unlikely that we perform Operation $2$ on an $\epsilon n$-branching point. This implies that the chords that we discover until $u=K_n$ are close in $\bL(\tilde{U}_n)$ and $\bL(\tilde{U}_n^{(K_n)})$. On the other hand, if $K_n \gg \sqrt{n}$, after having discovered $K_n$ edges, $\bL_{K_n}(\tilde{U}_n)$ is already close to the Brownian triangulation $\bL_\infty^{(2)}$ which is maximum for the inclusion on the set of laminations. Since, by our first point, $\bL_{K_n}(\tilde{U}_n)$ is close to $\bL_{K_n}(\tilde{U}_n^{(K_n)})$, adding the chords labelled from $K_n+1$ to $n$ in any order will not change much the laminations and both stay close to $\bL_\infty^{(2)}$.

We now go into the details. Assume first that $K_n=o(n)$. In order to prove the first part of Lemma~\ref{lem:changedtree}, as usual, we focus on studying the large chords in both laminations. We call \textit{displacement} of a (labelled) chord $c$ of $\bL_u(\tilde{U}_n)$ the Hausdorff distance in the unit disk between $c$ and the chord with the same label in the modified lamination $\bL_u(\tilde{U}_n^{(K_n)})$.

Let us precisely study this notion of displacement: fix $\epsilon>0$ and let $x$ be a vertex of $\tilde{U}_n$ with label $e_x \leq K_n$, such that $|\theta_x(\tilde{U}_n)| > \epsilon n$. The displacement of the chord $c_x$ corresponding to $x$ is due to performing Operation $2$ on some ancestors of $x$. Therefore, the displacement of $c_x$ can be bounded by the sum of the sizes of the subtrees of the children of an ancestor of $x$ that do not contain $x$, the sum being taken over all ancestors of $x$ on which Operation $2$ is performed (that is, one of its children has label $\leq K_n$). See Fig. \ref{fig:Sdelta}, right. Remark that the length of the chords with label $e_x$ is the same in both laminations (indeed, since $x$ has label $\leq K_n$, Operation $2$ is performed on its parent and therefore $|\theta_x(\tilde{U}_n)| = |\theta_x(\tilde{U}_n^{(K_n)})|$). Hence, the displacement of the chord only corresponds to the displacement of its endpoints. 

Let us set some notation: for $x \in \tilde{U}_n$, we denote by $E(x)$ the set of ancestors of $x$ in $\tilde{U}_n$ ($x$ included), and by $\hat{E}(x)$ the set of ancestors of $x$ on which Operation $2$ is performed. The maximum possible displacement of the chord $c_x$ is defined as 
$$MPD(x) := \frac{1}{n} \sum_{\substack{v \in \hat{E}(x) \\ v \neq x}} \sum_{\substack{w \in K_v(\tilde{U}_n) \\ w \notin E(x)}} |\theta_w(\tilde{U}_n)|,$$
where $K_v(\tilde{U}_n)$ denotes the set of children of $v$. Indeed, subtrees which were on the right of the ancestral line of $x$ may be transferred to the left or conversely. This maximum possible displacement corresponds to the sum of the sizes of the green subtrees on Fig.~\ref{fig:Sdelta}, right. We admit the following statement, which we will prove later: for any $\epsilon>0$ fixed, assuming that the convergence of Theorem \ref{thm:duquesne} holds,
\begin{equation}
\label{eq:maxdisp}
\underset{x, e_x \leq K_n, |\theta_x(\tilde{U}_n)|>\epsilon n}{\sup} MPD(x) \overset{\P}{\rightarrow} 0.
\end{equation}

\begin{figure}
\center
\caption{Left: continuous setting (Aldous' CRT). $CMPD_\delta(v)$ is the sum of the sizes of the green subtrees. In red, a subtree of size $>\delta$, which is therefore not counted in $CMPD_\delta(v)$. $m_v(x)$ is the mass of the green tree rooted in $v$. Right: discrete setting (finite tree). Dots represent ancestors of $v$ on which Operation $2$ is performed, so that they have an influence on the displacement of the chord $c_v$ corresponding to $v$: the correponding subtrees are colored in green. With high probability, none of the green subtrees is large. The cross represents an ancestor of $v$ on which Operation $1$ is performed.}
\label{fig:Sdelta}
\begin{tabular}{c c c}
\begin{tikzpicture}[scale=.7]
\draw[green] (0,5) -- (-.8,6.3) -- (-.2,6.3) -- cycle;
\draw[red] (0,4) -- (1,6.5) -- (1.5,6.5) -- cycle;
\draw[green] (0,3) -- (-.8,4.3) -- (-.2,4.3) -- (0,3);
\draw[green] (0,2) -- (.8,3) -- (.2,3) -- (0,2);
\draw[green] (0,1) -- (.3,1.5) -- (.1,1.5) -- (0,2);
\draw (0,0) -- (0,6);
\draw[dashed] (0,6) -- (-.5,7.5) -- (.5,7.5) -- (0,6);
\draw[fill] (.3,6) node{$x$};
\draw[fill] (-.3,2) node{$v$};
\draw[fill] (0,-.03) circle (.03);
\draw (0,8) node{$\theta_x(\cT^{(2)})$};

\draw[fill=green] (0,5) circle (.02);
\draw[fill=red] (0,4) circle (.02);
\draw[fill=green] (0,3) circle (.02);
\draw[fill=green] (0,2) circle (.02);
\draw[fill=green] (0,1) circle (.02);
\end{tikzpicture}
&
\begin{tikzpicture}
\draw[white] (0,0) -- (5,5);
\end{tikzpicture}
&
\begin{tikzpicture}[scale=.7]
\draw (0,2) -- (0,6);
\draw[dashed] (0,6) -- (-.5,7.5) -- (.5,7.5) -- (0,6);
\draw[fill] (.3,6) node{$v$};
\draw[green] (0,5) -- (-.5,5.5) -- (-.8,6.3) -- (-.2,6.3) -- (-.5,5.5);
\draw (0,4) -- (1,5) -- (.5,6.5) -- (1.5,6.5) -- (1,5);
\draw[green] (0,3) -- (-.5,3.5) -- (-.8,4.3) -- (-.2,4.3) -- (-.5,3.5);
\draw[green] (0,3) -- (.2,3.2) -- (.3,3.5) -- (.1,3.5) -- (.2,3.2);
\draw[fill] (0,5) circle (.08);
\draw[fill] (0,3) circle (.08);
\draw (-.05,3.9) -- (.05,4.1) (.05,3.9) -- (-.05,4.1);
\draw (0,8) node{$\theta_v(\tilde{U}_n)$};
\end{tikzpicture}
\end{tabular}
\end{figure}
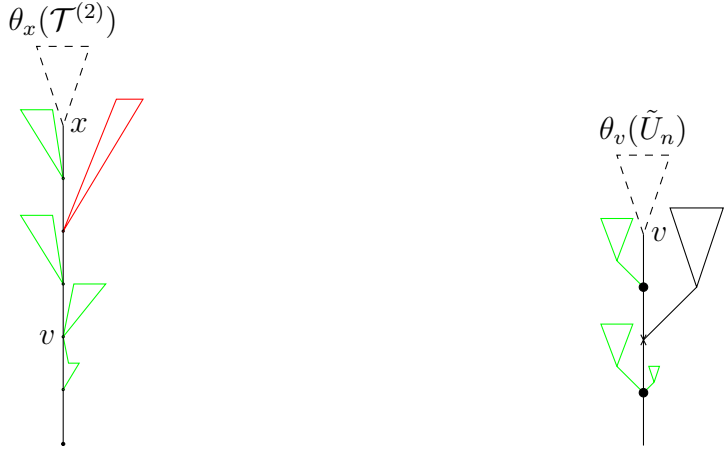

This implies that, uniformly in $u \in [0,K_n]$, with high probability as $n \rightarrow \infty$,
\begin{align*}
d_H\left( \bL_u\left( \tilde{U}_n \right), \bL_u\left( \tilde{U}_n^{(K_n)} \right) \right) \leq 2 \epsilon,
\end{align*}
which proves the first part of Lemma~\ref{lem:changedtree}.

Now, assume in addition that $K_n \gg \sqrt{n}$. Then, by Theorem~\ref{thm:discreteconvergencestable}, jointly with the convergence of Theorem~\ref{thm:duquesne}, with high probability $d_H( \bL_{K_n}(\tilde{U}_n^{(K_n)}), \bL_\infty^{(2)}) \underset{n \rightarrow \infty}{\rightarrow} 0$.
On the other hand, by the first part of Lemma~\ref{lem:changedtree}, $d_H (\bL_{K_n}(\tilde{U}_n), \bL_{K_n}(\tilde{U}_n^{(K_n)}) ) \underset{n \rightarrow \infty}{\rightarrow} 0$ in probability. Since $\bL_\infty^{(2)}$ is a maximum lamination for the inclusion, this implies that for any $\epsilon>0$:
\begin{align*}
\P \left( \exists u \in [K_n,n], d_H\left( \bL_u(\tilde{U}_n), \bL_\infty^{(2)} \right)>\epsilon \right) \rightarrow 0
\end{align*}
as $n \rightarrow \infty$, which proves the second part of Lemma~\ref{lem:changedtree}.
\end{proof}

We now need to prove \eqref{eq:maxdisp}, which states that the supremum of maximum displacements of all $x$ whose label is $\leq K_n$ and such that $|\theta_x(\tilde{U}_n)| \geq \epsilon n$ converges to $0$ in probability.

\begin{proof}[Proof of \eqref{eq:maxdisp}]
We prove in fact a slightly stronger result. Let $0<\delta<\epsilon$. We define the $\delta$-maximum possible displacement of a point $x \in \tilde{U}_n$, denoted by $MPD_\delta(x)$, as 
$$MPD_\delta(x) := \frac{1}{n} \sum_{v \in E^{(\delta)}(x)}\sum_{\substack{w \in K_v(\tilde{U}_n) \\ w \notin E(x)}} |\theta_w(\tilde{U}_n)|,$$
where $E^{(\delta)}(x)$ denotes the set of ancestors of $x$ that are not $\delta n$-branching points. We prove that, as $\delta \downarrow 0$,
\begin{equation}
\label{eq:deltabp}
\underset{\delta \downarrow 0}{\lim} \quad \underset{n \rightarrow \infty}{\limsup} \quad S_\delta(\tilde{U}_n) \quad = \quad 0
\end{equation}
in probability, where $S_\delta(\tilde{U}_n) := \underset{x, e_x \leq K_n, |\theta_x(\tilde{U}_n)|>\epsilon n}{\sup} MPD_\delta(x)$.
Let us first see how this implies \eqref{eq:maxdisp}. We only have to prove that, at $\delta$ fixed, with high probability Operation $2$ is not performed on any $\delta n$-branching point. Indeed, under this event, for all $x$, $\hat{E}(x) \subset E^{(\delta)}(x)$, and $MPD(x) \leq MPD_\delta(x)$.

To prove that, let $p_n$ be the probability that there exists a $\delta n$-branching point in $\tilde{U}_n$ having at least one child with label $\leq K_n$, conditionally given $\tilde{U}_n$. We show that $p_n \rightarrow 0$ with high probability as $n \rightarrow \infty$. First, remark that:
\begin{align*}
p_n = 1 - \frac{\binom{n-N_{\delta n}(\tilde{U}_n)}{K_n}}{\binom{n}{K_n}} \leq 1 - \left( 1-\frac{N_{\delta n}(\tilde{U}_n)}{n-K_n} \right)^{K_n}
\end{align*}
Take $g:\Z_+ \rightarrow \Z_+$ such that $g(n) \underset{n \rightarrow \infty}{\rightarrow} \infty$ and   $g(n) K_n / n \rightarrow 0$, and take $f:\Z_+ \rightarrow \Z_+$ such that $f(n) \underset{n \rightarrow \infty}{\rightarrow} \infty$ and $f(n)/g(n) \underset{n \rightarrow \infty}{\rightarrow} 0$. Then, by Lemma~\ref{lem:branchingpoints} (ii), there exists $C_2(\delta)$ such that, for $n$ large enough, $\E [ N_{\delta n} (\tilde{U}_n ) \mathds{1}_{A_n}] \leq C_2(\delta) f(n)$, where we recall that $A_n \coloneqq \{ H(\tilde{U}_n) \leq f(n)\sqrt{n} \}$. By Markov inequality and since $\P(A_n) \rightarrow 1$ as $n \rightarrow \infty$, we get that $\P( N_{\delta n} (\tilde{U}_n ) \geq g(n) | A_n) \leq 2 C_2(\delta) f(n)/g(n) \underset{n \rightarrow \infty}{\rightarrow} 0$. Hence, with high probability as $n \rightarrow \infty$, 
\begin{align*}
p_n \leq 1 - \left( 1- \frac{g(n)}{n-K_n} \right)^{K_n} \sim \frac{g(n)K_n}{n}
\end{align*}
which tends to $0$ as $n \rightarrow \infty$. Hence, with high probability, Operation $2$ is not performed on any $\delta n$-branching point and \eqref{eq:deltabp} implies \eqref{eq:maxdisp}.

Now we prove \eqref{eq:deltabp}. To this end, let us define the continuous analogue of $S_\delta(\tilde{U}_n)$ on the Brownian tree $\cT^{(2)}$. For a point $x \in \cT^{(2)}$, let $E(x)$ be the set of ancestors of $x$. Recall that $h$ is the uniform probability measure on the set of leaves of $\cT^{(2)}$ and, for $v \in E(x)$, we denote by $m_v(x)$ the $h$-mass of the connected component of $\cT^{(2)} \backslash \{v\}$ which does not contain $x$ nor the root ($m_v(x)$ may be $0$ if $v$ is not a branching point). See Fig. \ref{fig:Sdelta}, left for an example. Then, define $CMPD_\delta(x)$ (for Continuum MPD) as  $CMPD_\delta(x) := \underset{v \in E(x),v \neq x}{\sum} m_v(x) \mathds{1}_{m_v(x) \leq \delta}$ and $S_\delta(\cT^{(2)}) := \underset{x, h(\theta_x(\cT^{(2)})) > \epsilon}{\sup} CMPD_\delta(x)$.
At $\delta$ fixed, it is clear by Theorem~\ref{thm:duquesne} that, in distribution, $S_\delta(\tilde{U}_n) \rightarrow S_\delta(\cT^{(2)})$ as $n \rightarrow \infty$. What is left to prove is that, almost surely, $S_\delta(\cT^{(2)}) \rightarrow 0$ as $\delta \rightarrow 0$.
Assume that it is not the case. Then, there exists $\eta>0$ and a sequence of vertices $v_n \in \cT^{(2)}$ such that $h(\theta_{v_n}(\cT^{(2)})) > \epsilon$ and $CMPD_{1/n} (v_n) \geq \eta$ for all $n$. Since $\cT^{(2)}$ is compact, one can assume without loss of generality that $v_n$ converges to some $v_\infty \in \cT^{(2)}$. Clearly, $h(\theta_{v_\infty}(\cT^{(2)})) > \epsilon$ and $v_\infty$ should verify, for any $\delta > 0$, $CMPD_{\delta}(v_\infty) \geq \eta$, which is not possible. This provides the result. Note that we need the condition that the subtrees rooted in the vertices $(v_n)$ have sizes at least $\epsilon$. This allows to say that $CMPD_{\delta}(v_\infty) \geq \eta$ for any $\delta$, as we avoid the case of a sequence of vertices with small subtrees rooted at them, converging to a point of the skeleton of $\cT^{(2)}$.
\end{proof}

Let us finally prove the estimates of Lemma \ref{lem:branchingpoints}.

\begin{proof}[Proof of Lemma~\ref{lem:branchingpoints}]
Let us start by proving (i). In this proof, we denote by $\mu$ the $Po(1)$ distribution. In particular, $\mu$ is in the domain of attraction of a $2$-stable law. Let us denote by $\cT$ a nonconditioned $\mu$-GW tree and fix $\epsilon>0$. For $n \geq 1$ and $i \geq 2 \epsilon n$, one can write:

\begin{align*}
\E \left[ k_\emptyset(\tilde{U}_i) \mathds{1}_{\emptyset \in E_{\epsilon n}(\tilde{U}_i)} \right] &= \frac{1}{\P \left( |\cT|=i \right)} \sum_{j \in \Z_+} j \P \left( k_\emptyset(\cT)=j \right) \P \left( \left. |\cT|=i, \emptyset \in E_{\epsilon n}  (\cT) \right| k_\emptyset(\cT)=j \right)\\
&\leq \frac{1}{\P \left( |\cT|=i \right)} \sum_{j \in \Z_+} j \P \left( k_\emptyset(\cT)=j \right)  \sum\limits_{1 \leq a < b \leq j} \P \left( \left. |\cT|=i, B_{\epsilon,a,b}  \right| k_\emptyset(\cT)=j \right)
\end{align*}
where $B_{\epsilon,a,b}$ is the event that the $a$th and $b$th subtrees of the root $\emptyset$ have a subtree of size $\geq \epsilon n$. Hence, we can write
\begin{align*}
\E \left[ k_\emptyset(\tilde{U}_i) \mathds{1}_{\emptyset \in E_{\epsilon n}(\tilde{U}_i)} \right] &\leq \frac{\sum_{j\in \Z_+} j \mu_j \binom{j}{2}}{\P \left( |\cT|=i \right)} \sum\limits_{\substack{t_1 \geq \epsilon n \\ t_2 \geq \epsilon n \\ t_1+t_2 \leq i}} \P \left( |\cT|=t_1 \right) \P \left( |\cT|=t_2 \right) \P \left( \left| \mathcal{F}_{j-2} \right| = i-t_1-t_2 \right),
\end{align*}
where $\mathcal{F}_{j-2}$ is a forest of $j-2$ i.i.d. $\mu$-GW trees. Using the local limit theorem~\ref{llt} (ii), we deduce that

\begin{align*}
\E \left[ k_\emptyset(\tilde{U}_i) \mathds{1}_{\emptyset \in E_{\epsilon n} (\tilde{U}_i)} \right] &\leq C(\epsilon) n^{3/2} \sum_{j \in \Z_+} j^3 \mu_j \sum\limits_{\substack{t_1 \geq \epsilon n \\ t_2 \geq \epsilon n \\ t_1+t_2 \leq i}} n^{-3} \P \left( \left| \mathcal{F}_{j-2} \right| = i-t_1-t_2 \right)\\
&\leq C(\epsilon) n^{-3/2} \sum\limits_{\substack{t_1 \geq \epsilon n \\ t_2 \geq \epsilon n \\ t_1+t_2 \leq i}} \sum_{j \geq 2} j^3 \mu_j \frac{j-2}{i-t_1-t_2} \P \left( S_{i-t_1-t_2}=-(j-2) \right)
\end{align*}
for some constant $C(\epsilon)$, by the so-called Kemperman formula (see \cite[6.1]{Pit02}), where $S_k$ denotes the sum of $k$ i.i.d. variables of law $\mu(\cdot+1)$. Therefore, by Theorem~\ref{llt} (i), since $\mu$ has variance $1$,
\begin{align*}
\E \left[ k_\emptyset(\tilde{U}_i) \mathds{1}_{\emptyset \in E_{\epsilon n} (\tilde{U}_i)} \right] &\leq C'(\epsilon) n^{-3/2} \sum\limits_{\substack{t_1 \geq \epsilon n \\ t_2 \geq \epsilon n \\ t_1+t_2 \leq i}} \sum_{j \geq 2} j^4 \mu_j \frac{1}{(i-t_1-t_2)^{3/2}} \\
&\leq C'(\epsilon) n^{-3/2} \sum_{j \in \Z_+} j^4 \mu_j \sum_{q=1}^i \frac{i-q}{q^{3/2}} \leq C_1(\epsilon) n^{-1/2}
\end{align*}
uniformly for $i \geq 2 \epsilon n$, for some constants $C'(\epsilon), C_1(\epsilon)$. Note that we use the fact that $\mu$ has a finite fourth moment. Remark that there exists a nonincreasing choice of $C_1$ since, almost surely, $k_\emptyset(\tilde{U}_i) \mathds{1}_{\emptyset \in E_{\epsilon n}(\tilde{U}_i)} \geq k_\emptyset(\tilde{U}_i) \mathds{1}_{\emptyset \in E_{\epsilon' n}(\tilde{U}_i)}$ for $\epsilon \leq \epsilon'$.

\bigskip

Now we prove Lemma~\ref{lem:branchingpoints} (ii). Remember that we denote by $A_n$ the event $ \{ H(\tilde{U}_n) \leq f(n) \sqrt{n} \}$. Then:

\begin{align*}
\E \left[ N_{\epsilon n}(\tilde{U}_n) \mathds{1}_{A_n} \right] &= \E \left[ \mathds{1}_{A_n} \sum_{u \in \tilde{U}_n} k_u(\tilde{U}_n) \mathds{1}_{u \in E_{\epsilon n}(\tilde{U}_n)} \right] = \E \left[ \mathds{1}_{A_n} \sum_{r=0}^{f(n) \sqrt{n}}\sum\limits_{u \in \tilde{U}_n, |u|=r} k_u(\tilde{U}_n) \mathds{1}_{u \in E_{\epsilon n}(\tilde{U}_n)} \right]\\
&\leq \frac{1}{\P \left( |\cT|=n \right)} \sum_{r=0}^{f(n) \sqrt{n}} \E \left[ \mathds{1}_{|\cT|=n} \sum\limits_{\substack{u \in \cT, |u|=r}} k_u(\cT) \mathds{1}_{u \in E_{\epsilon n}(\cT)} \right]\\
&= \frac{1}{\P \left( |\cT|=n \right)}\sum_{r=0}^{f(n) \sqrt{n}} \sum_{i=0}^n \E \left[  \sum\limits_{\substack{u \in \cT, |u|=r}} k_u(\cT) \mathds{1}_{|Cut_u(\cT)|=n-i} \mathds{1}_{u \in E_{\epsilon n}(\cT), |\theta_u(\cT)|=i} \right].
\end{align*}
where, following \cite{Duq08}, we set $\theta_u(\cT)$ the subtree of $\cT$ rooted at $u$, and $Cut_u(\cT)$ the tree $\cT$ cut at the vertex $u$ ($\theta_u(\cT)$ is erased, along with the edge from $u$ to its parent).

Let us now mention the existence, when $\mu$ is critical with finite variance, of the \textit{local limit} $\cT^*$ of the conditioned $\mu$-GW trees $\left( \cT_n \right)_{n \in \Z_+}$. This limit is defined as the random variable on the set of infinite trees, such that, for any $r \in \Z_+$, 
\begin{align*}
B_r(\cT_n) \underset{n \rightarrow \infty}{\rightarrow} B_r(\cT^*)
\end{align*}
in distribution, where $B_r$ denotes the ball of radius $r$ centered at the root, for the graph distance. Its structure is known: $\cT^*$ is an infinite tree called Kesten's tree (see \cite{Kes86, AD13} for background), made of a unique infinite spine on which i.i.d. nonconditioned $\mu$-GW trees are planted. Notably, asymptotic local properties of large GW trees can be observed on $\cT^*$. In particular, by \cite[Equation $23$]{Duq08}, we get that for any $r \in \llbracket 0,f(n) \sqrt{n} \rrbracket$, any $i \in \llbracket 0,n \rrbracket$,
\begin{align*}
\E \left[ \sum\limits_{\substack{u \in \cT, |u|=r}} k_u(\cT) \mathds{1}_{|Cut_u(\cT)|=n-i}  \mathds{1}_{u \in E_{\epsilon n}(\cT), |\theta_u(\cT)|=i} \right] &= \P \left( \left| Cut_{U^*_r}(\cT^*)\right| = n-i \right) \\
& \qquad \qquad \times \E \left[ k_\emptyset(\cT) \mathds{1}_{\emptyset \in E_{\epsilon n}(\cT)} \mathds{1}_{|\cT|=i} \right]
\end{align*}
where $U^*_r$ is the vertex of the unique infinite branch of $\cT^*$ at height $r$ (see \cite{Kes86} for more background).

Remark that, if $\emptyset \in E_{\epsilon n}(\cT)$ then $|\cT| \geq  2 \epsilon n$. This allows us to write by Lemma~\ref{lem:branchingpoints} (i) and Theorem~\ref{llt} (ii), uniformly for $i \geq 2 \epsilon n$,
\begin{align*}
\frac{1}{\P \left( |\cT|=n \right)}\E \left[ k_\emptyset(\cT) \mathds{1}_{|\cT|=i} \mathds{1}_{\emptyset \in E_{\epsilon n}(\cT)} \right] &\leq \frac{\P \left( |\cT|=i \right)}{\P \left( |\cT|=n \right)} \E \left[ k_\emptyset(\tilde{U}_i) \mathds{1}_{\emptyset \in E_{\epsilon n}(\tilde{U}_i)} \right]\\
&\leq C_2(\epsilon) n^{-1/2}
\end{align*}
for some constant $C_2(\epsilon)$, which leads to
\begin{align*}
\E \left[ N_{\epsilon n}(\tilde{U}_n) \mathds{1}_{A_n} \right] &\leq \sum_{r=0}^{f(n) \sqrt{n}} \sum_{i=0}^n \P \left( \left| Cut_{U^*_r}(\cT^*)\right| = n-i \right) C_2(\epsilon) n^{-1/2}\\
&\leq f(n) \sqrt{n} \, C_2(\epsilon) n^{-1/2}.
\end{align*}
This completes the proof.
\end{proof}

\begin{rk}
The result holds as well for any $\mu$-Galton-Watson tree conditioned to have $n$ vertices, provided that $\mu$ is critical and has a finite fourth moment. 
\end{rk}

\subsection{Convergence of the associated noncrossing partitions}

The last part of this section is devoted to the study of the "last" transpositions of a minimal factorization of the $n$-cycle. More precisely, we investigate here a second way of coding a minimal factorization $t \coloneqq (t_1, \ldots, t_{n-1}) \in \kM_n$, which allows to get a grasp of the behaviour of its "end".

On one hand, for $u \in [ 0,n ]$, let us denote by $C_u(t)$ the union of the circle and all chords corresponding to the first $\lfloor u \rfloor$ transpositions that appear in $t$: $t_1, \ldots, t_{\lfloor u \rfloor}$. This lamination is simply the lamination $\cC(t)$, restricted to the first $\lfloor u \rfloor$ chords drawn in the process.

On the other hand, for $u \in [0,n]$, denote by $P_u(t)$ the union of the circle and the chords $[e^{-2i\pi \ell/n}, e^{-2i\pi \ell'/n}]$, where $\ell$ and $\ell'$ are two consecutive elements of a cycle of the partial product $t_1 \ldots t_{\lfloor u \rfloor}$. The faces of this lamination that have only chords in their boundary are called blocks of the lamination (see Fig. \ref{fig:partition}, right for an example; the hatched part is a block). Notice that $P_u(t)$ is a lamination, and notably the interior of a block is left empty. This new lamination corresponds to the noncrossing partition of $\llbracket 1,n \rrbracket$ induced by the cycles of $t_1 \ldots t_{\lfloor u \rfloor}$.

\begin{figure}[ht!]
\center
\caption{The two laminations $L_5(t)$ and $P_5(t)$, where $t \coloneqq (34)(89)(35)(13)(16)(18)(23)$ is a minimal factorization of the $9$-cycle. The hatched part is a block of $P_5(t)$.}
\label{fig:partition}
\begin{tabular}{c c c}
\begin{tikzpicture}[scale=0.7, every node/.style={scale=0.7}, rotate=-40]
\draw (0,0) circle (3);
\foreach \i in {1,...,9}
{
\draw[auto=right] ({3.3*cos(-(\i-1)*360/9)},{3.3*sin(-(\i-1)*360/9)}) node{\i};
}
\draw ({3*cos(-2*360/9)},{3*sin(-2*360/9)}) -- ({3*cos(-3*360/9)},{3*sin(-3*360/9)});
\draw ({3*cos(2*360/9)},{3*sin(2*360/9)}) -- ({3*cos(1*360/9)},{3*sin(1*360/9)});
\draw ({3*cos(-2*360/9)},{3*sin(-2*360/9)}) -- ({3*cos(-4*360/9)},{3*sin(-4*360/9)});
\draw ({3*cos(-2*360/9)},{3*sin(-2*360/9)}) -- ({3*cos(0*360/9)},{3*sin(0*360/9)});;
\draw ({3*cos(-5*360/9)},{3*sin(-5*360/9)}) -- ({3*cos(0*360/9)},{3*sin(0*360/9)});
\end{tikzpicture}
&
\begin{tikzpicture}
\draw[white] (0,0) -- (2,0);
\end{tikzpicture}
&
\begin{tikzpicture}[scale=0.7, every node/.style={scale=0.7}, rotate=-40]
\draw (0,0) circle (3);
\foreach \i in {1,...,9}
{
\draw[auto=right] ({3.3*cos(-(\i-1)*360/9)},{3.3*sin(-(\i-1)*360/9)}) node{\i};
}
\draw ({3*cos(2*360/9)},{3*sin(2*360/9)}) -- ({3*cos(1*360/9)},{3*sin(1*360/9)});
\draw[pattern=crosshatch dots, pattern color=blue!50] ({3*cos(0*360/9)},{3*sin(0*360/9)}) -- ({3*cos(-2*360/9)},{3*sin(-2*360/9)}) -- ({3*cos(-3*360/9)},{3*sin(-3*360/9)}) -- ({3*cos(-4*360/9)},{3*sin(-4*360/9)}) -- ({3*cos(-5*360/9)},{3*sin(-5*360/9)}) -- cycle;
\end{tikzpicture}
\end{tabular}
\end{figure}
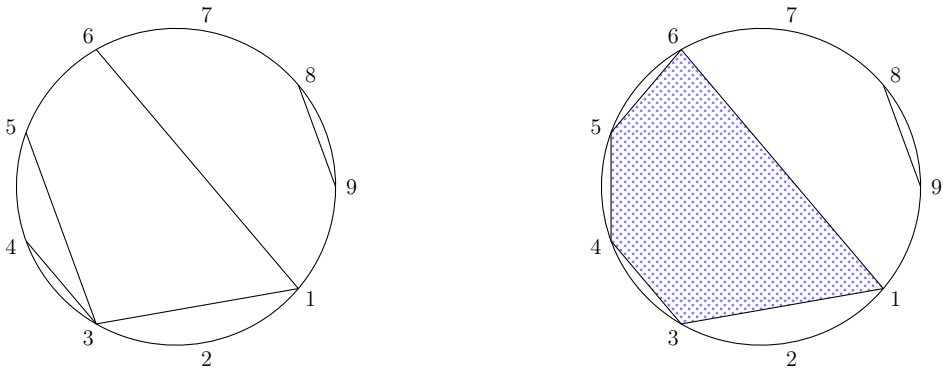

As before, let us denote by $t^{(n)}$ a uniform minimal factorization of the $n$-cycle. We set, for $u \in [0,n]$, $C_u^{(n)} = C_u(t^{(n)})$ and $P_u^{(n)} = P_u(t^{(n)})$. In particular, for $c \geq 0$, $C_{c \sqrt{n}}^{(n)} = \cL_c^{(n)}$. The next theorem answers a question of Féray and Kortchemski \cite{FK17}, who asked for a joint convergence of the lamination-valued processes $(C_{u}^{(n)})_{u \in [0,n]}$ and $(P_{u}^{(n)})_{u \in [0,n]}$. More precisely, the following two convergences hold jointly with Theorem \ref{thm:dicretecv2}, respectively in $\D(\R_+, \bL(\D)^2)$ and $\D(\R_+, \bL(\D))$:

\begin{thm}
\label{thm:jointpartition}
\begin{itemize}
\item[(i)]
The two processes asymptotically have the same behaviour at order $\sqrt{n}$:
\begin{align*}
\left( C^{(n)}_{c \sqrt{n}}, P^{(n)}_{c \sqrt{n}} \right)_{c \geq 0} \overset{(d)}{\rightarrow} \left( \bL_c, \bL_c \right)_{c \geq 0}.
\end{align*}
\item[(ii)]
Jointly with (i), the second process behaves as follows near $n$:
\begin{align*}
\left( P^{(n)}_{n-c \sqrt{n}} \right)_{c \geq 0} \overset{(d)}{\rightarrow} \left(\bL'_c \right)_{c \geq 0},
\end{align*}
where $\left(\bL'_c \right)_{c \geq 0}$ is distributed as $\left(\bL_c \right)_{c \geq 0}$, and is independent of $\left(\bL_c \right)_{c \geq 0}$ conditionally to $\bL_\infty$.
\end{itemize}
\end{thm}

In other terms, roughly speaking, the process $(P_{u}^{(n)})_{u \in [0,n]}$ is increasing at the beginning, when one adds chords which create new blocks in the corresponding partition, and decreasing later when blocks merge, which makes chords disappear.
In addition, these "increasing" and "decreasing" phases are asymptotically independent, conditionally to $\bL_\infty \coloneqq \underset{n \rightarrow \infty}{\lim} C^{(n)}_{n}$. This partition process gives therefore more information on $t^{(n)}$ than $(C_u^{(n)})$, as it explains the joint behaviour of its first and last transpositions. Note that these results were already conjectured in \cite[page $7$]{FK17}. 

We leave the proof of Theorem \ref{thm:jointpartition} (i) to the reader; it is a consequence of Theorem \ref{thm:dicretecv2} and \cite[Lemma $29$]{FK17}, which states that $P^{(n)}_u$ and $C^{(n)}_u$ are close with high probability, jointly for $u \leq \sqrt{n} \log n$.

Let us then focus on the proof of Theorem \ref{thm:jointpartition} (ii). The idea is to investigate the structure of the random tree $T(t^{(n)})$ and deduce a relation, for $u$ large, between the lamination $P_u^{(n)}$ and the set of chords of $C_n^{(n)}$ that have not yet been drawn at time $u$. To this end, denote, for $t \in \kM_n$ and $u \in [0,n]$, $\overset{\leftarrow}{C}_u(t)$ the lamination made only of the chords associated to the \textit{last} $\lfloor u \rfloor$ transpositions that appear in $t$. We set in addition $\overset{\leftarrow}{C}^{(n)}_u = \overset{\leftarrow}{C}_u(t^{(n)})$ this "inverse lamination" drawn from a uniform minimal factorization $t^{(n)}$. Then, the new process $\overset{\leftarrow}{C}^{(n)}$ is closely related to the partition process $P^{(n)}$:

\begin{lem}
\label{lem:closeinverselamination}
The process $\left(\overset{\leftarrow}{C}_{u}^{(n)}\right)_{u \in [0,n]}$ satisfies the following two properties:
\begin{itemize}
\item[(i)] In distribution,
$$\left(\overset{\leftarrow}{C}_{u}^{(n)}\right)_{u \in [0,n]} \overset{(d)}{=} \left(C_{u}^{(n)}\right)_{u \in [0,n]}$$.
\item[(ii)] The following holds in probability, as $n \rightarrow \infty$:
\begin{align*}
d_{Sk} \left( \left( \overset{\leftarrow}{C}^{(n)}_{c \sqrt{n}} \right)_{0 \leq c \leq \log n}, \left( P^{(n)}_{n-c \sqrt{n}} \right)_{0 \leq c \leq \log n} \right) \overset{\P}{\rightarrow} 0.
\end{align*}
\end{itemize}
\end{lem}

Let us immediately see how it implies Theorem \ref{thm:jointpartition} (ii). In what follows, we set $H_n \coloneqq \lfloor \sqrt{n} \log n \rfloor$.

\begin{proof}[Proof of Theorem \ref{thm:jointpartition}]
First, remark that by definition, for all $c \geq 0$, $C_{c \sqrt{n}}^{(n)} = \cL^{(n)}_{c}$. Therefore, by Lemma \ref{lem:closeinverselamination} (i) and (ii), 
\begin{align*}
\left( P^{(n)}_{n-c \sqrt{n}} \right)_{c \geq 0} \overset{(d)}{\rightarrow} (\bL'_c)_{c \geq 0},
\end{align*}
where $(\bL'_c)_{c \geq 0}$ is distributed as $(\bL_c)_{c \geq 0}$ (recall that $(\bL_c)_{c \geq 0}$ is the limit of the process $(\cL_c^{(n)})_{c \geq 0}$ constructed from the first transpositions in $t^{(n)}$). The only thing that we have to prove is that, conditionally to $\bL_\infty$, the processes $(\bL'_c)_{c \geq 0}$ and $(\bL_c)_{c \geq 0}$ are independent. By Lemma \ref{lem:closeinverselamination} (ii), it is enough to prove that $(C^{(n)}_{c \sqrt{n}})_{c \geq 0}$ and $(\overset{\leftarrow}{C}^{(n)}_{c \sqrt{n}})_{c \geq 0}$ are, in some sense, asymptotically independent. To this end, remember that the (non plane) tree $T(t^{(n)})$ is uniform among rooted trees of size $n$ with non-root vertices labelled from $2$ to $n$. Therefore, conditionally to the structure of this tree (that is, forgetting about labels), the sets $D_{H_n}$ (resp. $A_{H_n}$) of vertices labelled between $2$ and $H_n+1$ (resp. between $n + 1 - H_n $ and $n$) are two uniform sets of $H_n$ non-root vertices of $T(t^{(n)})$. Furthermore, $D_{H_n}$ and $A_{H_n}$ are independent conditionally to being disjoint. Notice finally that, conditionally to $(D_{H_n}, A_{H_n})$, the processes $(C_u^{(n)})_{u \leq H_n}$ and $(\overset{\leftarrow}{C}^{(n)}_u)_{u \leq H_n}$ are distributed as follows: order the vertices of $D_{H_n}$ (resp. $A_{H_n}$) uniformly at random, and draw the associated chords in this order.

We will prove that, roughly speaking, as $n \rightarrow \infty$, asymptotically we can get rid of this conditioning to be disjoint. In other words, there is only a small difference between two independent sets of $H_n$ vertices of the tree, and two such sets conditioned to be disjoint, in the sense that they give birth to close lamination-valued processes. To prove this, let us provide a way of sampling $D_{H_n}$ and $A_{H_n}$: first sample $D_{H_n}$, a $H_n$-tuple of non-root vertices in the tree, and then sample $A$ a $H_n$-tuple of non-root vertices, independent of $D_{H_n}$. Then remove from $A$ the vertices of $A$ that are in $D_{H_n}$, and resample $B$, a $|A \cap D_{H_n}|$-tuple of non-root vertices of the tree, independent of $D_{H_n}$ and $A$, conditioned to contain no vertex of $A \cup D_{H_n}$. Then, set $A_{H_n}=(A \backslash D_{H_n}) \cup B$. It is clear that $(D_{H_n}, A_{H_n})$ is distributed as a couple of uniform sets of $H_n$ vertices of the tree, conditioned to be disjoint. 

Now, we show that with high probability no point of $B \cup (A \cap D_{H_n})$ codes a large chord in the unit disk. This will prove that there is asymptotically no difference between the sets of chords coded respectively by the vertices of $A$ and the vertices of $A_{H_n}$. Roughly speaking, this will imply that only points of $A \backslash D_{H_n}$ and $D_{H_n} \backslash A$ matter, and thus that the lamination-valued processes corresponding to $D_{H_n}$ and $A_{H_n}$ (recall that it consists in ordering uniformly at random the vertices of the set, and drawing the associated chords in this order) are asymptotically independent. To prove this, remark that, by Markov inequality, 
$$\P(|A \cap D_{H_n}| \geq (\log n)^3) \leq n (H_n/n)^2 (\log n)^{-3} \leq (\log n)^{-1}.$$ 
Thus, with high probability $|B \cup (A \cap D_{H_n})|\leq (\log n)^3$. Now, fix $\epsilon>0$ and remark that for $h \leq H(T(t^{(n)}))$, at most $1/\epsilon$ points in the tree at height $h$ are the root of a subtree of size $\geq \epsilon n$. This implies that, with high probability, by Theorem \ref{thm:duquesne}, there are less than $\sqrt{n} \log n$ such points in the whole tree. Hence, the intersection of the set of such points with $B \cup (A \cap D_{H_n})$ is empty with high probability. The result follows.
\end{proof}

We finish by proving the technical lemma \ref{lem:closeinverselamination}.

\begin{proof}[Proof of Lemma \ref{lem:closeinverselamination} (i)]
The idea is again to study the non plane tree $T(t^{(n)})$. Remember that this tree has the law of a uniform element of the set $\kU_n$, that is, the set of non plane rooted trees whose non-root vertices are labelled from $2$ to $n$. Define $g: \kU_n \rightarrow \kU_n$ the involution which consists in changing the label $e_x$ of each non-root vertex $x$ in a tree $T \in \kU_n$ to $n+2-e_x$. Then, for $F \in \kM_n$, the tree $g(T(F))$ is the image of a factorization $\overset{\leftarrow}{F}$ by the Goulden-Yong bijection, which verifies
\begin{align*}
\left( C_u \left( \overset{\leftarrow}{F} \right) \right)_{u \in [0,n]} = \left( \overset{\leftarrow}{C}_u(F) \right)_{u \in [0,n]}.
\end{align*}
Since $t^{(n)}$ is uniform on $\kM_n$, $\overset{\leftarrow}{t^{(n)}}$ is uniform on $\kM_n$ as well and Lemma \ref{lem:closeinverselamination} (i) follows.
\end{proof}

In order to prove Lemma \ref{lem:closeinverselamination} (ii), we focus as usual on large chords of these lamination-valued processes. Fixing $\epsilon>0$, we shall check that, jointly for all $u \leq H_n \coloneqq \sqrt{n} \log n$, for any chord of $\overset{\leftarrow}{C}_u^{(n)}$ of length $> \epsilon$ there is always a chord of $P_u^{(n)}$ close to it, and conversely any chord of $P_u^{(n)}$ of length $> \epsilon$ can be approximated by a large chord of $\overset{\leftarrow}{C}_u^{(n)}$. 

\bigskip

Let $A_{H_n}$ be the set of vertices in $T(t^{(n)})$ with labels between $n-H_n$ and $n$. The proof of Lemma \ref{lem:closeinverselamination} (ii) is based on the following result, which provides useful properties of the set of vertices $A_{H_n}$:

\begin{lem}
\label{lem:structure}
The points of the set $A_{H_n}$ are well spread in the random tree $T(t^{(n)})$, in the sense that, for any $\epsilon>0$ fixed, the following two properties hold with high probability as $n \rightarrow \infty$:
\begin{itemize}
\item[(i)] There is no ancestral line of size $3$ in the tree (that is, a vertex, its parent and its grandparent) made only of points of $A_{H_n}$.

\item[(ii)] No point of $A_{H_n}$ is an $\epsilon n$-node, nor the child of an $\epsilon n$-node.
\end{itemize}
\end{lem}

\begin{proof}[Proof of Lemma \ref{lem:structure}]
In order to get (i), remark that the probablity that a vertex, its parent and its grandparent all are in $A_{H_n}$ is of order $(H_n/n)^3 = (\log n)^3 n^{-3/2}$. Since such a triple of vertices is uniquely characterized by the first one, there are at most $n$ of them, and the probability of seeing an ancestral line of size $3$ made only of elements of $A_{H_n}$ is less than $(\log n)^3 n^{-1/2}$.

On the other hand, (ii) is a consequence of the small number of children of the $\epsilon n$-nodes. First, since $T(t^{(n)})$ converges in distribution to the Brownian CRT, then with high probability all $\epsilon n$-nodes are $\epsilon n/2$-branching points. Now, by Lemma \ref{lem:branchingpoints} (ii) (taking $f(n) \coloneqq \log n$) and Theorem \ref{thm:duquesne}, with high probability there are less than $C(\epsilon) \log n$ children of $\epsilon n/2$-branching points in $T(t^{(n)})$, for some constant $C(\epsilon)$ depending only on $\epsilon$. Thus, on this event, since a branching point has at least one child, there are at most $2 C(\epsilon) \log n$ vertices that are either an $\epsilon n$-node or the child of one of them. The result follows: with high probability none of these points belongs to $A_{H_n}$, since $|A_{H_n}| = \lfloor \sqrt{n} \log n \rfloor$.
\end{proof}

Let us now see how this structural result implies Lemma \ref{lem:closeinverselamination} (ii):

\begin{proof}[Proof of Lemma \ref{lem:closeinverselamination} (ii)]
In the whole proof, $\epsilon>0$ and $u \leq H_n$ are fixed, and we investigate the two chord configurations $\overset{\leftarrow}{C}^{(n)}_u$ and $P_{n-u}^{(n)}$. Specifically, we prove that any large chord of $\overset{\leftarrow}{C}^{(n)}_u$ is close to a large chord of $P_{n-u}^{(n)}$, and conversely; furthermore, this holds uniformly in $u \leq H_n$.

First, let $c$ be a chord of length $ \ell(c)> \epsilon$ in $\overset{\leftarrow}{C}_u^{(n)}$. Let $e(c)$ be the location of the associated transposition in $t^{(n)}$, so that $e(c) \geq n-H_n$, and let $x(c)$ be the vertex of $T(t^{(n)})$ labelled $e(c)$. It is to note that, by Theorem \ref{thm:duquesne}, with high probability the root and its children are not coded by chords of length $>\epsilon$, and thus $x(c)$ has height $\geq 3$.

 Then, by Lemma \ref{lem:structure} (i), with high probability the parent or the grandparent of $x(c)$ has a label $< n-H_n$. We claim that the chord associated to this ancestor is close to $c$.

If the parent $y(c)$ of $x(c)$ has such a small label, denote by $\tilde{c}$ the chord associated to it. By assumption, $\tilde{c} \subset C^{(n)}_{n-u}$. By construction of the tree $T(t^{(n)})$, if $\tilde{c}$ has length $\geq  2\ell(c)$ or $\leq \epsilon/2$, then necessarily either $x(c)$ or $y(c)$ is an $\epsilon n/2$-node. However, with high probability this does not happen, by Lemma \ref{lem:structure} (ii). Thus, the chord $\tilde{c}$ is in $C^{(n)}_{n-u}$ and is at distance $\leq \epsilon$ from $c$. 

On the other hand, if $y(c)$ itself belongs to $A_{H_n}$, then with high probability the grandparent $z(c)$ of $x(c)$ is not in $A_{H_n}$. Furthermore, by Lemma \ref{lem:structure} (ii), $y(c)$ is not an $\epsilon n$-node nor the child of an $\epsilon n$-node. Thus, as before, the chord $\tilde{c}$ associated to $z(c)$ is necessarily at distance less than $\epsilon$ from $c$.

In both cases, this chord $\tilde{c}$ associated to $y(c)$ or $z(c)$ is in $C^{(n)}_{n-u}$. Therefore, it lies inside a block $B$ of $P_{n-u}^{(n)}$ (see Fig. \ref{fig:partitionblock}, left for an example). Let us prove that one of the chords in the boundary of $B$ is at distance less than $\epsilon$ from $c$. To this end, denote by $(a b)$ the transposition associated to $c$, where $1 \leq a < b \leq n$. Since $C^{(n)}_n$ satisfies the previously mentioned condition $(C_\Delta)$, its chords are sorted in decreasing labelling order around each point of the form $e^{-2i\pi x/n}$ for $1 \leq x \leq n$. Then there is no chord in $C^{(n)}_{n-u}$ connecting $e^{-2i\pi a/n}$ to $e^{-2i\pi x/n}$ where $x \notin \llbracket a,b \rrbracket$, nor connecting $e^{-2i\pi b/n}$ to $e^{-2i\pi y/n}$ where $y \in \llbracket a,b \rrbracket$. Thus, since the chord $\tilde{c}$ is inside the block $B$, the boundary of $B$ contains a chord inbetween $c$ and $\tilde{c}$, which is therefore at distance less than $\epsilon$ from $c$.

In conclusion, any large chord of $\overset{\leftarrow}{C}_u^{(n)}$ is close to a chord of $P_{n-u}^{(n)}$, uniformly for $u \leq H_n$.

\bigskip

We use the same trick to prove the converse. Specifically, take $c'$ a chord in $P_{n-u}^{(n)}$ of length greater than $\epsilon$, and define $1 \leq a<b \leq n$ such that $c'=[e^{-2i\pi a/n}, e^{-2i\pi b/n}]$.
Now, let $\mathbb{S}^{a,b}$ (resp. $\overline{\mathbb{S}}_{a,b}$) be the set of points of the form $e^{-2i\pi x/n}$ for $a < x < b$ (resp. $a \leq x \leq b$), and assume in a first time that $a$ and $b$ are not connected to any point of $\mathbb{S}_{a,b}$. In other words, the block of $P_{n-u}^{(n)}$ whose boundary contains $c'$ is on the side of $c'$ which contains $1$ (see an example on Fig. \ref{fig:partitionblock}, right). Consider now the face $F_a$ of $C_n^{(n)}$ whose boundary contains the arc $(\wideparen{e^{-2i\pi a/n}, e^{-2i\pi (a+1)/n}})$. It appears (see \cite[Proposition $2.3$]{GY02}) that the rest of its boundary is only made of chords. Since the labels of the chords in $C_n^{(n)}$ are decreasing in clockwise order around each vertex of this face, it is a simple matter to check that the boundary of $F_a$ contains $e^{-2i \pi b/n}$, and that this boundary is made exclusively of chords of $C_{n-u}^{(n)}$ between $e^{-2i\pi b/n}$ and $e^{-2i\pi a/n}$ (clockwise), and of chords of $\overset{\leftarrow}{C}_u^{(n)}$ between $e^{-2i\pi (a+1)/n}$ and $e^{-2i\pi b/n}$ (clockwise, red chords on Fig. \ref{fig:partitionblock},right). Let $\tilde{c}'$ be the largest of these chords of $\overset{\leftarrow}{C}_u^{(n)}$. If $\tilde{c}'$ has length less than $\ell(c)-\epsilon/2$, then the associated vertex in $T(t^{(n)})$ is necessarily the child of an $\epsilon n/2$-node, which with high probability does not happen by Lemma \ref{lem:structure} (ii). Therefore $d_H(\tilde{c}', c') \leq \epsilon$.

If, on the other hand, one assumes that the block containing $c'$ is on the "other side" of $c'$ (that is, this block only contains chords connecting points of $\overline{\mathbb{S}}_{a,b}$), then we use the same argument on the face $F_b$ containing the arc $(\wideparen{e^{-2i\pi b/n}, e^{-2i\pi (b+1)/n}})$. Using the same argument as before, the boundary of $F_b$ contains with high probability a chord of $\overset{\leftarrow}{C}_u^{(n)}$ at distance less than $\epsilon$ from $c'$ (otherwise the associated point in $A_{H_n}$ would be an $\epsilon n/2$-node, which with high probability does not happen by Lemma \ref{lem:structure} (ii)).

\bigskip

Finally, in probability, jointly for $u \leq H_n$,
\begin{align*}
d_H\left(\overset{\leftarrow}{C}^{(n)}_u, P^{(n)}_{n-u}\right) \overset{\P}{\rightarrow} 0.
\end{align*}

\begin{figure}
\center
\caption{Left: the red chord $c$ is a large chord of $\overset{\leftarrow}{C}^{(n)}_u$, and the gray chord $\tilde{c}$ is a chord of $C^{(n)}_{n-u}$, which is close to $c$. $B$ denotes the block of $P^{(n)}_{n-u}$ containing $\tilde{c}$. Thus, the boundary of $B$ necessarily contains a chord inbetween $c$ and $\tilde{c}$. Right: the blue chord $c'$ is a large chord of $P^{(n)}_{n-u}$, and the red chords are the chords of $\overset{\leftarrow}{C}^{(n)}_u$ that are part of the boundary of $F_a$. Among these chords, with high probability, one of them (denoted by $\tilde{c}'$ here) is not far from $c'$.}
\label{fig:partitionblock}
\begin{tabular}{c c c}
\begin{tikzpicture}[scale=3]

\draw[fill=black] (1,0) circle (.01);

\draw (0,0) circle (1);

\draw[dashed,blue,pattern=crosshatch dots, pattern color=blue!50] ({cos(12)},{sin(12)}) -- ({cos(155)},{sin(155)}) -- ({cos(165)},{sin(165)}) --({cos(180)},{sin(180)}) -- ({cos(243)},{sin(243)}) -- ({cos(323)},{sin(323)}) -- ({cos(330)},{sin(330)}) -- cycle;

\draw[red] ({cos(20)},{sin(20)}) -- ({cos(150)},{sin(150)});

\draw[dashed,blue,pattern=crosshatch dots, pattern color=blue!50] ({cos(35)},{sin(35)}) -- ({cos(140)},{sin(140)}) -- ({cos(120)},{sin(120)}) -- cycle;

\draw[dashed,blue,pattern=crosshatch dots, pattern color=blue!50] ({cos(40)},{sin(40)}) -- ({cos(110)},{sin(110)}) -- ({cos(100)},{sin(100)}) -- ({cos(85)},{sin(85)}) -- ({cos(80)},{sin(80)}) -- cycle;

\draw[gray] ({cos(12)},{sin(12)}) -- ({cos(165)},{sin(165)});

\draw (1.1,0) node{$1$};

\draw[blue,fill=white] (-.2,0) node[circle,fill=white]{$B$};

\draw[gray] (-1.05,.3) node{$\tilde{c}$};

\draw[red] (0,.5) node{$c$};
\end{tikzpicture}
&
\begin{tikzpicture}[scale=3]
\draw[white] (0,0) -- (.2,0);
\end{tikzpicture}
&
\begin{tikzpicture}[scale=3]

\draw[fill=black] (1,0) circle (.01);

\draw (1.1,0) node{$1$};

\draw (0,0) circle (1);

\draw[dashed,blue,pattern=crosshatch dots, pattern color=blue!50] ({cos(12)},{sin(12)}) -- ({cos(155)},{sin(155)}) -- ({cos(165)},{sin(165)}) --({cos(180)},{sin(180)}) -- ({cos(243)},{sin(243)}) -- ({cos(323)},{sin(323)}) -- ({cos(330)},{sin(330)}) -- cycle;

\draw[blue] ({cos(12)},{sin(12)}) -- ({cos(155)},{sin(155)});

\draw[dashed,blue,pattern=crosshatch dots, pattern color=blue!50] ({cos(35)},{sin(35)}) -- ({cos(140)},{sin(140)}) -- ({cos(120)},{sin(120)}) -- cycle;

\draw[dashed,blue,pattern=crosshatch dots, pattern color=blue!50] ({cos(40)},{sin(40)}) -- ({cos(110)},{sin(110)}) -- ({cos(100)},{sin(100)}) -- ({cos(85)},{sin(85)}) -- ({cos(80)},{sin(80)}) -- cycle;

\draw[blue] (.6,.35) node{$c'$};

\draw[red] ({cos(12)},{sin(12)}) -- ({cos(26)},{sin(26)}) -- ({cos(140)},{sin(140)}) -- ({cos(152)},{sin(152)});

\draw[red,fill=red] ({cos(152)},{sin(152)}) circle (.01);

\draw[blue,fill=blue] ({cos(155)},{sin(155)}) circle (.01);

\draw[red] (0,.5) node{$\tilde{c}'$};

\draw[red] (-.5,.5) node{$F_a$};
\draw[blue] ({1.2*cos(155)},{1.2*sin(155)}) node{$e^{-2i\pi a/n}$};
\draw[blue] ({1.25*cos(12)},{1.2*sin(12)}) node{$e^{-2i\pi b/n}$};
\end{tikzpicture}
\end{tabular}
\end{figure}
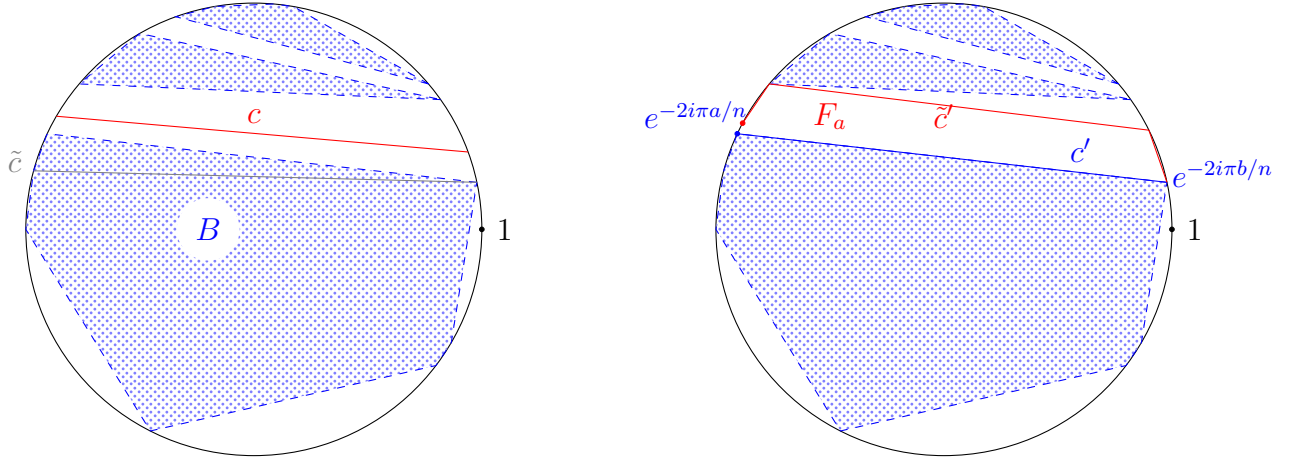

\end{proof}

\section{Computation of the distribution of \texorpdfstring{$\bL^{(\alpha)}_{c}$}{Lac} at \texorpdfstring{$c$}{c} fixed}
\label{sec:distribution}

In this section, we fix $c \in \R_+$. Recall that the Lévy process $\tau^{(\alpha),c}$ is defined as
\begin{align*}
\tau^{(\alpha),c}_s := \inf \left\{ t > 0, Y^{(\alpha)}_t - c^{1/\alpha} t < - c^{1+1/\alpha} s \right\} - c s
\end{align*}
where $Y^{(\alpha)}$ is the $\alpha$-stable Lévy process. Our goal is to prove Theorem \ref{thm:levyprocess}, which states that $\bL_c^{(\alpha)}$ is the lamination coded (in the sense of Section \ref{ssec:fraglam}) by the excursion of $\tau^{(\alpha),c}$.

To this end, we notably introduce a sequence of random trees whose associated sequence of laminations converges towards $\bL^{(\alpha)}_{c}$ and $\bL( \tau^{(\alpha),c,exc})$ at the same time. 

\subsection*{Notations of Section \ref{sec:distribution}}

\renewcommand{\arraystretch}{1.6}
\begin{center}
\begin{tabular}{|c|c|}
\hline
$F_\nu$ & generating function of a law $\nu$\\
\hline
$p_n$ & $c/B_n$\\
\hline
$\mu$ & critical distribution in the domain of attraction of an $\alpha$-stable law\\
\hline
$\mu_n$ & law such that $F_{\mu_n}(x) = F_\mu(p_n x + (1-p_n) F_{\mu_n}(x))$ \\
\hline
$\cT^{(n)}$ & $\mu_n$-GW tree\\
\hline
$W(T)$ & Lukasiewicz path of a tree $T$\\
\hline
$\bL_{Luka}(T)$ & lamination coded by $W(T)$\\
\hline
$S^{(n)}$ & random walk with i.i.d. jumps of law $\mu_n(\cdot + 1)$\\
\hline
\end{tabular}
\end{center}

Here and in the next section, we define the functions $z \rightarrow \log z$ and $z \rightarrow z^a$ (for $a \in \R$) on $\C \backslash \R_-$ the following way:

\begin{defi}
\label{def:log}
Let $z \in \C \backslash \R_-$. Then there exists a unique couple $\rho,\theta \in \R_+^* \times (-\pi,\pi)$ such that $z=\rho e^{i\theta}$. Then we define 
\begin{align*}
\log z \coloneqq \log \rho + i \theta \qquad \text{ and } \qquad z^a \coloneqq e^{a \log z},
\end{align*}
for any $a \in \R$. 
\end{defi}

\subsection{Definition and study of the process \texorpdfstring{$\tau^{(\alpha),c}$}{tac}}

This part is devoted to the study of the process $\tau^{(\alpha),c}$. We start by defining the excursion $\tau^{(\alpha),c,exc}$, and therefore the lamination $\bL(\tau^{(\alpha),c,exc})$. Let us explain some notations. To a Lévy process $X$, we can associate its Laplace exponent $\phi: \R_+^* \rightarrow \R \cup \{ + \infty, - \infty \}$ verifying $\E\left[ e^{-\lambda X_s} \right] := \exp \left( - s \phi(\lambda) \right)$, and its characteristic exponent $\psi: \R \rightarrow \C$ such that $\E\left[ e^{i t X_s} \right] := \exp \left( - s \psi(t) \right)$. A Lévy process $X$ is said to be \textit{spectrally positive} if it makes only positive jumps, i.e. almost surely $\forall s \in \R_+, X_{s-} \leq X_s$. The following theorem, which can be found notably in \cite{CUB11} (see \cite{Kal81} for the original result), gives sufficient conditions for a Lévy process to admit a density:

\begin{thm}
\label{thm:cub}
Let $X$ be a spectrally positive Lévy process and $\psi$ its characteristic exponent. Then, if $t \rightarrow \exp \left( -s \psi \right)$ is integrable for any $s>0$, then $X_s$ admits a density for each $s>0$.
\end{thm}
We refer to \cite{CUB11} for more details. From a Lévy process $X$ verifying the assumption of Theorem \ref{thm:cub}, following \cite{CUB11}, we can construct the so-called \textit{Lévy bridge} $X^{br}$ and \textit{Lévy excursion} $X^{exc}$. From an informal point of view, the Lévy bridge $X^{br}$ has the law of $(X_s)_{s \in [0,1]}$ conditioned to go back to $0$ at $s=1$, while the Lévy excursion $X^{exc}$ has the law of $X^{br}$ conditioned to stay nonnegative between $0$ and $1$. More formally, the Lévy bridge $\left(X_s^{br}\right)_{0 \leq s \leq 1}$ is a random càdlàg process such that, for any $u \in (0,1)$, any bounded continuous function $F: \mathbb{D}([0,u], \R) \rightarrow \R$,

\begin{equation}
\label{eq:bridge}
\E \left[ F \left(\left( X^{br}_s\right)_{0 \leq s \leq u} \right)\right] = \E \left[ F \left(\left( X_s \right)_{0 \leq s \leq u}\right) \frac{q_{1-u}(-X_u)}{q_1(0)} \right]
\end{equation} 
where, for $t>0$, $q_t$ is the density of $X_t$.
In order to define $X^{exc}$, following Miermont \cite[Definition $1$]{Mie01}, we introduce the Vervaat transform of a càdlàg process $f$ going back to $0$ at time $1$, under the additional assumption that $f(1-)=0$.

\begin{defi}
\label{def:vervaat}
Let $f \in \mathbb{D}([0,1], \R)$ such that $f(0)=f(1)=f(1-)=0$. Let $t_{min}$ be the location of the right-most minimum of $f$ (that is, the largest $x$ such that $\min(f(x-),f(x))=inf f$). We define the Vervaat transform of $f$, denoted by $\tilde{f}$, as
\begin{align*}
\tilde{f}(t) = f \left(t+t_{min} \pmod{1}\right) - \underset{[0,1]}{\inf} f
\end{align*}
for $t \in [0,1)$, and $\tilde{f}(1) = \underset{t \rightarrow 1-}{\lim} \tilde{f}(t)$.
\end{defi}

Note that, by time-reversal, for any Lévy process $X$ verifying the assumption of Theorem \ref{thm:cub}, $X^{br}_{1-} = 0$. Thus, we can define $X^{exc} := \tilde{X}^{br}$ (see Fig.~\ref{fig:levy} for an example). In particular, $X^{exc}$ is always nonnegative on $[0,1]$ and, if $X$ is spectrally positive, $X^{exc}$ is an excursion-type function.

\begin{figure}
\center
\caption{Approximation of a bridge obtained from a $1.5$-stable Lévy process, and its Vervaat transform}
\label{fig:levy}
\includegraphics[scale=.9]{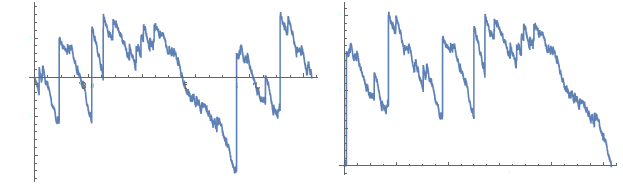}
\end{figure}

Since $\inf \{ t > 0, Y^{(\alpha)}_t \leq 0 \} = 0$ almost surely, we get that $\tau^{(\alpha),c}_0 = 0$ almost surely. Moreover, $\tau^{(\alpha),c}$ is clearly càdlàg and Markov with stationary and independent increments, as $Y^{(\alpha)}$ has these properties, and therefore $\tau^{(\alpha),c}$ is a Lévy process. The following proposition computes its Laplace exponent and its characteristic exponent.

\begin{prop}
\label{prop:exponent}
Fix $\alpha \in (1,2]$, $c > 0$. Then 
\begin{enumerate}
\item[(i)] The Laplace exponent of $\tau^{(\alpha),c}$ has the form $\nu \rightarrow c\overline{\phi}(\nu)-c\nu$ where $\overline{\phi}(\nu)$ is the only real solution of the equation
\begin{equation}
\label{eq:nu}
\overline{\phi}(\nu)^\alpha + c \overline{\phi}(\nu) - c \nu = 0
\end{equation}
\item[(ii)] The characteristic exponent of $\tau^{(\alpha),c}$ has the form $t \rightarrow c \overline{\psi}(t)+itc$, where $\overline{\psi}(t)$ is the only solution with nonnegative real part of the equation
\begin{equation}
\label{eq:psi}
\overline{\psi}(t)^\alpha + c \overline{\psi}(t) + itc = 0.
\end{equation}
\end{enumerate}
\end{prop}

Remark that by Proposition~\ref{prop:exponent} (ii), as $|t| \rightarrow \infty$, $|\overline{\psi}(t)| \rightarrow \infty$, and therefore $\overline{\psi}(t)=o(\overline{\psi}(t)^\alpha)$. Hence, $\overline{\psi}(t)^\alpha \sim -itc$ as $|t| \rightarrow \infty$, and in particular 

\begin{equation}
\label{eq:realpart}
\Re(\overline{\psi}(t)) \underset{|t| \rightarrow \infty}{\sim} |tc|^{\frac{1}{\alpha}} \cos \left( \frac{\pi}{2\alpha} \right).
\end{equation}
Thus, $\tau^{(\alpha),c}$ verifies the assumption of Theorem \ref{thm:cub}, and therefore admits a density. In addition, one can easily check that $\tau^{(\alpha),c}$ is spectrally positive. This allows us to define the excursion $\tau^{(\alpha),c,exc}$ and the lamination $\bL(\tau^{(\alpha),c,exc})$.

\begin{proof}[Proof of Proposition~\ref{prop:exponent}]
Let us first prove (i). Since $\tau_s^{(\alpha),c}$ is a stopping time according to the canonical filtration associated to $Y^{(\alpha)}$ and is almost surely finite, for any $\lambda \in \R$, by Doob's stopping time theorem,

\begin{align*}
\E \left[ \exp \left(-\lambda Y^{(\alpha)}_{\tau_s^{(\alpha),c}+cs} - \left(\tau_s^{(\alpha),c}+cs \right) \lambda^\alpha \right) \right] = 1.
\end{align*}
Now remark that for $s \geq 0$, $Y_{\tau_s^{(\alpha),c}+cs}^{(\alpha)} = c^{1/\alpha} \tau_s^{(\alpha),c}$. Therefore $\E \left[ e^{-\lambda c^{1/\alpha} \tau_s^{(\alpha),c} - (\tau_s^{(\alpha),c}+cs) \lambda^\alpha} \right] = 1$, which can be rewritten $\E \left[ e^{-( \lambda c^{1/\alpha}+\lambda^\alpha) \tau_s^{(\alpha),c}} \right] = e^{cs\lambda^\alpha}$.

Since $x \rightarrow x^\alpha + c^{1/\alpha} x$ is a bijection from $\R_+$ to itself, we get that for all $\nu \geq 0$, 

\begin{align*}
\E \left[ e^{-\nu \tau_s^{(\alpha),c}} \right] = e^{- s c \overline{\phi}(\nu) + s c \nu}
\end{align*}
where $\overline{\phi}(\nu)$ verifies \eqref{eq:nu}. Finally, it is easy to see that, for all $\nu > 0$, \eqref{eq:nu} has exactly one real solution. 

By analytic continuation, the characteristic exponent of $\tau^{(\alpha),c}$ has the form $c \overline{\psi}(t) + itc$ where $\overline{\psi}(t)$ is solution of \eqref{eq:psi}. Remark that $\overline{\psi}(t)$ has nonnegative real part, as $| \E[ e^{i t \tau_s^{(\alpha),c}}]| \leq \E [ |e^{i t \tau_s^{(\alpha),c}} |] = 1$. The fact that \eqref{eq:psi} has exactly one solution with nonnegative real part is postponed to the end of the section (see Theorem~\ref{thm1}).
\end{proof}

\subsection{A new family of random trees}

The key idea of the proof of Theorem~\ref{thm:levyprocess} is to introduce a new sequence of conditioned random trees $( \cT^{(n)}_{\lfloor cB_n \rfloor} )_{n \in \Z_+}$, such that the sequence $( \bL (\cT^{(n)}_{\lfloor cB_n \rfloor} ) )_{n \in \Z_+}$ converges in distribution towards both $\bL( \tau^{(\alpha),c,exc} )$ and $\bL_c^{(\alpha)}$ (Theorem~\ref{thm:biasedgwtree}). These trees are Galton-Watson trees conditioned by their number of vertices, whose offspring distribution varies with $n$. 
 
Let $\mu$ be a critical distribution in the domain of attraction of an $\alpha$-stable law, and $\left( B_n \right)_{n \in \Z_+}$ a sequence verifying \eqref{eq:Bn}. We recall that $\cT$ denotes a $\mu$-GW tree, and that $\cT_n$ denotes a $\mu$-GW tree conditioned to have $n$ vertices.
For $n \in \Z_+$ large enough so that $c B_n/n \leq 1$, define 
\begin{align*}
p_n := c\frac{B_n}{n}.
\end{align*}
and let $\mu_n$ be the law whose generating function $F_{\mu_n}$ verifies
\begin{equation}
\label{eq:analytic}
\forall x \in [-1,1], F_{\mu_n}(x) = F_\mu \left( p_n x + (1-p_n) F_{\mu_n}(x) \right)
\end{equation}
where $F_\mu$ is the generating function of $\mu$ (that is, for $x \in [-1,1]$, $F_\mu(x)=\sum_{i \in \Z_+} \mu(i) x^i$). Remark, by taking $x=1$ in \eqref{eq:analytic}, that $\mu_n$ is also critical for all $n$. We let $\cT^{(n)}$ be a nonconditioned GW tree with offspring distribution $\mu_n$, and $\cT^{(n)}_s$ be the tree $\cT^{(n)}$ conditioned to have $s$ vertices, for $s \in \Z_+$. Remark that, by \eqref{eq:analytic}, for any $n\geq 1$, any $k \geq 1$, $\P(|\cT^{(n)}|=k)>0$ as soon as $0<p_n<1$. Note also that \eqref{eq:analytic} appears in \cite[Proposition $1$ (i)]{Ber10} (taking in this Proposition $x=1$), where Bertoin studies a similar model of random trees coding rare mutations in a population.

\begin{thm}
\label{thm:biasedgwtree}
The following two convergences hold in distribution, as $n \rightarrow \infty$:
\begin{enumerate}
\item[(i)]
$\bL \left(\cT^{(n)}_{\lfloor c B_n \rfloor} \right) \overset{(d)}{\rightarrow} \bL_c^{(\alpha)}$
\item[(ii)]
$\bL \left( \cT^{(n)}_{\lfloor c B_n \rfloor} \right) \overset{(d)}{\rightarrow} \bL \left( \tau^{(\alpha),c,exc} \right)$
\end{enumerate}
\end{thm} 

We prove the two parts of Theorem~\ref{thm:biasedgwtree} separately.

\subsection{Proof of Theorem~\ref{thm:biasedgwtree} (i)}

In order to prove Theorem~\ref{thm:biasedgwtree} (i), we start by seeing $\cT^{(n)}$ as a reduced version of a $\mu$-GW tree. To this aim, let us define the notion of vertex-marking process on a tree. Let $T$ be a plane tree, and $V(T)$ be the set of its vertices. A \textit{vertex-marking process} on $T$ is a function $\cV: V(T)\rightarrow \{ 0,1 \}$ such that $\cV(\emptyset)=1$. We say that a vertex $x \in V(T)$ is \textit{marked} if $\cV(x)=1$. To a vertex-marking process $\cV$ on a plane tree $T$, we associate the \textit{reduced tree} $T^{\cV}$ defined the following way: 
\begin{itemize}
\item the set of vertices of $T^\cV$ is the set of marked vertices of $T$: $V(T^\cV) \coloneqq \left\{ x \in V(T), \cV(x)=1 \right\}$.
\item we erase all the edges of the initial tree $T$.
\item we put a new edge between two vertices of $T^\cV$ if one is the nearest marked ancestor of the other in $T$.
\end{itemize}
(see an example on Fig.~\ref{fig:reductionfigure}).

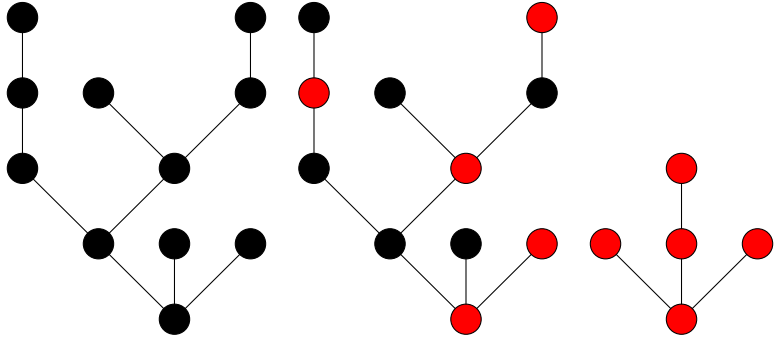
\begin{figure}
\caption{A tree, the same tree with a vertex-marking process (the marked vertices are colored in red) and the associated reduced tree.}
\label{fig:reductionfigure}
\centering
\begin{tabular}{lll}
\begin{tikzpicture}
\draw (0,1) -- (0,0) -- (1,1) ;
\draw (0,0) -- (-1,1) -- (-2,2) -- (-2,3) -- (-2,4) ;
\draw (-1,1) -- (0,2) -- (-1,3) ;
\draw (0,2) -- (1,3) -- (1,4) ;
 \foreach \Point in {(0,0), (1,1), (0,2), (-2,3), (1,4)}
        \draw[fill=black] \Point circle (0.2);
 \foreach \Point in {(0,1), (-1,1), (-2,2), (-1,3), (1,3), (-2,4)}
        \draw[fill=black] \Point circle (0.2);
\end{tikzpicture}
&
\begin{tikzpicture}
\draw (0,1) -- (0,0) -- (1,1) ;
\draw (0,0) -- (-1,1) -- (-2,2) -- (-2,3) -- (-2,4) ;
\draw (-1,1) -- (0,2) -- (-1,3) ;
\draw (0,2) -- (1,3) -- (1,4) ;
 \foreach \Point in {(0,0), (1,1), (0,2), (-2,3), (1,4)}
        \draw[fill=red] \Point circle (0.2);
 \foreach \Point in {(0,1), (-1,1), (-2,2), (-1,3), (1,3), (-2,4)}
        \draw[fill=black] \Point circle (0.2);
\end{tikzpicture}
&
\begin{tikzpicture}
\draw (1,1) -- (0,0) -- (-1,1) ;
\draw (0,0) -- (0,1) -- (0,2) ;
\foreach \Point in {(0,0), (1,1), (-1,1), (0,1), (0,2)}
        \draw[fill=red] \Point circle (0.2);
\end{tikzpicture}
\end{tabular}
\end{figure}

A natural vertex-marking process on a tree $T$ consists in marking the root, and marking each other vertex independently with probability $p_n$. We denote this process by $\cV_{n,c}$. 
Notice that the associated reduced tree is essentially a conditioned version of the tree of alleles of Bertoin \cite{Ber10}, where one forgets about the labels of the vertices. 

The proof is based on the study of the reduced tree, and consists in proving that the lamination associated to this tree is roughly the sublamination of $\bL(\cT_n)$ built by drawing only chords that correspond to marked vertices. For this, we mostly use concentration inequalities on binomial variables. First, remark that the (nonconditioned) GW tree $\cT^{(n)}$ is distributed as $\cT^{\cV_{n,c}}$. Therefore we can focus on the lamination $\bL((\cT^{\cV_{n,c}})_{\lfloor c B_n \rfloor})$, where $\cT$ is a $\mu$-GW tree.

The first technical lemma concerns the size of $\cT$, conditionally to the event that $|\cT^{\cV_{n,c}}|=\lfloor cB_n \rfloor$. Its proof is postponed to the end of the paragraph. Let us introduce a notation: a sequence $(x_n)_{n \in \Z_+}$ being given, we say that $x_n=oe(n)$ if there exists $C>0$, $\epsilon>0$ such that $x_n \leq C e^{-n^\epsilon}$ for all $n$.

\begin{lem}
\label{lem:concentrated}
As $n \rightarrow \infty$,
\begin{align*}
\P \left(\left||\cT|-n\right| \geq n^{1-1/3\alpha} \Big| |\cT^{\cV{n,c}}|=\lfloor cB_n \rfloor\right) =oe(n).
\end{align*}
\end{lem}

\begin{proof}
In order to prove this lemma, observe that for a tree $T$, conditionally to $|~T^{\cV_{n,c}}~|~=~\lfloor cB_n \rfloor$, we have for $A \in \Z_+$,

\begin{align*}
\P \left( |T|=A \Big| |T^{\cV_{n,c}}| = \lfloor cB_n \rfloor \right) &= \frac{\P \left(  |T|=A \right)}{\P \left( |T^{\cV_{n,c}}| = \lfloor cB_n \rfloor \right)} \P \left(|T^{\cV_{n,c}}| = \lfloor cB_n \rfloor \Big| |T|=A \right)\\
&= \frac{\P \left(  |T|=A \right)}{\P \left( |T^{\cV_{n,c}}| = \lfloor cB_n \rfloor \right)} \P \left( Bin \left(A,cB_n/n \right) = \lfloor c B_n \rfloor \right)\\
&\leq \frac{1}{\P \left( |T^{\cV_{n,c}}| = \lfloor cB_n \rfloor \right)} \exp \Big( - A D \left( cB_n/A || cB_n/n \right) \Big)
\end{align*}
by Chernoff inequality, where $D(x||y) = x \log \frac{x}{y} + (1-x) \log \frac{1-x}{1-y}$. Observe that, for any $x,y$, $D(x||y) \geq \frac{(x-y)^2}{2(x+y)}$. This allows us to write:
\begin{align*}
\P \left( |T|=A \Big| |T^{\cV_{n,c}}| = \lfloor cB_n \rfloor \right) &\leq \frac{1}{\P \left( |T^{\cV_{n,c}}| = \lfloor cB_n \rfloor \right)} \exp \Big( - \frac{cB_n}{2n(n+A)} (A-n)^2 \Big).
\end{align*}
On the other hand, remark that 
\begin{align*}
\P \left( |T^{\cV_{n,c}}| = \lfloor cB_n \rfloor \right) &\geq \P \left( |T^{\cV_{n,c}}| = \lfloor cB_n \rfloor \Big| |T|=n  \right) \times \P \left( |T|=n \right)\\
&= \P \left( Bin(n,cB_n/n) = \lfloor cB_n \rfloor \right) \times \P \left( |T|=n \right).
\end{align*}
By the local limit theorem~\ref{llt}, $\P \left( |T|=n \right)$ decays at most polynomially in $n$. At the same time, $\P \left( Bin(n,cB_n/n) = \lfloor cB_n \rfloor \right) \sim \frac{1}{\sqrt{2\pi cB_n}} \left( \frac{cB_n}{\lfloor cB_n \rfloor} \right)^{\lfloor cB_n \rfloor} \left(\frac{n-cB_n}{n-\lfloor cB_n \rfloor} \right)^{n-\lfloor cB_n \rfloor}$ decays at most polynomially as well.

On the other hand, $\exp \Big( - \frac{cB_n}{2n(n+A)} (A-n)^2 \Big) = oe(n) \text{ for } |A-n| \geq n^{1-1/3\alpha}$ and is bounded by $\exp \Big( -  \frac{A}{2 n} \Big)$ for $n$ large enough and $A \geq n^2$. The result follows.
\end{proof}

This lemma allows us to restrict ourselves to the study of a tree $\cT$ with roughly $n$ vertices, exactly $\lfloor cB_n \rfloor$ of which are marked. In what follows, we fix $A>0$ and place ourselves under the two conditions: $||\cT|-n|\leq n^{1-1/3\alpha}$ and $H(\cT) \leq A |\cT|/B_{|\cT|}$. Indeed, by Lemma~\ref{lem:concentrated} and Theorem~\ref{thm:duquesne}, proving the convergence of Theorem~\ref{thm:biasedgwtree} (i) under these conditions in enough to get it in whole generality (again, this convergence has to be understood as: under these conditions, the lamination admits a limit, which converges to $\bL_c^{(\alpha)}$ as $A \rightarrow \infty$).
We denote by $Z_n$ the set of trees verifying these two conditions.

For a given finite tree $T$, we denote by $\overline{V}(T)$ the set of marked vertices of $T$. In what follows, for $\epsilon<1$, $T^{(\epsilon)}$ denotes the set of vertices $x$ of $T$ such that $|\theta_x(T)| > \epsilon |T|$ (where we recall that $\theta_x(T)$ is the subtree of $T$ rooted in $x$). Remark that $T^{(\epsilon)}$ is always nonempty since it contains at least the root. We now define three events on a finite tree $T$ with $\lfloor cB_n \rfloor$ marked vertices (including the root). 

$E(T)$: there exists $x \in T$ such that $|\theta_x(T)|\leq \epsilon |T|$ and that the number of marked vertices in $\theta_x(T)$ is $\geq 2 \epsilon c B_n$. In other words, this is the event that there exists a small subtree which contains a large number of marked vertices.

$F(T,k)$: $|\overline{V}(T) \cap T^{(\epsilon)}| = k$. The number of marked vertices whose subtree contains more than $\epsilon |T|$ vertices is equal to $k$.

Notice that, under the event $F(T,k)$, one can separate $T \backslash (\overline{V}(T) \cap T^{(\epsilon)})$ into $2k-1$ components the following way: taking the first and last times that each element of $\overline{V}(T) \cap T^{(\epsilon)}$ is visited by the contour function of $T$, we get $2k$ times between $0$ and $2n$, which we order increasingly. The components correspond to the vertices visited for the first time by the contour exploration between two consecutive of these times. Since the root is in $\overline{V}(T) \cap T^{(\epsilon)}$, $0$ and $2n$ belong to this set of times and these components form a partition of $T \backslash (\overline{V}(T) \cap T^{(\epsilon)})$. Denote the components by $K_1,\ldots,K_{2k-1}$, and their respective sizes by $s(K_1),\ldots,s(K_{2k-1})$. Finally, denote by $N(K_i)$ the number of marked vertices in $K_i$.

$G(T,k)$: $F(T,k)$ holds and there exists $i \leq 2k-1$ such that $s(K_i) \geq \epsilon n$ and such that, in addition, $\left|N(K_i)- s(K_i) \, cB_{|T|}/|T|\right| \geq B_{|T|}^{3/4}$. We will prove that, for all $k$, with high probability, this even does not occur. In other words, the number of marked vertices in each of these components is very concentrated around its mean.

We get convergences of the probabilities of these three events, uniformly on $Z_n$, as $n \rightarrow \infty$. In the following theorem, the probability has to be understood in the sense that the tree $T$ is fixed, and the marked vertices are random.

\begin{prop}
\label{prop:tripleconvergence}
\begin{itemize}
\item[(i)] $\underset{T \in Z_n}{\sup} \P(E(T)) \underset{n \rightarrow \infty}{\rightarrow} 0$.
\item[(ii)] For any $k \in \Z_+$, $\underset{n \rightarrow \infty}{\lim} \underset{T \in Z_n}{\inf} \P(F(T,k)) = \P(X=k)$ where $X \sim Po(1)$. In particular, these values sum to $1$.
\item[(iii)] For any $k \in \Z_+$, $\underset{T \in Z_n}{\sup} \P(G(T,k)) \underset{n \rightarrow \infty}{\rightarrow} 0$.
\end{itemize}
\end{prop}
Let us immediately see how it implies Theorem~\ref{thm:biasedgwtree} (i)

\begin{proof}[Proof of Theorem~\ref{thm:biasedgwtree} (i) using Proposition~\ref{prop:tripleconvergence}]
We will prove that $\bL((\cT^{\cV_{n,c}})_{\lfloor c B_n \rfloor})$ is close in distribution to $\bL_c(\tilde{C}(\cT)))$ under the assumptions  $||\cT|-n|\leq n^{1-1/3\alpha}$ and $H(\cT) \leq A |\cT|/B_{|\cT|}$, which will straightforwardly imply Theorem~\ref{thm:biasedgwtree} (i) by Theorem \ref{thm:discreteconvergencestable}. To this end, we will prove that large chords have almost the same location in both laminations.

For $u \in \overline{V}(\cT) \cap \cT^{(\epsilon)}$, define $X_u$ a uniform variable on the edge from $u$ to its parent, so that $(X_u)_{u \in \overline{V}(\cT) \cap \cT^{(\epsilon)}}$ are independent. Then, remark that, on one hand, the chords corresponding to $u$ and $X_u$ are at distance at most $2\pi/|\cT|$ in $\bL(\tilde{C}(\cT))$. On the other hand, by Proposition \ref{prop:tripleconvergence} (ii), the set $\{ X_u, u \in \overline{V}(\cT) \cap \cT^{(\epsilon)}\}$ is asymptotically distributed as a Poisson point process $\cP$ of intensity $p_n d\ell$, on the set of edges of $\cT$ whose endpoints are in $\cT^{(\epsilon)}$ (conditionally given that no two points of $\cP$ are in the same edge, which happens with high probability).

Proposition \ref{prop:tripleconvergence} (i) ensures that large chords (namely, chords that have length $\geq 2\pi\epsilon$) in $\bL((\cT^{\cV_{n,c}})_{\lfloor c B_n \rfloor})$ are necessarily coded by points of $\overline{V}(\cT) \cap \cT^{(\epsilon)}$.

Finally, by Proposition \ref{prop:tripleconvergence},(iii), each chord in $\bL((\cT^{\cV_{n,c}})_{\lfloor c B_n \rfloor})$ coded by a vertex $u$ of $\overline{V}(\cT) \cap \cT^{(\epsilon)}$ is asymptotically close to the chord corresponding to $u$ in $\bL(\tilde{C}(\cT))$, which concludes the proof.
\end{proof}

Now we prove Proposition~\ref{prop:tripleconvergence}.

\begin{proof}[Proof of Proposition~\ref{prop:tripleconvergence}]
The proofs of these three statements rely on estimates of binomial tails. Let us start by proving (ii). For $T \in Z_n$, 
\begin{align*}
\P \left( \left|\overline{V}(T) \cap T^{(\epsilon)}\right|=k \right) = \frac{\P(Y_1=k) \, \P(Y_2=\lfloor cB_n \rfloor-k)}{\P(Y=\lfloor cB_n \rfloor)}
\end{align*}
where $Y_1=Bin(|T^{(\epsilon)}|,p_n),Y_2=Bin(|T|-|T^{(\epsilon)}|,p_n)$ and $Y=Bin(|T|,p_n)$. 
We now use the following key fact: for any tree $T$, any $q >0$, let $n_q(T)$ be the number of vertices $x$ of $T$ such that $|\theta_x(T)|>q$. Then 
\begin{equation}
\label{eq:nqt}
n_q(T) \leq |T| \frac{H(T)}{q}.
\end{equation}
Indeed, $h \in \llbracket 0,H(T) \rrbracket$ being fixed, there are at most $|T|/q$ such vertices with height exactly $h$. The result follows by summing over all $h$.
In particular, uniformly in $T \in Z_n$, $|T^{(\epsilon)}| \leq A|T|/B_{|T|} \times 1/\epsilon$. Hence, as $n \rightarrow \infty$, $\P(Y_2=\lfloor cB_n \rfloor-k) \sim \P(Y=\lfloor cB_n \rfloor)$ and
\begin{align*}
\P \left( \left|\overline{V}(T) \cap T^{(\epsilon)}\right|=k \right) \sim \P(Y_1=k) \sim \P \left(X = k \right)
\end{align*}
where $X$ is a Poisson variable of parameter $1$.

In order to prove (i), we use a similar method. Take $T \in Z_n$ and $x$ such that $|\theta_x(T)| \leq \epsilon |T|$. Then the probability that there are $K$ marked vertices in $\theta_x(T)$ is
\begin{align*}
\frac{\P(Y'_1=K) \, \P(Y'_2=\lfloor cB_n \rfloor-K)}{\P(Y'=\lfloor cB_n \rfloor)}
\end{align*}
where $Y'_1=Bin(|\theta_x(T)|,p_n), Y'_2=Bin(|T|-|\theta_x(T)|,p_n)$ and $Y'=Bin(|T|,p_n)$.
By the local limit theorem \ref{llt}, there exists a constant $C_1$ depending only on $\epsilon$ such that, uniformly in $K$, $\P(Y'_2=\lfloor cB_n \rfloor-K) \leq C_1 \P(Y'=\lfloor cB_n \rfloor)$. Hence, the probability $r_x$ that there are more that $2\epsilon cB_n$ vertices in $\theta_x(T)$ satisfies $r_x \leq C_1 \P(Y'_1 \geq 2\epsilon cB_n)$. By Bienaymé-Tchebytchev inequality, for $n$ large enough, 

$$\P(Y'_1 \geq 2\epsilon cB_n) \, \leq \, \P\left(\left|Y'_1-\E(Y'_1)\right| \geq \frac{\epsilon}{2} cB_n\right) \, \leq \, \frac{4 Var(Y'_1)}{\epsilon^2 c^2 B_n^2}.$$
Since $Var(Y'_1) = |\theta_x(T)| p_n (1-p_n) \leq |\theta_x(T)| p_n$, we obtain:
\begin{align*}
r_x \leq 4 C_1 \frac{|\theta_x(T)|}{cB_n \epsilon^2 n}.
\end{align*}

Now observe that, by \eqref{eq:nqt}, the number of vertices $x$ in $T$ (marked or not) such that $|\theta_x(T)|\geq |T|/\log n$ is $\leq H(T) \log n \leq A \log n \, |T|/B_{|T|}$. Hence, with high probability, the number of such vertices that are marked is $O(\log n)$. Distinguishing marked vertices $x$ such that $|\theta_x(T)|\leq |T|/\log n$ and marked vertices such that $\epsilon |T| \geq |\theta_x(T)|\geq |T|/\log n$, we get that there exists a constant $C$ such that
\begin{align*}
\sum_{\substack{x \text{ marked }\\|\theta_x(T)| \leq \epsilon |T|}} r_x \leq  C \left( \log n \frac{\epsilon |T|}{c\epsilon^2 n B_n} + B_n \frac{|T|/\log n}{c \epsilon^2 n B_n} \right) \leq 2 C \left( \frac{\log n}{cB_n \epsilon} + \frac{1}{c\epsilon^2 \log n} \right),
\end{align*}
using the fact that the number of marked vertices in $T$ is exactly  $\lfloor c B_n \rfloor$. This quantity tends to $0$, which provides the result.

Finally, we sketch the idea of the proof of (iii). 
Remark that $N(K_1),\ldots,N(K_{2k-1})$ are distributed as binomials of parameters $(s(K_1),p_{|T|}),\ldots,(s(K_{2k-1}),p_{|T|})$, conditionally to their sum being equal to $\lfloor cB_n \rfloor-k$. The only thing that we need to prove is that, as $n$ grows, for any $T \in Z_n$, for any $M \geq \epsilon n$, for any subset of $M$ points of $T$ independent of the vertex-marking process, the number $N$ of marked vertices among these $M$ points is concentrated enough around its mean. More precisely, since $N$ follows a binomial distribution of parameters $(M,p_n)$, we only need to prove that
\begin{align*}
\P \left( \left|B-\E[B] \right| \geq B_n^{3/4} \right) \underset{n \rightarrow \infty}{\rightarrow} 0.
\end{align*}
where $B \sim Bin(M,p_n)$.
As in the proof of Lemma \ref{lem:concentrated}, this is a direct application of Chernoff inequality. The result follows.
\end{proof}

\subsection{Proof of Theorem~\ref{thm:biasedgwtree} (ii)}

In order to prove this part of the theorem, we need to introduce an other way of coding a finite tree, called the \textit{Lukasiewicz path} of the tree (see Fig. \ref{fig:luka} for an example). Let $T$ be a plane tree with $n$ vertices. Its Lukasiewicz path $\left( W_t(T) \right)_{0 \leq t \leq n}$ is constructed as follows: $W_0(T)=1$ and, for $i \in \llbracket 0,n-1 \rrbracket$, $W_{i+1}(T) - W_i(T) =k_{v_i}(T)-1$. In particular, $W_n(T)=-1$. We define it on the whole interval $[0,n]$ by taking its linear interpolation. 

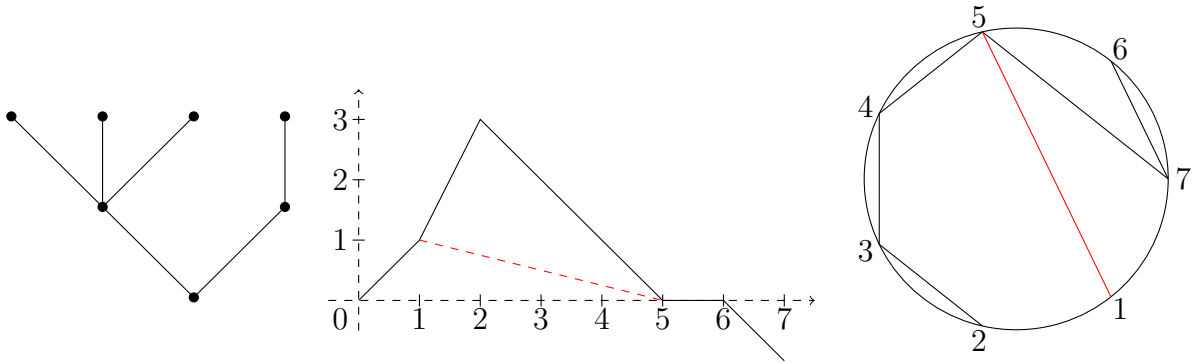
\begin{figure}[h!]
\center
\caption{A tree $T$, its Lukasiewicz path $W(T)$ and the lamination $\bL_{Luka}(T)$. In red, a chord of $\bL_{Luka}(T)$ and the way to draw it from $W(T)$.}
\label{fig:luka}
\begin{tabular}{c c c}
\begin{tikzpicture}[scale=1.2]
\draw[white] (0,0) -- (0,-.7);
\draw (0,0)--(-1,1) -- (-2,2) (-1,2) -- (-1,1) -- (0,2) (0,0)--(1,1) -- (1,2);
\draw[fill] (0,0) circle (.05);
\draw[fill] (-1,1) circle (.05);
\draw[fill] (-2,2) circle (.05);
\draw[fill] (0,2) circle (.05);
\draw[fill] (-1,2) circle (.05);
\draw[fill] (1,1) circle (.05);
\draw[fill] (1,2) circle (.05);
\end{tikzpicture}
&
\begin{tikzpicture}[scale=.8]
\draw[dashed,->] (0,-.5) -- (0,3.5);
\draw[dashed,->] (-.5,0) -- (7.5,0);
\draw (0,0)--(1,1) -- (2,3) -- (5,0) -- (6,0) -- (7,-1);
\draw (-.3,-.3) node{$0$};
\draw (1,-.1) -- (1,.1) (2,-.1) -- (2,.1) (3,-.1) -- (3,.1) (4,-.1) -- (4,.1) (5,-.1) -- (5,.1) (6,-.1) -- (6,.1) (7,-.1) -- (7,.1);
\draw (-.1,1) -- (.1,1) (-.1,2) -- (.1,2)
(-.1,3) -- (.1,3);
\draw (1,-.3) node{$1$};
\draw (2,-.3) node{$2$};
\draw (3,-.3) node{$3$};
\draw (4,-.3) node{$4$};
\draw (5,-.3) node{$5$};
\draw (6,-.3) node{$6$};
\draw (7,-.3) node{$7$};
\draw (-.3,1) node{$1$};
\draw (-.3,2) node{$2$};
\draw (-.3,3) node{$3$};
\draw[red,dashed] (1,1) -- (5,0);
\end{tikzpicture}
&
\begin{tikzpicture}[scale=2]
\draw (0,0) circle (1);
\draw 
({cos(-2*360/7},{sin(-2*360/7}) -- ({cos(-3*360/7},{sin(-3*360/7})
({cos(-3*360/7},{sin(-3*360/7}) -- ({cos(-4*360/7},{sin(-4*360/7})
({cos(-4*360/7},{sin(-4*360/7}) -- ({cos(-5*360/7},{sin(-5*360/7})
({cos(-0*360/7},{sin(-0*360/7}) -- ({cos(-5*360/7},{sin(-5*360/7})
({cos(-0*360/7},{sin(-0*360/7}) -- ({cos(-6*360/7},{sin(-6*360/7})
;
\draw[red] ({cos(-1*360/7},{sin(-1*360/7}) -- ({cos(-5*360/7},{sin(-5*360/7});
\draw ({1.1*cos(-0*360/7},{1.1*sin(-0*360/7}) node{$7$};
\draw ({1.1*cos(-1*360/7},{1.1*sin(-1*360/7}) node{$1$};
\draw ({1.1*cos(-2*360/7},{1.1*sin(-2*360/7}) node{$2$};
\draw ({1.1*cos(-3*360/7},{1.1*sin(-3*360/7}) node{$3$};
\draw ({1.1*cos(-4*360/7},{1.1*sin(-4*360/7}) node{$4$};
\draw ({1.1*cos(-5*360/7},{1.1*sin(-5*360/7}) node{$5$};
\draw ({1.1*cos(-6*360/7},{1.1*sin(-6*360/7}) node{$6$};
\end{tikzpicture}
\end{tabular}
\end{figure}

Recall that we defined in Section \ref{ssec:fraglam} a lamination $\bL(C(T))$  associated to a tree $T$ through its contour function. Here, we shall need another lamination, which is discrete, defined through its Lukasiewicz path. Specifically, fix a plane tree $T$ with $n$ vertices. For every  $0 \leq a \leq n-1$, set $d(a)=\min \{b \in \{a+1, a+2, \ldots,n\}: W_{b}(T)<W_{a}(T) \} $, and set
$$\bL_{Luka}(T)= \bigcup_{a=0}^{n-1} \left[e^{-2i \pi a/n},e^{-2i \pi d(a)/n}\right].$$
(see Fig. \ref{fig:luka} for an example).

The following result shows that the laminations $\bL(C(T))$ and $\bL_{Luka}(T)$ are close, provided that $T$ is a large tree with rather small height.

\begin{lem}
\label{lem:lukacontour}
Let $f: \Z_+ \rightarrow \Z_+$ be such that $f(n) = o(n)$. Then
\begin{align*}
\underset{|T|=n, H(T) \leq f(n)}{\sup} d_{H} \left( \bL(C(T)), \bL_{Luka}(T) \right)  \quad \mathop{\longrightarrow}_{n \rightarrow \infty} \quad  0.
\end{align*}
\end{lem}

\begin{proof}
Let $T$ be a tree with $n$ vertices and height $\leq f(n)$. Let $0 \leq r \leq n-1$, and recall that $v_r(T)$ denotes the $(r+1)$-th vertex of $T$ in the lexicographical order. Let $k$ be the size of the subtree rooted at $v_r(T)$. Then, $v_r(T)$ corresponds to a chord between $e^{-2i\pi r/n}$and $e^{-2i\pi (r+k)/n}$ in $\bL_{Luka}(T)$. In $\bL(C(T))$, $v_r(T)$ codes a chord between $e^{-2i\pi (2r-h(v_r(T))) /2n}$ and $e^{-2i\pi (2(r+k-1)-h(v_r(T))) /2n}$, while points in the edge between $v_r(T)$ and its parent code (infinitely many) chords at distance $\leq 2\pi/n$ to this first one. The result follows since, by assumption, uniformly for all $r$, $h(v_r(T)) = o(n)$.
\end{proof}

Recall that $ \cT^{(n)}_{\lfloor cB_n \rfloor} $ is a Galton-Watson tree with offspring distribution $\mu_{n}$, conditioned on having $\lfloor cB_n \rfloor$ vertices. The main tool to establish Theorem~\ref{thm:biasedgwtree} (ii) is the fact that the Lukasiewicz path of $ \cT^{(n)}_{\lfloor cB_n \rfloor} $ is distributed as a conditioned random walk (see \cite[Section 1.2]{LG05}). More precisely, let $S^{(n)}$ be the integer-valued random walk started from $0$ with i.i.d. jumps, whose jump distribution is given by $\P(S^{(n)}_{1}=k)=\mu(k+1)$ for $k \geq -1$. We extend it on $\R_+$ by linear interpolation. Then, $(S^{(n)}_a)_{0 \leq a \leq \lfloor cB_n \rfloor}$ conditioned on the event $ \{S^{(n)}_{\lfloor cB_n \rfloor} = -1 \textrm{ and } S^{(n)}_a \geq  0  \textrm{ for } a \leq \lfloor cB_n \rfloor -1 \}$ is distributed as the Lukasiewicz path of the tree $\cT^{(n)}_{\lfloor cB_n \rfloor}$. In order to obtain a limit theorem for $S^{(n)}$, we rely on the following local limit theorem.

\begin{thm}
\label{thm:biasedllt}
Fix $0 <u \leq 1$. The following convergence holds as $n \rightarrow \infty$:
\begin{align*}
\underset{|j| \leq n^{3/8}}{\sup} \  \underset{k \in \Z}{\sup} \left| B_n \P \left( S^{(n)}_{\lfloor ucB_n + j \rfloor} = k \right) - q_u \left( \frac{k}{B_n}\right)  \right| \underset{n \rightarrow \infty}{\rightarrow} 0
\end{align*}
where $q_u$ is the density of $\tau_u^{(\alpha),c}$. 
\end{thm}

The proof of this result is postponed to Section \ref{ssec:local}; let us first explain how it entails Theorem~\ref{thm:biasedgwtree} (ii).

\begin{proof}[Proof of Theorem~\ref{thm:biasedgwtree} (ii) from Theorem~\ref{thm:biasedllt}] The first step is to show that the convergence
$$ \left( \frac{S^{(n)}_{\lfloor cB_n \rfloor t}}{B_n} \right)_{0 \leq t \leq 1 }  \textrm{ under } \P( \, \cdot \, |   S^{(n)}_{\lfloor cB_n \rfloor} = -1 \textrm{ and } \forall  a \leq \lfloor cB_n \rfloor -1, S^{(n)}_a \geq  0  )  \quad \mathop{\longrightarrow}^{(d)}_{n \rightarrow \infty} \quad (\tau^{(\alpha),c,exc}_{t})_{0 \leq t \leq 1}$$
holds in distribution. To this end, we follow the classical path, which consists in first showing a convergence under a ``bridge'' condition by combining an unconditioned convergence with absolute continuity and time-reversal, and then using the Vervaat transformation. To do it, we start by proving an unconditioned convergence, namely
\begin{equation}
\label{eq:eq1}
\left( \frac{S^{(n)}_{\lfloor cB_n \rfloor t}}{B_n} \right)_{0 \leq t \leq 1 }   \quad \mathop{\longrightarrow}^{(d)}_{n \rightarrow \infty} \quad  \left(\tau^{(\alpha), c}_t\right)_{0 \leq t \leq 1}
\end{equation}
By \cite[Theorem 16.14]{Kal02}, to prove \eqref{eq:eq1}, it is enough to check that the one-dimensional convergence holds for $t=1$, which is an immediate consequence of Theorem~\ref{thm:biasedllt}. 

Next, we prove the "bridge" version of this theorem, first up to time $u \in (0,1)$. Let $F: \mathbb{D}([0,1],\R) \rightarrow \R$ be a continuous bounded function. Then, setting $\phi_{k}(i)=\P(S^{(n)}_{k}=i)$, by absolute continuity, we have
\begin{align*}
\E \left[ F \left(\left( \frac{S^{(n)}_{\lfloor cB_n \rfloor t}}{B_n} \right)_{0 \leq t \leq u} \right) \Big| S^{(n)}_{\lfloor cB_n \rfloor} = -1 \right] = \E \left[F\left(\left( \frac{S^{(n)}_{\lfloor cB_n \rfloor t}}{B_n} \right)_{0 \leq t \leq u} \right) \frac{\phi_{\lfloor cB_n \rfloor (1-u)}(-S^{(n)}_{\lfloor cB_n \rfloor u}-1)}{\phi_{ \lfloor cB_n \rfloor}(-1)}\right]
\end{align*}

By combining Theorem~\ref{thm:biasedllt} and \eqref{eq:eq1}, this quantity converges to $\E[F((\tau^{(\alpha),c}_t )_{0 \leq t \leq u}) \frac{q_{1-u}(-\tau^{(\alpha),c}_u)}{q_1(0)}]$ as $n \rightarrow \infty$. By \eqref{eq:bridge}, this is equal to $\E[F((\tau^{(\alpha),c,br}_t )_{0 \leq t \leq u})]$. In order to obtain the convergence up to time $1$, it is enough to show tightness on $[0,1]$. Observe that we already know that, conditionally given $S^{(n)}_{\lfloor cB_n \rfloor} = -1$, the sequence $(\frac{S^{(n)}_{\lfloor cB_n \rfloor t}}{B_n})_{0 \leq t \leq 1}$ is tight on $[0,u]$. In order to prove that it is tight on $[0,1]$, we prove that, for $u \in [0,1]$, the process $( \frac{S^{(n)}_{\lfloor c B_n \rfloor - \lfloor c B_n \rfloor t}}{B_n} )_{0 \leq t \leq u}$ is tight on $[0,u]$. For this, just remark that by time-reversal, 
\begin{align*}
\left( \frac{S^{(n)}_{\lfloor cB_n \rfloor}-S^{(n)}_{\lfloor cB_n \rfloor - \lfloor cB_n \rfloor t }}{B_n} \right)_{0 \leq t \leq u} \overset{(d)}{=} \left( \frac{S^{(n)}_{\lfloor cB_n \rfloor t}}{B_n} \right)_{0 \leq t \leq u}
\end{align*}
which is tight conditionally given $S^{(n)}_{\lfloor cB_n \rfloor} = -1$ by the previous observation. 

In order to deduce the convergence of the excursions from the convergence of the bridge versions of the processes, we make use of the Vervaat transform, following Definition~\ref{def:vervaat}. Note that the minimum of $\tau^{(\alpha),c,br}$ is almost surely unique. Indeed, it is true for the unconditioned version $\tau^{(\alpha),c}$ and transfers to the bridge by the absolute continuity relation \eqref{eq:bridge}. Therefore, the Verwaat transform is continuous at $\tau^{(\alpha),c,br}$, and by applying it to the bridge convergence this completes the first step.

To prove that the convergence of the rescaled Lukasiewicz paths of $ \cT^{(n)}_{\lfloor c B_n \rfloor} $ to $\tau^{(\alpha),c,exc}$ implies the convergence of $\bL ( \cT^{(n)}_{\lfloor c B_n \rfloor} )$ to  $\bL ( \tau^{(\alpha),c,exc} )$, first note that a straightforward adaptation of \cite[Proposition 3.5]{Kor14} shows that $\bL_{Luka}(\cT^{(n)}_{\lfloor c B_n \rfloor})$ converges in distribution to $\bL ( \tau^{(\alpha),c,exc} )$ as $n \rightarrow \infty$. To conclude the proof, in view of  Lemma~\ref{lem:lukacontour}, it remains to check that $H (\cT^{(n)}_{\lfloor c B_n \rfloor})=o(B_n)$ with high probability. Let us prove that in fact, with high probability, $H (\cT^{(n)}_{\lfloor c B_n \rfloor}) \leq B_n^{3/4}$. To this end, remark that the height of a vertex in $\cT^{\cV_{n,c}}$ is the number of marked vertices in the ancestral line of the corresponding vertex in $\cT$. Now let $x \in \cT$ be a marked vertex and $h(x)$ be its height. Then, copying the proof of Proposition \ref{prop:tripleconvergence} (i), there exists a constant $C>0$ such that, if $||\cT|-n| \leq n^{1-1/3\alpha}$ and $H(\cT) \leq A n/B_n$, we have by Chernoff inequality:
\begin{align*}
\P \left( N_x \geq B_n^{3/4} \right) &\leq C \, \P \left( Bin(h(x),p_n) \geq B_n^{3/4} \right) \leq C \, \P \left(Bin(A n/B_n,p_n) \geq B_n^{3/4} \right)\\ 
&\leq C \, \exp \left( -2 \frac{B_n^{5/2}}{An} \right) \leq C \, \exp \left(-2 A^{-1} n^{1/8} \right)
\end{align*}
for $n$ large enough, where $N_x$ is the number of marked vertices in the ancestral line of $x$ in $\cT$ (the exponents used here are not optimal but are sufficient to get our result). Here we have used the fact that there exists a constant $K$ such that $B_n \geq K \sqrt{n}$ for all $n$ large enough.
Hence, by a union bound over all $\lfloor cB_n \rfloor$ vertices in the tree, with high probability no marked vertex has more than $B_n^{3/4}$ marked vertices in its ancestral line, which concludes the proof.
\end{proof}

\subsection{Proof of the local estimate}
\label{ssec:local}

In this section, we establish Theorem \ref{thm:biasedllt}. For $n \geq 1$ and $t \in \R$, the following quantity will play an important role:
$$ Y_n(t) \coloneqq B_n \left( 1 - F_{\mu_n} \left( e^{\frac{it}{B_n}} \right)\right).$$
The proof relies on the following estimates. Recall that $c \overline{\psi}(t)+itc$ denotes the characteristic exponent of $\tau^{(\alpha),c}$. 
\begin{lem}
\label{lem:conv}
The following assertions are satisfied:
\begin{enumerate}
\item[(i)] The convergence $F_{\mu_n}(e^{{it}/{B_n}}) \rightarrow 1$ holds as $n \rightarrow \infty$, uniformly for $t \in \R$.
\item[(ii)] Let $\cK$ be a compact subset of $\R$ which does not contain $0$. The convergence $Y_n(t) \rightarrow \overline{\psi}(t)$ holds as $n \rightarrow \infty$, uniformly for $t \in \cK$.
\item[(iii)] For $t \notin 2\pi B_n \Z$, set 
$$K_n(t)= \left(\frac{L\left(  B_n/|Y_n(t)|  \right)}{L(B_n)} \right)^{\frac{1}{\alpha}} \text{ and }A_n(t)=\left(-c B_n ( e^{it/B_n}-1) \right)^{\frac{1}{\alpha}}.$$
Then, for every $\eta>0$, there exists $A>0$ such that, for $n$ large enough, for every $t$ such that $|t| \in [A,\pi B_n]$, we have $|Y_{n}(t)| \geq 1$ and $|K_n(t)Y_n(t)-A_n(t)| \leq \eta |A_n(t)|$.
\end{enumerate}
\end{lem}

Let us first explain how Theorem \ref{thm:biasedllt} follows from Lemma \ref{lem:conv}.

\begin{proof}[Proof of Theorem \ref{thm:biasedllt} using Lemma \ref{lem:conv}.]
Fix $u \in (0,1]$. In the whole proof, for convenience, we will write $u c B_n + j$ instead of $\lfloor u c B_n + j \rfloor$. We let $f(t) \coloneqq e^{- u \ c \ \overline{\psi}(t)}$ be the characteristic function of $\tau_u^{(\alpha),c}$ and $q_u(x) \coloneqq \frac{1}{2\pi} \int_{-\infty}^{\infty} e^{-itx} f(t) dt$ be its density (which we recall exists by Theorem \ref{thm:cub}). Fix $\epsilon>0$. The goal is to prove that, for $n$ large enough, uniformly in $x \in \R$ such that $xB_n \in \Z$, uniformly in $|j| \leq n^{3/8}$,

\begin{equation}
\label{eq:llt}
\left| B_n \P \left( S^{(n)}_{ucB_n + j} =xB_n \right) - q_u(x)  \right| \leq \epsilon.
\end{equation}
For $n \in \Z_+$, $t \in \R$, we set $\phi^{(n)}(t) = F_{\mu_n} (e^{it})$. First, by Fourier inversion, we have for all $k \in \Z$:
$\P( S^{(n)}_{ucB_n + j} =k) = \frac{1}{2\pi} \int_{-\pi}^{\pi} e^{-itk} \left( \phi^{(n)}(t) \right)^{ucB_n+j} dt$.
Hence, for $x \in \R$ such that $xB_n$ is an integer, we can write
\begin{align*}
B_n \P \left( S^{(n)}_{ucB_n + j} =xB_n \right) = \frac{1}{2\pi} \int_{-\pi B_n}^{\pi B_n} e^{-itx} \left( \phi^{(n)}\left(\frac{t}{B_n}\right) \right)^{ucB_n+j} dt
\end{align*}
Therefore, for any $A>\epsilon >0$, we can write
\begin{align*}
\left| B_n \P \left( S^{(n)}_{ucB_n + j} =xB_n \right) - q_u(x)  \right| \leq \frac{1}{2\pi} (I_\epsilon + I_1(A) + I_2(A) + I_3(A))
\end{align*}
where
\begin{align*}
I_\epsilon= \int_{-\epsilon}^\epsilon \left| e^{-itx} \left( \left( \phi^{(n)}\left(\frac{t}{B_n}\right) \right)^{ucB_n+j} -f(t) \right) \right| dt,
\end{align*}

\begin{align*}
I_1(A)= \int_{\epsilon \leq |t| \leq A}\left| e^{-itx} \left( \left( \phi^{(n)}\left(\frac{t}{B_n}\right) \right)^{ucB_n+j} -f(t) \right) \right| dt, 
\end{align*}
\begin{align*}
I_2(A)= \int_{A \leq |t| \leq \pi B_n} \left| e^{-itx} \left( \phi^{(n)}\left(\frac{t}{B_n}\right) \right)^{ucB_n+j} \right| dt \text{ and }
I_3(A)= \int_{A \leq |t| < \infty} \left| e^{-itx} f(t) \right| dt.
\end{align*}
We now bound these four quantities, for certain $A$ well chosen.

\emph{Bounding $I_\epsilon$.} Straightforwardly, since $|\phi^{(n)}|$ and $|f|$ are bounded by $1$ on $\R$, $I_\epsilon \leq 4\epsilon$ for all $n \geq 1$.

\emph{Bounding $I_{1}(A)$.} Since, by definition, $\phi^{(n)}(t/B_{n})=1-Y_{n}(t)/B_{n}$,  Lemma \ref{lem:conv} (ii) entails that, at $\epsilon, A$ fixed, $I_{1}(A) \rightarrow 0$ as $n \rightarrow \infty$, uniformly in $|j| \leq n^{3/8}$.

\emph{Bounding $I_{3}(A)$.} We have already seen that $\Re(\overline{\psi}(t)) \sim |tc|^{\frac{1}{\alpha}} \cos \left( \frac{\pi}{2\alpha} \right)$ as $|t| \rightarrow \infty$. Thus, $|f(t)|$ decays exponentially fast as $|t| \rightarrow + \infty$, and $I_{3}(A) \rightarrow 0$ as $A \rightarrow \infty$ (remark that $I_3(A)$ does not depend on $n$). Hence, for $A$ large enough, $I_3(A) \leq \epsilon$.

\emph{Bounding $I_{2}(A)$.} The main challenge is in fact to bound $I_{2}(A)$. To this aim, we deeply use Lemma~\ref{lem:conv} (ii) and (iii).  For $t \in \R$, we have
\begin{eqnarray}
\left|\phi^{(n)}\left( \frac{t}{B_n} \right) \right|^2 = \left| 1- \frac{Y_n(t)}{B_{n}} \right|^2 &=&\left( 1-\frac{\Re(Y_n(t))}{B_n}\right)^2 + \left( \frac{\Im(Y_n(t))}{B_n} \right)^2    \notag \\
 & \leq  & 1-2\frac{\Re(Y_n(t))}{B_n}+ 2\left(\frac{|Y_n(t)|}{B_n}\right)^2  \label{eq:phin}
\end{eqnarray}
We keep the notation of  Lemma \ref{lem:conv} (iii) and assume that $A>0$ is large enough, so that for every $n$ large enough and $|t| \in [A,\pi B_n]$ we have $|Y_{n}(t)| \geq 1$ and
\begin{equation}
\label{eq:AK}
\left|K_n(t)Y_n(t)-A_n(t)\right| \leq  \frac{1}{2}|A_n(t)|.
\end{equation}
Note that $K_n(t) \in \R_+^*$ for all $t$, that for $t \in [A,\pi B_n]$, $arg(A_n(t)) = \frac{t-\pi B_n}{2\alpha B_n} \in [-\frac{\pi}{2\alpha},0]$ and that for $t \in [-\pi B_n, -A]$, $arg(A_n(t)) = \frac{t+\pi B_n}{2\alpha B_n} \in [0,\frac{\pi}{2\alpha}]$. Therefore, by \eqref{eq:AK},  uniformly for $|t| \in [A,\pi B_n]$, $arg(Y_n(t))$ is bounded away from $\pi/2+ \pi \Z$, and therefore $\Re(Y_n(t)) \geq C |Y_n(t)|$ for some constant $C>0$. Recall indeed that $\Re(Y_n(t)) \geq 0$ for all $t \in \R$.
Then, by \eqref{eq:phin},
\begin{align*}
\left|\phi^{(n)}\left( \frac{t}{B_n} \right) \right|^2 \leq 1-2C\frac{|Y_n(t)|}{B_n} + 2 \left( \frac{|Y_n(t)|}{B_n}\right)^2.
\end{align*}
On the other hand, uniformly for $t \in \R$, ${|Y_n(t)|}/{B_n} \rightarrow 0$ by Lemma~\ref{lem:conv} (ii). Hence, for $n$ large enough and $|t| \in [A,\pi B_n]$,
\begin{align*}
\left|\phi^{(n)} \left( \frac{t}{B_n} \right) \right|^2 \leq 1-C\frac{|Y_n(t)|}{B_n} \leq 1-C\frac{\sqrt{K_n(t)|Y_n(t)|}}{B_n},
\end{align*}
where we have used the Potter bounds \ref{thm:potterbounds} and the fact that $|Y_n(t)|\geq 1$. Hence, \eqref{eq:AK}  gives:
\begin{align*}
\left|\phi^{(n)} \left( \frac{t}{B_n} \right) \right|^2 \leq 1-\frac{C}{\sqrt{2}} \frac{\sqrt{|A_n(t)|}}{B_n} \leq 1-\frac{C}{\sqrt{2}B_n}\left(2 c B_n \sin \left( \frac{|t|}{2B_n}\right) \right)^{1/2\alpha}
\end{align*}
which is less than $1-\frac{C'}{B_n} |t|^{1/2\alpha}$ for some absolute constant $C'>0$, using the fact that $\sin x \geq \frac{2}{\pi} x$ for $x \in [0,\frac{\pi}{2}]$. We finally get for $A$ large enough, for every $n$ large enough and $|t| \in [A,\pi B_n]$%:
\begin{align}
\label{eq:I2}
I_2(A) &= \int_{A \leq |t| \leq \pi B_n} \left| \left( \phi^{(n)}\left(\frac{t}{B_n}\right) \right)^{ucB_n+j} \right| dt \leq \int_{A \leq |t| \leq \pi B_n} \left(1-\frac{C'}{B_n} |t|^{1/2\alpha}\right)^{\frac{ucB_n+j}{2}} dt \nonumber \\
&\leq \int_{A \leq |t| < +\infty} e^{-\frac{C'}{2}(uc+j/B_n)|t|^{1/2\alpha}} dt \leq \int_{A \leq |t| < +\infty} e^{-\frac{ucC'}{4}|t|^{1/2\alpha}} dt.
\end{align}
Thus, for $A>0$ large enoug, for any $n$ large enough and any $|j| \leq n^{3/8}$, $I_{2}(A) \leq \epsilon$. This completes the proof.
\end{proof}

We now prove separately the three  parts of Lemma \ref{lem:conv}.

\begin{proof}[Proof of  Lemma \ref{lem:conv} (i)] It is enough to show that $F_{\mu_n}(0)\rightarrow 1$ as $n \rightarrow \infty$.
Let us denote, for $n \in \Z_+$, $x_n\coloneqq(1-p_n)F_{\mu_n}(0)$. By \eqref{eq:analytic}, $x_n=F_\mu\left(x_n\right)-p_nF_{\mu_n}(0)$. In particular, $F_\mu(x_n)-x_n \rightarrow 0$. Since $f: x \rightarrow F_\mu(x)-x$ is continuous on $[0,1]$ (and hence uniformly continuous), we just have to prove that $1$ is the only fixed point of $F_\mu$. For this, we use the fact that $\mu$ is critical, which implies that $f'(x)= F_\mu'(x)-1$ is negative on $(0,1)$ and $f$ is decreasing. Since $f(1)=0$, $f > 0$ on $[0,1)$ which concludes the proof. 
\end{proof}

The proofs of Lemma \ref{lem:conv} (ii) and (iii) use the following estimate.
\begin{lem}
\label{lem:o1} As $n \rightarrow \infty$, uniformly for $t \in \R \backslash 2\pi B_n \Z$,
$$\frac{L\left( B_n/|Y_n(t)| \right)}{L(B_n)}Y_n(t)^\alpha\left(1+o(1)\right) + c Y_n(t) + c B_n \left( e^{\frac{it}{B_n}}-1 \right) = 0
$$
where the $o(1)$ holds when $n \rightarrow \infty$, uniformly in $t \in \R \backslash 2\pi B_n \Z$, and where $L$ is the slowly varying function defined in \eqref{eq:L}. 
\end{lem}

\begin{proof}
Our main object of interest is the generating function $F_\mu$ of $\mu$. It is notably known (see \cite[$XVII.5$, Theorem $2$]{Fel08}) that $F_\mu$ has the following Taylor expansion at $1-$, on the real axis:

\begin{equation}
\label{eq:Fgenerating}
F_\mu(1-s)-(1-s) \underset{s \downarrow 0}{\rightarrow} \frac{\Gamma(3-\alpha)}{\alpha(\alpha-1)} s^\alpha L\left(\frac{1}{s}\right),
\end{equation} 
where $L$ is the slowly varying function given by \eqref{eq:L}.

Now, observe that, if $t/B_n \neq 0 \mod [2\pi]$, $F_{\mu_n} \left(e^{it/B_n} \right) \neq e^{it/B_n}$. To see this, remark that by \eqref{eq:analytic},
\begin{align*}
F_{\mu_n} \left(e^{it/B_n} \right) = e^{it/B_n} \Rightarrow F_\mu \left(e^{it/B_n} \right) = e^{it/B_n}
\end{align*}
which is possible only if $e^{{it}/{B_n}}=1$ by the case of equality in the triangular inequality (using the fact that $F_\mu(0)>0$). This implies that, if $t/B_n \neq 0 \mod [2\pi]$, $p_n e^{it/B_n} + (1-p_n) F_{\mu_n} \left( e^{it/B_n} \right) < 1$ and we can apply Theorem~\ref{thm:estimate}  to \eqref{eq:analytic}. To simplify notation, set $y_n(t) \coloneqq 1- F_{\mu_n}(e^{it/B_n} ) = Y_n(t)/B_n$. By Lemma \ref{lem:conv} (i), $y_n(t) \rightarrow 0$ uniformly in $t \in \R$, and when $t/B_n \neq 0 \mod [2\pi]$ we can write:
\begin{align}
\label{eq:eeeeh}
F_{\mu_n}\left(e^{it/B_n} \right) &= F_\mu \left( p_n e^{it/B_n} + (1-p_n) F_{\mu_n}\left(e^{it/B_n} \right) \right) \nonumber \\
1-y_n(t) &= F_\mu \left( p_n e^{it/B_n} + (1-p_n)(1-y_n(t)) \right) = F_\mu \left(1+ p_n(e^{it/B_n}-1) - y_n(t) (1-p_n) \right) \nonumber .
\end{align}
Hence, by Theorem~\ref{thm:estimate} and \eqref{eq:Fgenerating},
$$1-y_n(t) = 1 + X_n(t) + \frac{\Gamma(3-\alpha)}{\alpha (\alpha -1)} L \left(  \frac{1}{\left|X_n(t)\right|}\right) \left(-X_n(t)\right)^\alpha (1+o(1)),$$
where we have set  $X_n(t) = p_n(e^{it/B_n}-1) - y_n(t) (1-p_n)$ to simplify notation. Therefore:
\begin{equation}
\label{eq:pnun}
-y_n(t) = X_n(t) + \frac{\Gamma(3-\alpha)}{\alpha (\alpha -1)} L\left(\frac{1}{|X_n(t)|}\right) (-X_n(t))^\alpha (1+o(1)).
\end{equation}
By Lemma \ref{lem:conv} (i), $y_n(t)$ - and therefore $X_n(t)$ - both converge to $0$ uniformly for $t \in \R \backslash 2\pi B_n \Z$. Hence \eqref{eq:pnun} immediately implies that $X_n(t) \sim -y_n(t)$, and thus that $p_n(e^{it/B_n}-1) = o(y_n(t))$. This allows us to reduce \eqref{eq:pnun} to
\begin{align*}
-y_n(t) &= e^{it/B_n}-1 + \frac{1}{p_n} \frac{\Gamma(3-\alpha)}{\alpha (\alpha -1)} L \left(\frac{1}{|y_n(t)|} \right) y_n(t)^\alpha (1+o(1)).
\end{align*}
Remember that by definition $Y_n(t) \coloneqq y_n(t) B_n$. Then 
\begin{align*}
-Y_n(t) = B_n(e^{it/B_n}-1) + \frac{n B_n^{-\alpha}}{c} \frac{\Gamma(3-\alpha)}{\alpha (\alpha -1)} L \left(\frac{B_n}{|Y_n(t)|} \right) Y_n(t)^\alpha (1+o(1))
\end{align*}
which boils down, by \eqref{eq:Bn}, to
\begin{align*}
-c Y_n(t) = cB_n(e^{it/B_n}-1) + \frac{1}{L(B_n)}  L \left(\frac{B_n}{|Y_n(t)|} \right) Y_n(t)^\alpha (1+o(1))
\end{align*}
uniformly in $t \in \R \backslash 2\pi B_n \Z$. This completes the proof.
\end{proof}

\begin{proof}[Proof of  Lemma \ref{lem:conv} (ii)] We show this convergence by analyzing the implicit equation \eqref{eq:analytic}. Let $\cK$ be a compact subset of $\R$ which does not contain $0$. Lemma \ref{lem:o1} implies that, uniformly for $t \in \cK$,
\begin{equation}
\label{eq:uniform}
\left(\frac{L\left(  B_n/|Y_n(t)|  \right)}{L(B_n)}Y_n(t)^\alpha+itc\right)(1+o(1)) + c Y_n(t) = 0.
\end{equation}
Now remark that, by the Potter bounds \ref{thm:potterbounds}, for $n$ large enough,
\begin{eqnarray*}
\min \left(|Y_n(t)|^{(\alpha+1)/2}, |Y_n(t)|^{(3\alpha-1)/2}\right) & \leq& \frac{L\left(  B_n/|Y_n(t)|  \right)}{L(B_n)}|Y_n(t)|^\alpha \\
&& \qquad  \leq \quad  \max \left(|Y_n(t)|^{(\alpha+1)/2}, |Y_n(t)|^{(3\alpha-1)/2} \right).
\end{eqnarray*}
Hence, by \eqref{eq:uniform}, there exists $C>0$ such that, for $n$ large enough and for all $t \in \cK$ (using the fact that $0 \notin \cK$), $C^{-1} \leq |Y_n(t)| \leq C$. This implies that, uniformly for $t \in \cK$,  $\frac{L\left(  B_n/|Y_n(t)|  \right)}{L(B_n)} \rightarrow 1$ as $n \rightarrow \infty$, and that \eqref{eq:uniform} reduces to $(Y_n(t)^\alpha+itc)(1+o(1)) + c Y_n(t) = 0$. Remember that for all $n$, $t$, $\Re Y_n(t) \geq 0$. Therefore $Y_n(t)$ converges to the unique solution of \eqref{eq:psi} with nonnegative real part, which is the characteristic exponent of $\tau^{(\alpha),c}$. 
\end{proof}

\begin{proof}[Proof of  Lemma \ref{lem:conv} (iii)]  From Lemma~\ref{lem:o1}, we get 
\begin{equation}
\label{eq:knyn}
(K_n(t)Y_n(t))^\alpha(1+o(1)) + c Y_n(t) - A_n(t)^\alpha = 0.
\end{equation}
First, for $|t| \in [A, \pi B_n]$, we have $|A_n(t)^\alpha|=2cB_n \sin(\frac{|t|}{2B_n}) \geq \frac{2c}{\pi} A$, which tends to $+\infty$ as $A \rightarrow +\infty$. Second, by the Potter bounds \ref{thm:potterbounds}, $|K_n(t)Y_n(t)|^\alpha \leq |Y_n(t)|^{(\alpha+1)/2} + |Y_n(t)|^{(3\alpha-1)/2}$ for $n$ large enough. These two observations, combined with \eqref{eq:knyn}, readily entail that for fixed $\eta \in (0,1)$, we can find $A>0$ such that  uniformly for $|t| \in [A,\pi B_n]$, $\underset{n \rightarrow \infty}{\liminf} |Y_n(t)| \geq \frac{2}{\eta}$.

Now fix $A>0$ and $n_0 \geq 1$ such that, for all $n \geq n_0$, $|Y_n(t)| \geq \frac{1}{\eta}$. In particular, for $n \geq n_0$, $|Y_n(t)|\geq 1$ and we get from \eqref{eq:knyn}:
\begin{align*}
\left| (K_n(t) Y_n(t))^\alpha -A_n(t)^\alpha \right| &\leq c \left|Y_n(t)\right|^{(\alpha+1)/2} \eta^{(\alpha-1)/2} + o \left( K_n(t) Y_n(t)^\alpha \right)\\
&\leq 2 c \eta^{(\alpha-1)/2} \left|K_n(t) Y_n(t)\right|^\alpha
\end{align*}
for $n$ large enough, using the fact that $|K_n(t)|^\alpha \geq |Y_n(t)|^{\frac{1-\alpha}{2}}$ by the Potter bounds. Therefore, $| 1 - ( \frac{A_n(t)}{K_n(t)Y_n(t)})^\alpha | \leq 2c \eta^{(\alpha-1)/2}$. Now remark that $arg(\frac{A_n(t)}{K_n(t)Y_n(t)})$ is bounded away from $\pi + 2 \pi \Z$, uniformly in $n$. Then $arg(\frac{A_n(t)}{K_n(t)Y_n(t)})$ is necessarily close to $0$, which readily entails that $|1-\frac{K_n(t)Y_n(t)}{A_n(t)}| \leq \eta'$ where $\eta' \rightarrow 0$ as $\eta \rightarrow 0$. This completes the proof.
\end{proof}

\subsection{Study of the solutions of the implicit equation \texorpdfstring{\eqref{eq:psi}}{eqpsi}}

We finish this section by proving that \eqref{eq:psi} has only one solution with nonnegative real part and that this real part is positive for $t>0$; this will imply that this solution is $\overline{\psi}(t)$ by Proposition~\ref{prop:exponent} (ii). Fix $c>0$, and denote by $f: \C \backslash \R_- \times (1,+\infty) \times \R_+^* \rightarrow \C$ the function 
\begin{align*}
f(x, \alpha, t) \coloneqq x^{\alpha} + c x + itc
\end{align*}
Therefore, \eqref{eq:psi} can be rewritten $f(\overline{\psi}(t),\alpha,t)=0$, and we are interested in the solutions in $x$, at $\alpha$ and $t$ fixed, of the equation
\begin{equation}
\label{eq:fzero}
f(x,\alpha,t)=0.
\end{equation}
Note that we also define $f$ for $\alpha > 2$ although we are only interested in the case $\alpha \leq 2$, as this allows to use the implicit function theorem at $\alpha=2$. 

\begin{thm}
\label{thm1}
For any $\alpha \in (1,2]$ and $t>0$, \eqref{eq:fzero} has exactly one solution with nonnegative real part, and this real part is positive.
\end{thm}

\begin{proof}[Proof of Theorem~\ref{thm1}]
We first prove that \eqref{eq:fzero} has a unique such solution for $t$ large enough. Then we use the local continuity in $\alpha$ and $t$ of the solutions of (\ref{eq:fzero}) to extend it to all $t>0$. First, remark that, at $t$ fixed, $f$ is $C^1$ on $\C \backslash \R_- \times (1,+\infty) \times \R_+^*$, and its derivative with respect to $x$ is 
\begin{equation}
\label{eq:ift}
\frac{\partial f}{\partial x}(x,\alpha,t) = \alpha x^{\alpha - 1} + c
\end{equation}  
which is always nonzero when $x$ is a solution of \eqref{eq:fzero}.

In the case $\alpha=2$, \eqref{eq:fzero} has two solutions that are $\frac{-c \pm \sqrt{c^2-4itc}}{2}$. As $t \rightarrow +\infty$, these solutions are equivalent to $\pm \sqrt{tc} e^{-i\pi/4}$. Therefore, we can take $t_0 > 0$ large such that \eqref{eq:fzero} has exactly one solution with positive real part for $\alpha=2$ and $t=t_0$.
Assume that the real part of a solution of \eqref{eq:fzero} is never $0$. Then, by \eqref{eq:ift}, we can use the implicit function theorem around any solution of \eqref{eq:fzero}. This entails that for any $\alpha \in (1,2]$ there exists exactly one solution of $f(x,\alpha, t_0)=0$ that has positive real part. Using again the implicit function theorem at $\alpha$ fixed by letting $t$ vary from $t_0$ to any positive value of $t$, we get Theorem~\ref{thm1}.

Let us finally prove that, indeed, for $t>0$ the real part of a solution of \eqref{eq:fzero} is never $0$. Let $x$ be a solution of \eqref{eq:fzero} and assume that $x=ia$ for some $a\in \R$. Then $0=(ia)^\alpha +iac+itc=a^\alpha e^{i\alpha \pi/2} + c (a+t) e^{-i\pi/2}$ which has no solution.
\end{proof}

\begin{rk}
One can prove that, for $\alpha \in (3/2,2]$ and $t$ large enough, \eqref{eq:fzero} has a second solution which has negative real part. This "negative branch" ultimately vanishes at some $t(\alpha)$, and the corresponding solutions of (\ref{eq:fzero}) converge to the negative real line. The discontinuity of the branch shall therefore be related to the fact that the function $\log$ is not defined on this line.
\end{rk}
\section{Generating functions of stable offspring distributions}

\label{sec:appendix}

This section is devoted to the proof of Theorem \ref{thm:estimate}. We fix a critical offspring distribution $\mu$ (that is, a probability distribution on the nonnegative integers with mean $1$), and we assume that there exists $ \alpha \in (1,2]$ and a slowly varying function $\ell: \R_{+} \rightarrow \R^{*}_{+}$ such that

\begin{equation*}
F_\mu(1-s) - (1-s)  \quad \underset{s \downarrow 0}{\sim} \quad s^\alpha \ell \left( \frac{1}{s} \right),
\end{equation*}
where $F_\mu$ denotes the generating function of $\mu$. This is equivalent to saying that $\mu$ is in the domain of attraction of an $\alpha$-stable law.
We define $L$, the slowly varying function such that
\begin{equation}
\label{eq:Fmu}
\forall x \in \R_+^*, \quad \ell(x) = \frac{\Gamma(3-\alpha)}{\alpha(\alpha-1)} L(x).
\end{equation} 
By e.g. \cite[$XVII.5$, Theorem $2$]{Fel08} and \cite[Lemma $4.7$]{BS15}, if $X$ is a random variable of law $\mu$, then the following statement holds:

\begin{equation}
\label{eq:L2}
\E \left[ X^2 \mathds{1}_{X \leq x} \right] \underset{x \rightarrow +\infty}{\sim} x^{2-\alpha} L(x)+1.
\end{equation}
where $L$ is the function appearing in \eqref{eq:Fmu}. Note that the "$+1$" term is negligible except when $\mu$ has finite variance, in which case $\alpha=2$.

Let us first introduce some notation. For $x \geq 0$, we set $M_x= \mu([x,\infty))$. The main tool in the proof of Theorem \ref{thm:estimate} is the following estimate.

\begin{prop}

\label{prop:sim}

As $|\omega| \rightarrow 0$, with $\Re (\omega) < 0$,

\begin{align*}
\int_{\R_+} \left( 1-e^{\omega x} \right) M_x \mathrm{d}x  \quad \mathop{\sim} \quad  \frac{\Gamma(3-\alpha)}{\alpha (\alpha-1)} (-\omega)^{\alpha-1} \left( L \left( \frac{1}{|\omega|} \right) + \mathds{1}_{\alpha=2} \right)
\end{align*}
where $\mathds{1}_{\alpha=2} = 1$ if $\alpha=2$ and $0$ otherwise.
\end{prop}

Note that there is an extra term "$+1$" when $\alpha=2$. Before proving this result, let us explain how Theorem \ref{thm:estimate} then readily follows.

\begin{proof}[Proof of Theorem \ref{thm:estimate}]

We first show that
\begin{equation}
\label{eq:res}
F_\mu(e^\omega) -1 - \omega  \quad \mathop{\sim}_{\substack{|\omega| \rightarrow 0\\ \Re(\omega)<0}} \quad  \frac{\Gamma(3-\alpha)}{\alpha (\alpha -1)} (-\omega)^\alpha \left( L \left( \frac{1}{|\omega|} \right) + \mathds{1}_{\alpha=2} \right).
\end{equation}
To this end, observe that for $\omega \in \mathbb{C}$ such that $\Re(\omega) < 0$,
\begin{equation}
\label{eq:estim}
F_\mu \left( e^\omega \right) = 1 + \omega - \omega \int_{\R_+} \left( 1-e^{\omega x} \right) M_x \mathrm{d}x.
\end{equation}
Indeed,
\begin{align*}
\int_{\R_+} \left( 1-e^{\omega x} \right) M_x \mathrm{d}x &= \sum_{k \in \Z_+} \mu_k \int_{0}^k  \left( 1-e^{\omega x} \right) \mathrm{d}x = \sum_{k \in \Z_+} k \mu_k - \frac{1}{\omega} \sum_{k \in \Z_+} \mu_k \left( e^{\omega k} - 1 \right)
\end{align*}
which is equal to $1 + \frac{1}{\omega} - \frac{1}{\omega} F_\mu \left( e^{\omega} \right)$. The estimate \eqref{eq:res} then follows from Proposition \ref{prop:sim}.

Now, remark that, for $\omega \in \C$ such that $0<|1+\omega|<1$, $\Re \log (1+\omega) = \log |1+\omega| < 0$, where $\log$ is defined as in Definition \ref{def:log}. Hence, we can apply \eqref{eq:res}  to $\log (1+\omega)$. Then, as $|\omega| \rightarrow 0$ while $0 < |1+\omega|<1$,  by expanding $x \rightarrow \log (1+x)$ around $0$ and using the fact that a slowly varying function varies more slowly than any polynomial, we get that $F_\mu(1+\omega) $ is equal to

\begin{align*}
 F_\mu \left(e^{\log(1+\omega)} \right) &= 1 + \log(1+\omega) + \frac{\Gamma(3-\alpha)}{\alpha (\alpha -1)} \left( - \log (1+\omega) \right)^\alpha  \left( L \left( \frac{1}{|\log(1+\omega)|} \right)+ \mathds{1}_{\alpha=2} \right) \left(1+o(1) \right)\\
&= 1+\omega + \frac{\Gamma(3-\alpha)}{\alpha (\alpha -1)} (-\omega)^\alpha L \left( \frac{1}{|\omega|} \right) \left(1+o(1) \right),
\end{align*}
and this completes the proof.
\end{proof}

The statement of Proposition \ref{prop:sim} is slightly different whether $\alpha=2$ or $\alpha<2$, and therefore we need two different proofs. The reason comes from the following useful estimate (see \cite[Corollary $XVII.5.2$ and $(5.16)$]{Fel08}): 

\begin{equation}
\label{carbink}
 M_x   \quad \mathop{\sim}_{x \rightarrow \infty} \quad 
 \begin{cases}\displaystyle 

 \quad \frac{2-\alpha}{\alpha} x^{-\alpha} L(x)  & \text{ when } \alpha\in (1,2) \\

\displaystyle \quad  x^{-2} L'(x) &  \text{ when } \alpha=2

  \end{cases}
\end{equation}
where $L'$ is a slowly varying function such that ${L'(x)}/{L(x)} \underset{x \rightarrow \infty}{\rightarrow} 0$.

\subsection{Proof of Proposition \ref{prop:sim} for \texorpdfstring{$\alpha=2$}{a2}}

We start with the case $\alpha=2$, which is easier. In what follows, we set $C>0$ such that, for all $N \in \Z$, $N \geq 1$,
\begin{equation}
\label{eq:c}
L(N)+1 \leq C L(N).
\end{equation}
The existence of such a $C$ is guaranteed by \eqref{eq:L2} as soon as $\mu \neq \delta_1$. The proof of Proposition \ref{prop:sim} is based on the following lemma:

\begin{lem}

\label{lem:alpha2} The following assertions hold.

\begin{enumerate}

\item[(i)] As $N \rightarrow \infty$, $ \displaystyle \int_0^N x M_x  \mathrm{d}x \sim \left(L(N)+1 \right)/2$.

\item[(ii)] Fix $\epsilon>0$ and $C$ verifying \eqref{eq:c}. Then, for $ N $ large enough and $\omega \in \C$ such that $C e N |\omega| \leq \epsilon$, we have

\begin{align*}
\left| \int_0^N \left( 1-e^{\omega x} \right) M_x \mathrm{d}x - \int_0^N (-\omega x) M_x  \mathrm{d}x \right| \leq \epsilon |\omega| L(N).
\end{align*}
\end{enumerate}
\end{lem}

\begin{proof}

For the first assertion simply write, for $N \geq 1$,

\begin{align*}
\int_0^N x M_x  \mathrm{d}x = \sum_{k=1}^N \left( k-\frac{1}{2} \right) M_k &= \sum_{k=1}^N \left( k-\frac{1}{2} \right) \sum_{\ell=k}^\infty \mu_\ell= \frac{N^2}{2} M_{N+1} + \sum_{\ell=1}^N \frac{\ell^2}{2} \mu_\ell,
\end{align*}
which is asymptotic to $(L(N)+1)/2$ by \eqref{eq:L2} and \eqref{carbink}.

For (ii), observe that for $x \in \C$ such that $|x|\leq 1$, we have $|e^x-1-x| \leq e |x|^2$. Hence, when $CeN|\omega| \leq \epsilon$, one has:

\begin{align*}
\left| \int_0^N \left( 1-e^{\omega x} \right) M_x \mathrm{d}x - \int_0^N (-\omega x) M_x  \mathrm{d}x \right| &\leq e |\omega|^2 \int_0^N x^2 M_x \mathrm{d}x \leq e N |\omega|^2 \int_0^N x M_x \mathrm{d}x.
\end{align*}
Hence, by (i), for $N$ large enough and $CeN|\omega| \leq \epsilon$, we have $e N |\omega|^2 \int_0^N x M_x \mathrm{d}x \leq C e N |\omega|^2 L(N)$, which is at most $\epsilon |\omega| L(N)$. This completes the proof.

\end{proof}

\begin{proof}[Proof of Proposition \ref{prop:sim} for $\alpha=2$] We assume that $\alpha=2$. Fix $\epsilon>0$. For $\omega \in \C$ with $\Re(\omega)<0$, let $N_\omega \coloneqq \lfloor \frac{\epsilon}{2 C e |\omega|} \rfloor$. Therefore, Lemma~\ref{lem:alpha2} (ii) holds with $N=N_\omega$ for $|\omega|$ small enough and we get

\begin{align*}
& \left| \int_{\R_+} \left( 1-e^{\omega x} \right) M_x \mathrm{d}x - \int_{0}^{N_\omega} (-\omega x) M_x  \mathrm{d}x \right| \\
& \qquad \qquad \leq \left|\int_0^{N_\omega} \left( 1-e^{\omega x} \right) M_x \mathrm{d}x - \int_0^{N_\omega} (-\omega x) M_x  \mathrm{d}x \right| + \left| \int_{N_\omega}^\infty \left( 1-e^{\omega x} \right) M_x \mathrm{d}x \right|\\
& \qquad \qquad \leq  \epsilon |\omega| L({N_\omega}) + 2 \int_{N_\omega}^\infty M_x \mathrm{d}x \, \leq \, \epsilon |\omega| L({N_\omega}) + 3 \frac{L'(N_\omega)}{N_\omega},
\end{align*}
where we have used Lemma~\ref{lem:alpha2} (ii) and the fact that $\int_N^\infty M_x \mathrm{d}x \sim \int_N^\infty \frac{L'(x)}{x^2} \mathrm{d}x \sim \frac{L'(N)}{N}$ as $ N \rightarrow \infty $ 
(see  \cite[Proposition $1.5.10$]{BGT89}). Since  ${L'(x)}/{L(x)} \underset{x \rightarrow \infty}{\rightarrow} 0$, it follows that for $|\omega|$ small enough,

\begin{equation}
\left| \int_{\R_+} \left( 1-e^{\omega x} \right) M_x \mathrm{d}x - \int_{0}^{N_\omega} (-\omega x) M_x  \mathrm{d}x \right| \leq 2 \epsilon |\omega| L(N_\omega).
\end{equation}
But by Lemma~\ref{lem:alpha2} (i), $\int_0^{N_\omega} (-\omega x) M_x \mathrm{d}x \sim -\frac{1}{2} \omega (L(N_\omega)+1)$ as $|\omega| \rightarrow 0$. The desired result is obtained by taking $\epsilon \rightarrow 0$, using the facts that $L(N_\omega) \sim L(\frac{1}{|\omega|})$ as $|\omega| \rightarrow 0$, and that $\frac{\Gamma(3-\alpha)}{\alpha(\alpha-1)}=\frac{1}{2}$ when $\alpha=2$.

\end{proof}

\subsection{Proof of Proposition \ref{prop:sim} for \texorpdfstring{$\alpha \in (1,2)$}{a12}}

We now fix $\alpha \in (1,2)$. In the sequel, we fix $a_{0}>0$ such that for every $z \in \mathbb{C}$:

\begin{equation}
\label{eq:a0} 
|z| \leq a_{0} \implies |1-e^{z}| \leq 2 |z|.
\end{equation}

The proof is based on two technical  estimates.

\begin{lem}
\label{lem:omega}
The following assertions hold:
\begin{enumerate}
\item[(i)]  uniformly for $\omega$ with negative real part,
\begin{equation*}
\underset{\substack{a \rightarrow 0 \\ B \rightarrow \infty}}{\lim}\int_{-a \omega/|\omega|}^{-B \omega/|\omega|} \left( 1-e^{-y} \right)  y ^{-\alpha} \mathrm{d}y = \int_{\R_+} \left( 1-e^{-y} \right) y^{-\alpha} \mathrm{d}y = \frac{\alpha}{2-\alpha} \frac{\Gamma(3-\alpha)}{\alpha (\alpha - 1)};
\end{equation*}

\item[(ii)] for any fixed $\eta \in (0,1)$, we have

$$\int_{|\omega|^{-\eta}}^{+\infty} (1-e^{\omega x}) x^{-\alpha} L(x) \mathrm{d}x \underset{\substack{|\omega| \rightarrow 0 \\ \Re(\omega)<0}}{\sim} (-\omega)^{\alpha-1} L \left(\frac{1}{|\omega|} \right)  \cdot  \frac{\alpha}{2-\alpha} \frac{\Gamma(3-\alpha)}{\alpha (\alpha - 1)}.$$
\end{enumerate}

\end{lem}

\begin{proof}

For the first assertion, we use tools from complex analysis.   For $0<a<B<+\infty$, define the path $\gamma_a^B$ as in Fig. \ref{fig:path}, as the union of two straight lines and two arcs $\gamma_a$ and $\gamma^B$. Since $y \longmapsto \left( 1-e^{-y} \right) y^{-\alpha}$ is holomorphic on $\C \backslash \R_-$, the value of its integral on this path is $0$.

\begin{figure}[h!]

\center

\caption{The path $\gamma_a^B$}

\label{fig:path}

\begin{tikzpicture}[scale=.5]

\draw[->] ({sqrt(50)},0) -- (9,0);

\draw (-1,0) -- ({sqrt(2)},0);

\draw[->] (0,-1) -- (0,7);

\draw[red] ({sqrt(50)},0) -- ({sqrt(2)},0);

\draw[red] (1,1) -- (5,5);

\draw[blue, auto=right] (5,5) arc (45:0:{sqrt(50)}) node[midway]{$\gamma^B$};

\draw[blue] (.8,.5) node{$\gamma_a$};

\draw[blue] (1,1) arc (45:0:{sqrt(2)});

\draw ({sqrt(2)},-.5) node{$a$};

\draw (0.8,2) node{$\frac{-\omega}{|\omega|}a$};

\draw ({sqrt(50)},-.5) node{$B$ };

\draw (5.5,5.7) node{$\frac{-\omega}{|\omega|} B$}; 

\draw[red,fill=red] (1,1) circle (.1);

\draw[red,fill=red] (5,5) circle (.1);

\end{tikzpicture}

\end{figure}
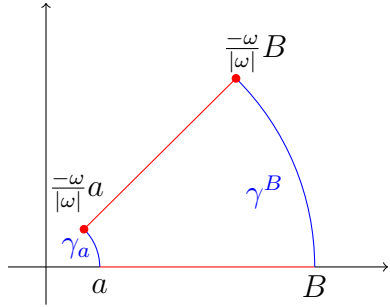

By \eqref{eq:a0}, for $0<a<a_0$, uniformly for $\omega$ with negative real part,

\begin{align*}
\left| \int_{\gamma_a} (1-e^{-y}) y^{-\alpha}\mathrm{d}y \right| \leq 2 \left| \int_{\gamma_a} |z|^{1-\alpha} \mathrm{d}z \right| \leq \pi a^{2-\alpha} \underset{a \rightarrow 0}{\rightarrow} 0
\end{align*}
and
\begin{align*}
\left| \int_{\gamma^B} (1-e^{-y}) y^{-\alpha}\mathrm{d}y \right| \leq 2  \left| \int_{\gamma^B} |z|^{-\alpha} \mathrm{d}z \right| \leq \pi B^{1-\alpha} \underset{B \rightarrow +\infty}{\rightarrow} 0.
\end{align*}
On the other hand, as $a \rightarrow 0$ and $B \rightarrow \infty$, $\int_a^B (1-e^{-y}) y^{-\alpha}\mathrm{d}y  \rightarrow  \int_{\R_+} (1-e^{-y}) y^{-\alpha}\mathrm{d}y$. This shows the first equality in (i). The second one is a simple computation.

For (ii), the idea is to write $$\int_{|\omega|^{-\eta}}^{+\infty} (1-e^{\omega x}) x^{-\alpha} L(x) \mathrm{d}x = \left(  \int_{|\omega|^{-\eta}}^{a / |\omega|} +  \int_{a/|\omega|}^{B / |\omega|}  +\int_{B / |\omega|}^{+\infty} \right) (1-e^{\omega x}) x^{-\alpha} L(x) \mathrm{d}x$$
for some $a<B$ to be fixed later, and to estimate the three terms. Let us fix $\epsilon>0$.

\emph{Third term.} By the Potter bounds, we may fix $B_0>0$ such that, for any $B \geq B_0$, for $|\omega| \leq B^{-1}$ and $x \geq {B}/{|\omega|}$, we have $L(x) \leq L(1/|\omega|) (x|\omega|)^{(\alpha-1)/2}$. This implies that, for $B \geq B_{0}$,

\begin{align*}
\left| \int_{B /|\omega|}^{+\infty} (1-e^{\omega x}) x^{-\alpha} L(x) \mathrm{d}x \right| &\leq 2 L \left( \frac{1}{|\omega|} \right) \int_{B /|\omega|}^{+\infty} |\omega|^{(\alpha-1)/2} x^{-(\alpha+1)/2} \mathrm{d}x\\
&= 2 L \left( \frac{1}{|\omega|} \right) |\omega|^{\alpha-1} \int_{B}^{+\infty} x^{-(\alpha+1)/2} \mathrm{d}x,
\end{align*}
which is less than $\epsilon L ( {1}/{|\omega|} ) |\omega|^{\alpha-1} $ for $B$ large enough (independent of $\omega$). In what follows, we take $B$ such that this holds.

\emph{First term.} By the Potter bounds, there exists $a \in (0,1)$ such that, for $|\omega|$ small enough and $|\omega|^{-\eta} \leq x \leq \frac{a}{|\omega|}$, we have $L(x) \leq L( {1}/{|\omega|} ) (x|\omega|)^{\alpha/2-1}$. Furthermore, by \eqref{eq:a0},  for $a$ small enough, 

\begin{align*}
\left| \int_{|\omega|^{-\eta}}^{a /|\omega|} (1-e^{\omega x}) x^{-\alpha} L(x) \mathrm{d}x \right| &\leq  2 |\omega| \int_{|\omega|^{-\eta}}^{a /|\omega|} x^{1-\alpha} L(x) \mathrm{d}x \leq 2 |\omega|^{\alpha/2} L\left( \frac{1}{|\omega|} \right) \int_{0}^{a /|\omega|} x^{-\alpha/2} \mathrm{d}x\\
&\leq 2 |\omega|^{\alpha-1} L\left( \frac{1}{|\omega|} \right) \int_{0}^{a} y^{-\alpha/2}\mathrm{d}y
\end{align*}
which is less than $\epsilon L ( {1}/{|\omega|} ) |\omega|^{\alpha-1} $ for $a>0$ small enough (independent of $\omega$). In what follows, we take $a>0$ such that this holds.

\emph{Second term.}  Since $L$ is slowly varying, uniformly in $x \in ( {a}/{|\omega|}, {B}/{|\omega|})$, $L(x) \sim L( {1}/{|\omega|} )$ as $|\omega| \rightarrow 0$. Therefore, for any $\epsilon'>0$, for $|\omega|$ small enough (depending on $\epsilon'$),

\begin{align*}
&\left| \int_{a/|\omega|}^{B/|\omega|} (1-e^{\omega x}) x^{-\alpha} L(x) \mathrm{d}x - L \left( \frac{1}{|\omega|} \right) \int_{a/|\omega|}^{B/|\omega|} (1-e^{\omega x}) x^{-\alpha} \mathrm{d}x \right| \\
&\qquad\qquad \qquad \qquad \qquad \leq \epsilon' L \left( \frac{1}{|\omega|} \right) \int_{a/|\omega|}^{B/|\omega|} |1-e^{\omega x}| x^{-\alpha} \mathrm{d}x  \leq 2 \epsilon' |\omega|^{\alpha-1} L \left( \frac{1}{|\omega|} \right) \int_a^B y^{-\alpha}\mathrm{d}y,
\end{align*}
where the last inequality follows from a change of variables. We conclude that for $|\omega|$ small enough (depending on $a$ and $B$),

$$\left| \int_{a/|\omega|}^{B/|\omega|} (1-e^{\omega x}) x^{-\alpha} L(x) - L \left( \frac{1}{|\omega|} \right) \int_{a/|\omega|}^{B/|\omega|} (1-e^{\omega x}) x^{-\alpha} \mathrm{d}x \right| \leq \epsilon |\omega|^{\alpha-1} L \left( \frac{1}{|\omega|} \right).$$

By putting together the three previous estimates, we get

$$\int_{|\omega|^{-\eta}}^{+\infty} (1-e^{\omega x}) x^{-\alpha} L(x) \mathrm{d}x = L \left( \frac{1}{|\omega|} \right) \left( \int_{a/|\omega|}^{B/|\omega|} (1-e^{\omega x}) x^{-\alpha} \mathrm{d}x + o(|\omega|^{\alpha-1}) \right).$$
as $|\omega| \rightarrow 0, \Re(\omega)<0$.
To conclude the proof, remark that by change of variables, $$ (-\omega)^{1-\alpha} \int_{a/|\omega|}^{B/|\omega|} (1-e^{\omega x}) x^{-\alpha} \mathrm{d}x =\int_{-a \omega/|\omega|}^{-B \omega/|\omega|} (1-e^{-y}) y^{-\alpha}\mathrm{d}y,$$ which converges towards $ \int_{\R_+} (1-e^{-y}) y^{-\alpha}\mathrm{d}y = \frac{\alpha}{2-\alpha} \frac{\Gamma(3-\alpha)}{\alpha (\alpha - 1)}$ as $\, a \rightarrow 0 \,$, $B \rightarrow \infty$ by (i).

\end{proof}

\begin{proof}[Proof of Proposition~\ref{prop:sim} in the case $\alpha \in (1,2)$.] Let us assume that $\alpha \in (1,2)$ and recall that the goal is to estimate $\int_{\R_+} (1-e^{\omega x}) M_x \mathrm{d}x$.  The idea is to write

$$\int_{0}^{\infty} \left( 1-e^{\omega x} \right) M_x \mathrm{d}x =  \int_0^{|\omega|^{-\eta}}  (1-e^{\omega x}) M_x \mathrm{d}x  +  \int_{|\omega|^{-\eta}}^{\infty}  (1-e^{\omega x}) M_x \mathrm{d}x $$
for certain well chosen $\eta>0$ and to estimate separately these two terms. Using \eqref{carbink}, we shall show that as $\omega \rightarrow 0$, the first term is $o (|\omega|^{\alpha-1} L( {1}/{|\omega|} ) )$, while the second one is asymptotic to $  \frac{\Gamma(3-\alpha)}{\alpha (\alpha-1)} (-\omega)^{\alpha-1} L ( {1}/{|\omega|} )$. Again, some care is needed as we are dealing with complex-valued quantities.

\emph{First term.} 

First of all, by definition, for any $x \in \R_+$, $M_x \leq 1$. Therefore, setting $\eta = {(2-\alpha)}/{4} \in (0,1/4)$ and using \eqref{eq:a0},

\begin{equation*}
\left| \int_0^{|\omega|^{-\eta}}  (1-e^{\omega x}) M_x \mathrm{d}x \right| \leq \int_0^{|\omega|^{-\eta}} 2 |\omega| x \mathrm{d}x \leq  |\omega|^{1-2\eta} = |\omega|^{\alpha/2}
\end{equation*}
for $|\omega|$ small enough. As a consequence, $| \int_0^{|\omega|^{-\eta}}  (1-e^{\omega x}) M_x \mathrm{d}x | = o (|\omega|^{\alpha-1} L( {1}/{|\omega|} ) )$ as $|\omega| \rightarrow 0$.

\emph{Second term.} 

Fix $\epsilon>0$. By the estimate \eqref{carbink}, as $|\omega| \rightarrow 0$, uniformly for $x \geq |\omega|^{-\eta}$, $M_x~\sim~\frac{2-\alpha}{\alpha}~x^{-\alpha}~L(x)$. This allows us to write for any $\epsilon'>0$, for $|\omega|$ small enough (depending on $\epsilon'$): 

\begin{align*}
\left|\int_{|\omega|^{-\eta}}^{+\infty} (1-e^{\omega x})M_x \mathrm{d}x  - \frac{2-\alpha}{\alpha}\int_{|\omega|^{-\eta}}^{+\infty} (1-e^{\omega x}) x^{-\alpha} L(x) \mathrm{d}x \right| \leq \epsilon' \int_{|\omega|^{-\eta}}^{+\infty} |1-e^{\omega x}| x^{-\alpha} L(x)\mathrm{d}x.
\end{align*}
In particular, mimicking the proof of Lemma \ref{lem:omega} (ii), we bound the right-hand term and get, for $|\omega|$ small enough,

\begin{align*}
\left|\int_{|\omega|^{-\eta}}^{+\infty} (1-e^{\omega x})M_x \mathrm{d}x  - \frac{2-\alpha}{\alpha}\int_{|\omega|^{-\eta}}^{+\infty} (1-e^{\omega x}) x^{-\alpha} L(x) \mathrm{d}x \right| \leq \epsilon |\omega|^{\alpha-1} L \left( \frac{1}{|\omega|}\right).
\end{align*}
The desired result then follows from the estimate of Lemma \ref{lem:omega} (ii).

\end{proof}

\bibliographystyle{abbrv}
\bibliography{biblistablelaminations}

\end{document}